\documentclass[reqno,english,11pt]{amsart}
\usepackage[margin=1in]{geometry}
\usepackage{mathtools}
\mathtoolsset{showonlyrefs}

\usepackage{amsmath, amsthm, amssymb,cite,enumitem}
\usepackage[makeroom]{cancel}

\usepackage{amsmath}
\usepackage{amssymb}
\usepackage{amsrefs}
\usepackage{graphicx}
\usepackage{bm}
\usepackage{enumitem}
\usepackage[dvipsnames]{color}
\usepackage[colorlinks=true, pdfstartview=FitV, linkcolor=black, citecolor=blue, urlcolor=blue]{hyperref}
\usepackage{mathtools}
\usepackage{xcolor}
\newcommand\R{\mathbb{R}}
\newcommand\Z{\mathbb{Z}}

\newcommand{\dd}{d}

\newcommand{\eps}{\varepsilon}
\newcommand{\jxi}{\langle \xi \rangle}

\newcommand{\snorm}[1]{\left\|#1\right\|}

\usepackage{comment}
\usepackage{mathtools}
\mathtoolsset{showonlyrefs}

\newtheorem{theorem}{Theorem}[section]
\newtheorem{lemma}[theorem]{Lemma}
\newtheorem{corollary}[theorem]{Corollary}
\newtheorem{proposition}[theorem]{Proposition}

\theoremstyle{definition}
\newtheorem{definition}[theorem]{Definition}

\numberwithin{equation}{section}

\allowdisplaybreaks

\begin{document}
\title{Scattering for the quintic generalized Benjamin-Bona-Mahony Equation}

\author[G.~Chen]{Gong Chen}
\author[Y.~Zhang]{Yingmo Zhang}
\address{Department of Mathematics, University of Wisconsin--Madison, Madison, WI 53706, USA}
\email{zhang2734@wisc.edu}
\address{School of Mathematics, Georgia Institute of Technology, Atlanta, GA 30332, USA}
\email{gc@math.gatech.edu}

\begin{abstract}
We consider the quintic generalized Benjamin–Bona–Mahony equation
\begin{equation*}
    u_t-u_{xxt}+\partial_x\big(u+u^{5}\big)=0,\qquad (t,x)\in \R_+\times\R.
\end{equation*}
Using the space–time resonance method, we prove that sufficiently small and smooth solutions scatter to the linear flow.

While the higher nonlinearity simplifies the treatment of nonresonant interactions compared to the quartic model in \cite{Morgan}, resonance analysis is more intricate. The resonance analysis occurs in a higher-dimensional geometric setting, and certain null or vanishing conditions present in the quartic case fail at specific resonance points. As a result, refined computations and precise estimates near the resonant set are required to close the bootstrap argument.

\end{abstract}
\maketitle
\tableofcontents

 \section{Introduction}

We study the long-time dynamics of small solutions to the generalized Benjamin--Bona--Mahony (gBBM) equation on the real line
\begin{equation}\label{eq:gBBM} 
u_t-u_{xxt}+\partial_x\big(u+u^{p+1}\big)=0,\qquad (t,x)\in \R_+\times\R,
\end{equation}
with the \emph{quintic} nonlinearity $p=4$.  Equation \eqref{eq:gBBM} is a classical long-wave model.  For more historical background on \eqref{eq:gBBM}, we refer to \cite{Morgan} and references therein for details. In this equation,  the dispersive mechanism is carried by the nonlocal linear operator $(1-\partial_x^2)^{-1}\partial_x$ rather than by a higher-order derivative as in the Korteweg–de Vries equation.  This  difference has substantial consequences for resonance: the linear dispersion relation has nontrivial stationary structure, and the nonlinear interactions can localize near specific frequency configurations. In the quintic case,  understanding whether these interactions remain integrable in time---so that solutions ultimately behave like linear solutions---is the central theme of this paper.

Our goal is to prove that sufficiently small solutions exist globally and \emph{scatter} to the linear BBM flow.  Note that even in this quintic case,  the long time behavior here does not follow from the decay. Schematically, given the formal $L^1\rightarrow L^\infty$ decay rate $t^{-\frac{1}{3}}$, the nonlinearity in the Duhamel formula gives
\begin{equation}
    \int_1^t\frac{1}{(t-s)^{\frac{1}{3}}} (s^{-\frac{1}{3}})^3\,ds
\end{equation}since we have to put two pieces of $u$ in $L^2$. The naive integral above results in a log loss.
A natural way to access such questions is the \emph{space--time resonance} (STR) method, see \cites{STR, GMS,GPR,GNT,GNT2}.
%: one decomposes the Duhamel formula for the profile into regions where the oscillatory phase is nonstationary (hence amenable to integration by parts) and regions where space--time resonances persist and must be handled by finer structural input.
The STR perspective is particularly well suited to \eqref{eq:gBBM} because it isolates the genuinely resonant interactions and converts the analysis into a precise geometric study of the resonant sets.

A closely related scattering result for \eqref{eq:gBBM} with \emph{quartic} ($p=3$) nonlinearity was obtained by Morgan \cite{Morgan}.  It is tempting to think that higher-power nonlinearities should only make the problem easier, since additional factors of $u$ often translate into additional decay.  Indeed, in the quintic gBBM setting, away from resonance the quintic nonlinearity yields better time integrability than the quartic one, and many nonresonant estimates can be shortened compared to the quartic analysis.

The surprise is that the \emph{resonant} analysis becomes \emph{more} subtle at  some resonance points.  At the quartic level, certain cancellations arising from algebraic structure of the phase function  (often described informally as ``null'' mechanisms) suppress the leading contribution of the most persistent resonant interactions.  In the quintic problem these vanishing mechanisms no longer align with the geometry of the resonant set, and the leading resonant contributions must be estimated directly.
As a result, the frequency space must be decomposed more finely, and the proof must exploit sharper information.
%about (i) the local geometry of the resonant manifolds, (ii) how the oscillatory phase degenerates near stationary points of the dispersion, and (iii) how to recover integrability using a combination of weighted bounds and localized oscillatory estimates.
%This paper develops a set of localized estimates that capture these effects and allow one to treat the resonant regions without appealing to cancellations that are specific to the quartic problem.
%
%
%In particular, we combine:
%\begin{itemize}
%\item a bootstrap framework that propagates high-regularity control and weighted bounds compatible with STR localizations;
%\item refined plane-wave interaction bounds near each resonant component, tailored to the local degeneracy of the phase;
%\item improved decay away from stationary points of the linear dispersion, which restores time integrability in the complementary region.
%\end{itemize}
The net effect is that the quintic nonlinearity is \emph{not} handled merely by adding one more factor of $u$ to Morgan's argument;  in particular,  at certain non-degenerate resonance points, new structures and symmetries have to be exploited.

\subsection{Main results}
Roughly speaking, we prove that if the initial data are sufficiently small and smooth, then the solution to \eqref{eq:gBBM} exists globally, satisfies sharp dispersive decay, and scatters to a linear BBM solution.
\begin{theorem}\label{thm:main}
Fix \footnote{We do not claim the optimality for the regularity here. This is picked for the sake of convenience.} $s \ge 100$. Consider the Cauchy problem \eqref{eq:gBBM} with $p = 4$ and initial condition 
\begin{equation}
u(t=1,x) = u_0(x).
\end{equation}
There exists $\eps_0 \in (0,1)$ such that if the initial data satisfy
$$\snorm{\widehat{u_0}}_{L^\infty_\xi}+\snorm{xu_0(x)}_{L^2_x}+\snorm{u_0(x)}_{H^s_x}<\eps_0,$$
then the following hold
\begin{itemize}
\item There exists a unique global-in-time solution $u(t,x)$ to \eqref{eq:gBBM};

\item If we denote
$$\omega(\xi) := \xi \langle \xi \rangle^{-2}\quad \text{and} \quad\widehat{f} := e^{i\omega(\xi)t}\widehat{u},$$
then:% there exist $p_0, p_1 \in (0,\tfrac16)$ such that $f(t,x)$ satisfies the mild growth estimate
$$\sup_{t \in [1,\infty)}\left[\|\widehat{f}\|_{L^\infty_\xi}+ \snorm{x f}_{L^2_x}+ \snorm{f}_{H^s_x}\right]\lesssim \eps_0;$$

\item $u(t,x)$ obeys the dispersive decay bound
$$\snorm{u(t,x)}_{L^\infty_x}\lesssim \eps_0\, t^{-\frac{1}{3}}\qquad \forall t \geq 1;$$

\item $u(t,x)$ scatters in $H^s_x$: there exists $F(x) \in H^s_x$ such that
$$\lim_{t \to \infty}\snorm{|u(t,x) - e^{-i\omega\left(\frac{1}{i}\partial_x\right)t}F(x)}_{H^s_x}= 0.$$
\end{itemize}
\end{theorem}

%\medskip

%\subsection{Organization of the paper}

\subsection{Notation}
Throughout the paper, $A\lesssim B$ means $A\le C B$ for a harmless constant $C>0$.
We write $\langle\xi\rangle=(1+\xi^2)^{\frac{1}{2}}$.

We use the Fourier transform
$$
\widehat{ f}(\xi)=\int_{\R}e^{-i x\xi}f(x)\,dx,
\qquad
\Check{f}(x)=\frac1{2\pi}\int_{\R}e^{i x\xi} f(\xi)\,d\xi.
$$
For a set $S\subset  \R$, we denote $S^\complement$ its complement in $\R$, i.e. $S^\complement=\R\setminus S.$
A set and its complement. Whenever we write $\sum_{\operatorname{freq.}}$, the summation is 
taken over all dyadic frequency parameters consistent with the Littlewood--Paley localizations under consideration.

\subsection{Acknowledgment}G.C. was partially supported by NSF Grant DMS-2350301, by the Simons Foundation MP-TSM-0000225, and by the Stefan Bergman Fellowship. 
Parts of this work was done during Y.Z.'s visit to School of Mathematics at Georgia Tech with the support of NSF Grant DMS-2350301.  Y.Z. would like to thank  School of Mathematics at Georgia Tech for its hospitality.
\section{Preliminaries}

\subsection{Dispersive relations}
We begin by verifying certain symmetries
\begin{equation}\label{eq:disperrelaton}
    \omega(\xi)=\frac{\xi}{\jxi^2}.
\end{equation} By direct computations
$$
\omega'(\xi)=\frac{1-\xi^{2}}{(1+\xi^{2})^{2}}.
$$
which is an even function.
Let
$$
r(\xi)=\mathrm{sgn}(\xi)\sqrt{\frac{\xi^{2}+3}{\xi^{2}-1}}
\qquad\Longrightarrow\qquad
r(\xi)^{2}=\frac{\xi^{2}+3}{\xi^{2}-1}.
$$
Since $\omega'(\xi)$ depends only on $\xi^{2}$, the factor $\mathrm{sgn}(\xi)$ is irrelevant. Note that
$$
1-r(\xi)^{2}
=1-\frac{\xi^{2}+3}{\xi^{2}-1}
=\frac{\xi^{2}-1-(\xi^{2}+3)}{\xi^{2}-1}
=-\frac{4}{\xi^{2}-1},
$$
and
$$
1+r(\xi)^{2}
=1+\frac{\xi^{2}+3}{\xi^{2}-1}
=\frac{\xi^{2}-1+\xi^{2}+3}{\xi^{2}-1}
=\frac{2(\xi^{2}+1)}{\xi^{2}-1},
$$
hence
$$
(1+r(\xi)^{2})^{2}=\frac{4(\xi^{2}+1)^{2}}{(\xi^{2}-1)^{2}}.
$$
Therefore,
\begin{align}\label{eq:omegarsym}
\omega'(r(\xi))
&=\frac{1-r(\xi)^{2}}{(1+r(\xi)^{2})^{2}}
=\frac{-4/(\xi^{2}-1)}{4(\xi^{2}+1)^{2}/(\xi^{2}-1)^{2}}
=-\frac{\xi^{2}-1}{(\xi^{2}+1)^{2}} \\
&=\frac{1-\xi^{2}}{(1+\xi^{2})^{2}}
=\omega'(\xi).
\end{align}
\subsection{Littlewood--Paley projections }

We use standard Littlewood--Paley projections to localize frequencies dyadically. Throughout the paper, we define $\varphi(\xi)$ as a smooth bump function which has the support of $[-2,2]$, and $\varphi_{[-1,1]}\equiv 1$. We define $\psi(\xi):=\varphi(\xi)-\varphi(2\xi)$. Given $k\in \Z$, let $\varphi_{\leq k}=\varphi(\frac{\xi}{2^k})$ and $\psi_k=\psi(\frac{\xi}{2^k})$.

\begin{definition}[LP Projections]
For any $f \in L_x^2(\mathbb{R})$ and any $k \in \mathbb{Z}$, define
$$f_k := (\psi_k \widehat{f})^\vee\qquad \text{and} \qquad f_{\le k} := (\phi_{\le k} \widehat{f})^\vee.$$
The function $f_k$ is called the $k^{\text{th}}$ LP piece of $f(x)$.
The multiplier operators $f \mapsto f_k$, $f \mapsto f_{\le k}$ are called
\emph{LP projections}.
\end{definition}

From standard Littlewood–Paley theory, we know any reasonable function can be reconstructed as a sum of
LP pieces.

\begin{lemma}[LP Decomposition]
For any $f \in L_x^2$ and any fixed $k_{lo} \in \mathbb{Z}$, the following identities hold:
$$f(x) = \sum_{k \in \mathbb{Z}} f_k(x)\qquad \text{(homogeneous LP decomposition)},$$
$$f(x) = f_{\le k_{\operatorname{lo}}}(x) + \sum_{k > k_{\operatorname{lo}}} f_k(x)\qquad \text{(inhomogeneous LP decomposition)}.$$
\end{lemma}

\subsection{Linear estimates}

%Put $\omega$ here.

%comput that thereare reflextion points $\pm \sqrt 3$.

Finally, we recall standard decay estimates.   For the proof, see Morgan \cite{Morgan}.

\begin{lemma}\label{lem:poitwise}
Let $C_{\mathrm{hi}} \ge 2^4$, $C_{\mathrm{lo}} \le 2^{-2}$ and pick $s \ge \tfrac{11}{2}$.
For any sufficiently nice function $u(t,x)$ with profile $f(t,x)$, we have the bounds
\begin{equation}\label{frequencybounds}
\|u_k(t,x)\|_{L_x^\infty} \lesssim
\begin{cases}
2^{-k(s-1)} \|f\|_{H_x^s},
& \text{if } 2^k \ge C_{\mathrm{hi}} t^{\frac{1}{9}}, \\[6pt]

t^{-\frac{1}{2}} 2^{\frac{3k}{2}} \|\widehat{f}\|_{L_\xi^\infty}
\;+\;
t^{-\frac{3}{4}} 2^{\frac{9k}{4}} \|\partial_\xi \widehat{f}_k\|_{L_\xi^2},
& \text{if } 2^3 \le 2^k < C_{\mathrm{hi}} t^{\frac{1}{9}}, \\[6pt]

t^{-\frac{1}{3}} \|\widehat{f}\|_{L_\xi^\infty}
\;+\;
t^{-\frac{1}{2}} \|\partial_\xi \widehat{f}_k\|_{L_\xi^2},
& \text{if } 2^{-1} \le 2^k < 2^3, \\[6pt]

t^{-\frac{1}{2}} 2^{-\frac{k}{2}} \|\widehat{f}\|_{L_\xi^\infty}
\;+\;
t^{-\frac{3}{4}} 2^{-\frac{3k}{4}} \|\partial_\xi \widehat{f}_k\|_{L_\xi^2},
& \text{if } C_{\mathrm{lo}} t^{-\frac{1}{3}} \le 2^k < 2^{-1}, \\[6pt]

2^k \|\widehat{f}\|_{L_\xi^\infty},
& \text{if } 2^k < C_{\mathrm{lo}} t^{-\frac{1}{3}}.
\end{cases}
\end{equation}
In particular, given the setting above,  for $t\geq1$, one has
\begin{equation}
    \|u(t,x)\|_{L_x^\infty} \lesssim  t^{-\frac{1}{3}} \|\widehat{f}\|_{L_\xi^\infty}+ t^{-\frac{1}{2}} \Big(\|f\|_{H_x^s}+\|\partial_\xi \widehat{f}\|_{L_\xi^2}\Big).
\end{equation}

\end{lemma}

\section{Space--time resonance sets}

For convenience, we start to solve the equation \eqref{eq:gBBM} with $p=4$ from $t=1$ using the Duhamel formula.

Define the profile
$$f(t,x):=e^{it\omega(D)}u(t,x),\qquad \widehat{f}(t,\xi)=e^{it\omega(\xi)}\widehat{u}(t,\xi)$$
where $D=-i\partial_x$ is a Fourier multiplier.

Taking the Fourier transform, one can directly check that $\hat{f}(t)$ satisfies
\begin{align}\label{eq:f_duhamel}
  \widehat{f}(t,\xi) =  \widehat{f}(1,\xi) 
+i\frac{\omega(\xi)}{\left(2\pi\right)^{2}}
    \int_{1}^{t}\int  e^{-i\tau\phi}  
    \hat{f}(\tau,\eta_{1})\ \hat{f}(\tau,\eta_{2})\ \hat{f}(\tau,\eta_{3})\ \hat{f}(\tau,\eta_4) \hat{f}\left(\tau,\eta_{5}\right)\,d\eta_{1234}  d \tau
\end{align}
where for the notation convenience, we set $\eta_5:=\xi-\sum^4_{j=1}\eta_j$ and $d\eta_{1234}:=d\eta_1d\eta_2d\eta_3d\eta_4$,  and 
\begin{align}
    \phi:=-\omega(\xi)+\sum^5_{j=1}\omega(\eta_j).
\end{align}
Motivated by Lemma \ref{lem:poitwise}, to study the long time behavior of the solution $u$ is reduced to bound certain norms of $\hat{f}$. In order to achieve these bounds, we analyze the wave interactions in the nonlinear term via space--time resonance analysis.

To determine all resonances, we have the following two categories
\begin{itemize}
    \item Time resonance: $\mathcal{T}:=\{(\eta_1,\eta_2,\eta_3, \eta_4; \xi)\in \mathbb{R}^5\,|\,\phi=0\}.$
   \item  Space resonance: $\mathcal{S}:=\{(\eta_1,\eta_2,\eta_3, \eta_4; \xi)\in \mathbb{R}^5\,|\,\partial_{\eta_j}\phi=0, \, \forall j=1,2,3,4.\}.$
\end{itemize}
Away from the time resonance set, we will perform integration by parts in time, and  away from the space resonance set, we will integrate by parts in frequency variables.

The space--time resonance set is given by $\mathcal{T}\bigcap \mathcal{S}$. In the remaining part of this section, we compute the resonance sets.

\subsection{Determining resonances}
\begin{proposition}\label{prop:resonances}
   Recall the following function which plays an important role in the symmetries of $\omega(\xi)$, see \eqref{eq:omegarsym}:
\begin{equation}\label{eq:reflection_def}
r(\xi) :=\operatorname{sgn}\left(\xi\right)\sqrt{\frac{\xi^2+3}{\xi^2-1}}
\end{equation}
then up to permuting and negating variables, the only space--time resonances for $p=4$ \eqref{eq:gBBM} is
\begin{itemize}
    \item on the line $L$ defined by
    \begin{equation}
        L=\Big\{\,(\sigma_1\eta,\sigma_2\eta,\sigma_3\eta,\sigma_4\eta;\ \xi=\sigma_\xi\eta)\ :\ \eta\in\mathbb R,\ \sigma_j\in\{\pm1\},\ \sigma_\xi-\!\sum_{j=1}^4\sigma_j=\pm1 \,\Big\}\subseteq \mathbb{R}^5,
        \label{eqn:resonant_line_discussion2}
    \end{equation}
 including the point $(0,0,0,0;0)$ and, when $|\eta|=\sqrt{3}$, \emph{all} sign–permutation points of the form
$$
(\eta_1,\eta_2,\eta_3,\eta_4;\xi)
= (\sigma_1 \eta,\sigma_2 \eta,\sigma_3 \eta,\sigma_4 \eta;\ \xi=\sigma_\xi \eta),
\quad \sigma_j\in\{\pm1\},
$$
that satisfy the resonance sign constraint
$$
\sigma_\xi-\sum_{j=1}^4 \sigma_j \in \{\pm1\}
$$
(equivalently, among $\eta_1,\eta_2,\eta_3,\eta_4,\eta_5$
there are exactly $3$ $+\eta$’s and $2$ $-\eta$’s, or vice versa).
For example,
$(-\sqrt3,-\sqrt3,\sqrt3,\sqrt3;\ \sqrt3)$,
$(-\sqrt3,\sqrt3,-\sqrt3,\sqrt3;\ \pm\sqrt3)$, 
$(\sqrt3,\sqrt3,\sqrt3,-\sqrt3;\ \sqrt3)$,
etc., all lie on $L$.
    \item on the curve $\Gamma$ defined by 
      \begin{equation}
    \Gamma =  \left\{\left(\eta_1,\eta_2,\eta_3,\eta_4;\xi\right) = \left(\eta,\eta,-\eta,-r(\eta); \xi=- r(\eta)\right) \ | \ \left|\eta\right|>1 \right\}\subseteq \mathbb{R}^5;
      \label{eqn:resonant_curve_discussion}
    \end{equation}
    \item at the point $\left(\eta_0, \eta_0, \eta_0, \eta_0;\xi_0\right)$ where
    \begin{equation}\label{eq:def_xi_eta_nod}
        \eta_0\approx 7.34 \qquad \text{and}\qquad \xi_0\approx 28.31
    \end{equation}
    are the unique positive solutions to the system 
\begin{subequations}
\label{eqn:fam_1_badpoints}
\begin{align}
4\eta_{1} - r\left(\eta_{1}\right) &=\xi,
\\
  -\omega(4\eta-r(\eta_1))+4\omega(\eta_1)-\omega(r(\eta_1))&=0.
 \end{align}
 \end{subequations}
\end{itemize} 
We refer to it as the \emph{anomalous resonance}.
\end{proposition}
The proof is presented in the rest of the section. \subsection{Proof of Proposition \ref{prop:resonances}}
We begin by finding the critical points of the nonlinear phase function $\phi(\eta_1,\eta_2,\eta_3,\eta_4; \xi)$. A critical point $(\eta_1,\eta_2,\eta_3,\eta_4;\xi)$ must satisfy
    \begin{align*}
        \partial_{\eta_{1}}\phi &= \omega'(\eta_{1})-\omega'(\xi-\eta_{1}-\eta_{2}-\eta_{3}-\eta_4) = 0,
        \\
        \partial_{\eta_{2}}\phi &= \omega'(\eta_{2})-\omega'(\xi-\eta_{1}-\eta_{2}-\eta_{3}-\eta_4) = 0,
        \\
        \partial_{\eta_{3}}\phi &= \omega'(\eta_{3})-\omega'(\xi-\eta_{1}-\eta_{2}-\eta_{3}-\eta_{4}) = 0,
        \\
        \partial_{\eta_{4}}\phi &= \omega'(\eta_{4})-\omega'(\xi-\eta_{1}-\eta_{2}-\eta_{3}-\eta_{4}) = 0.
    \end{align*}
    In particular,
    $$
    \omega'(\eta_{1}) = \omega'(\eta_{2}) = \omega'(\eta_{3}) = \omega'(\eta_{4}) = \omega'(\xi-\eta_{1}-\eta_{2}-\eta_{3}-\eta_{4}) . 
    $$
Using the symmetries of $\omega'$, \eqref{eq:omegarsym}, we know that there are several families of solutions to the above system:
\begin{itemize}
\item Family 1: $|\eta_{1}|=|\eta_{2}|=|\eta_{3}|=|\eta_{4}|$,
\item Family 2: $|\eta_{1}|=|\eta_{2}|=|\eta_{3}|=|r\left(\eta_{4}\right)|$,
\item Family 3: $|\eta_{1}|=|\eta_{2}|=|r\left(\eta_{3}\right)|=|\eta_{4}|$,
\item Family 4: $|\eta_{1}|=|r\left(\eta_{2}\right)|=|\eta_{3}|=|\eta_{4}|$,
\item Family 5:
$|r\left(\eta_{1}\right)|=|\eta_{2}|=|\eta_{3}|=|\eta_{4}|$,
\item Family 6: $|\eta_i| = |\eta_j| = |r(\eta_k)| = |r(\eta_l)|$ for some permutation $(i,j,k,l)$ of $(1,2,3,4)$.
\end{itemize}
Note that we can only have a solution in families 2, 3, or 4 when $|\eta_{1}|, |\eta_{2}|, |\eta_{3}|, |\eta_{4}| >1$. 
Additionally, in the special case $\xi=\eta$, we can pick out two types of resonant manifolds. 
\begin{enumerate}
    \item First, the line $L$ defined as
    \begin{equation}
        L = \left\{\left(\eta_1,\eta_2,\eta_3,\eta_4;\xi\right) = \left(-\eta,\eta,\eta,\eta; \eta\right) \ | \ \eta\in \mathbb{R} \right\}\subseteq \mathbb{R}^5,
        \label{eqn:resonant_line_discussion}
    \end{equation}
    including the points $(0,0,0,0;0)$ and $\left(-\sqrt{3},\sqrt{3},\sqrt{3},\sqrt{3};\sqrt{3}\right)$;  consists entirely of space--time resonances. We can get other resonant lines by changing the place of the single minus sign (for instance, we can have it fall on $\eta_2$ instead of $\eta_1$). 
    \item Additionally, the curve $\Gamma$ defined as 
        \begin{equation}
    \Gamma =  \left\{\left(\eta_1,\eta_2,\eta_3,\eta_4;\xi\right) = \left(\sigma_1\eta,\sigma_2\eta,\sigma_3\eta,\sigma_4 r(\eta); \sigma_\xi r(\eta)\right) \ | \ \left|\eta\right|>1 \right\}\subseteq \mathbb{R}^5.
    \end{equation}
    is space--time resonant too. Note that $\eta_5$ in this case is $\eta_1$. We can generate related space--time resonant curves by permuting the $\eta_j$'s for $j=1,...,5$, i.e.
    \begin{equation}
    \Gamma' =  \left\{\left(\eta_1,\eta_2,\eta_3,\eta_4;\xi\right) = \left(\sigma_1\eta,\sigma_2\eta,\sigma_3r(\eta),\sigma_4 r(\eta); \sigma_\xi r(\eta)\right) \ | \ \left|\eta\right|>1 \right\}\subseteq \mathbb{R}^5.
    \end{equation}
\end{enumerate}
Observe that the symmetric case $\xi=-\eta$ or $\xi=-r(\eta)$ also exists, and the process would be same. Therefore, without loss of generality, we only discuss $\xi=\eta, r(\eta)$ here.

To analyze the time resonance condition for these families, we substitute the spatial resonance constraints into the phase equation $\sum_{j=1}^5 \omega(\eta_j) - \omega(\xi) = 0$. By utilizing the odd symmetry of the dispersion relation $\omega(\cdot)$, the time resonance condition for all subfamilies can be unified into finding the real roots of the following generalized phase function:
\begin{equation} \label{eq:generalized_phase}
    \Phi_{A,B}(\eta) := A\omega(\eta) + B\omega(r(\eta)) - \omega(A\eta + Br(\eta)) = 0
\end{equation}
where $A$ and $B$ are integers representing the net algebraic multiplicity of the frequency $\eta$ and the reflected frequency $r(\eta)$, respectively. Because we are considering the quintic case ($p=4$ in \eqref{eq:gBBM}), the coefficients satisfy $|A| + |B| \leq 5$. The existence and location of space--time resonances are entirely determined by the roots of $\Phi_{A,B}(\eta)$. We classify these roots comprehensively in the following five lemmas.

\begin{lemma}\label{lem:resonance_manifold}
    If the coefficient pair $(A,B)$ satisfies
    $$(A,B) \in \{(1,0), (-1,0), (0,1), (0,-1)\},$$ the generalized phase function $\Phi_{A,B}(\eta)$ vanishes identically on its domain. 
    \begin{itemize}
        \item For $(A,B) = (\pm 1, 0)$, the condition reduces to $\xi = \pm \eta$, which corresponds to the resonant line $L$ defined in \eqref{eqn:resonant_line_discussion2}.
        \item For $(A,B) = (0, \pm 1)$, the condition reduces to $\xi = \pm r(\eta)$, which corresponds to the resonant curve $\Gamma$ defined in \eqref{eqn:resonant_curve_discussion} for $|\eta| > 1$.
    \end{itemize}
\end{lemma}
\begin{proof}
    Direct substitution into \eqref{eq:generalized_phase} yields $\pm \omega(\eta) - \omega(\pm \eta) = 0$ and $\pm \omega(r(\eta)) - \omega(\pm r(\eta)) = 0$, which hold trivially due to the odd symmetry of $\omega$.
\end{proof}
\begin{lemma}\label{lem:trivial_origin}
    If $B = 0$ and $|A| \ge 3$, specifically, $(A,B) \in \{\pm(3,0), \pm(5,0)\}$. Then, the equation $\Phi_{A,0}(\eta) = A\omega(\eta) - \omega(A\eta) = 0$ admits exactly one real root: $\eta = 0$. This corresponds to the degenerate space--time resonant point $(0,0,0,0;0)$ on the line $L$.
\end{lemma}
\begin{proof}
    We will argue by contradiction. Recall that the dispersion relation $\omega(\xi) = \frac{\xi}{\jxi^2}$ For $A \in \{\pm 3, \pm 5\}$, bring this into the generalized phase function, it follows that
    $$A\frac{\eta}{1+\eta^2}=\frac{A\eta}{1+(A\eta)^2}.$$
    Suppose for the sake of contradiction that there exists a non-zero real root $\eta \neq 0$. We may then divide both sides by $A\eta$ to obtain:
    \begin{equation}
        \frac{1}{1+\eta^2} = \frac{1}{1+A^2\eta^2}.
    \end{equation}
    This further implies that $1+A^2\eta^2 = 1+\eta^2$, which simplifies to $\eta^2(A^2-1) = 0$. Since $|A| > 1$, we have $A^2-1 \neq 0$, which requires $\eta^2 = 0$. Therefore, the equation admits no real solutions other than the trivial intersection at $\eta = 0$.
\end{proof}
\begin{lemma}\label{lem:sqrt3_res}
    If the coefficients satisfy $A + B = \pm 1$ and $|A| + |B| \ge 3$ i.e. $$(A,B) \in \{(2,-1), (-1,2), (3,-2), \dots\},$$ the equation $\Phi_{A,B}(\eta) = 0$ admits real roots at $\eta = \pm \sqrt{3}$.
\end{lemma}
\begin{proof}
    By the definition of $r(\eta)$ in \eqref{eq:reflection_def}, we observe the property that $r(\pm \sqrt{3}) = \pm \sqrt{3}$. Substituting $\eta = \sqrt{3}$ into \eqref{eq:generalized_phase}, we obtain:
    \begin{align*}
        \Phi_{A,B}(\sqrt{3}) &= A\omega(\sqrt{3}) + B\omega(\sqrt{3}) - \omega((A+B)\sqrt{3}) \\
        &= (A+B)\omega(\sqrt{3}) - \omega((A+B)\sqrt{3}).
    \end{align*}
    Since $A+B = \pm 1$ and $\omega$ is odd, this evaluates to $\pm \omega(\sqrt{3}) \mp \omega(\sqrt{3}) = 0$. By symmetry, the same holds for $\eta = -\sqrt{3}$.
\end{proof}
\begin{lemma}\label{lem:anomalous_res}
    For the coefficient pairs $(A,B) \in \{(4,-1), (-4,1)\}$, the equation $\Phi_{A,B}(\eta) = 0$ admits a unique positive root $\eta_0 \approx 7.34$ and its negative counterpart $-\eta_0$. This root corresponds to the anomalous resonance defined in \eqref{eq:def_xi_eta_nod}.
\end{lemma}
\begin{proof}
    By the odd symmetry of $\omega$, it suffices to consider $(A,B) = (4,-1)$. The equation becomes $4\omega(\eta) - \omega(r(\eta)) - \omega(4\eta - r(\eta)) = 0$.

Denote 
\begin{equation}
    H(\eta):=4\omega(\eta)-\omega\big(r(\eta)\big)
-\omega\big(4\eta-r(\eta)\big).
\end{equation}Note that $H$ is an odd function so it suffices to consider its zeros for  $1<\eta$.

We record some elementary computations. For $\eta>1$ we have
$$
r(\eta)=\sqrt{\frac{\eta^2+3}{\eta^2-1}}>1.
$$
A direct computation gives
$$
1+r(\eta)^2=\frac{2(\eta^2+1)}{\eta^2-1},
\qquad
\omega\big(r(\eta)\big)
=\frac{\sqrt{(\eta^2+3)(\eta^2-1)}}{2(\eta^2+1)}.
$$
Moreover
$$
\omega'(\xi)=\frac{1-\xi^2}{(1+\xi^2)^2},
\qquad
|\omega(\xi)|\le \frac12 \quad\text{for all }\xi.
$$
Differentiating $r$ for $\eta>1$,
%$$
%\frac{r'(\eta)}{r(\eta)}
%=\frac12\Big(\frac{2\eta}{\eta^2+3}-%\frac{2\eta}{\eta^2-1}\Big)
%=-\frac{4\eta}{(\eta^2+3)(\eta^2-1)},
%$$
%hence
\begin{equation}\label{eq:rprime}
r'(\eta)=-\frac{4\eta\,r(\eta)}{(\eta^2+3)(\eta^2-1)}<0.
\end{equation}
For $1<\eta\le2$, with $|\omega|\le\frac12$, one has
$$
\omega\big(r(\eta)\big)+\omega\big(4\eta-r(\eta)\big)\le1.
$$
On the other hand,
$$
4\omega(\eta)=\frac{4\eta}{1+\eta^2}
\ge \frac{8}{5}>1,
$$
where the minimum on $[1,2]$ occurs at $\eta=2$. Thus
$$
H(\eta)>0\qquad\text{for }1<\eta\le2.
$$
For $\eta\ge2$, one has
$$
H'(\eta)
=4\omega'(\eta)
-\omega'\big(r(\eta)\big)\,r'(\eta)
-\omega'\big(4\eta-r(\eta)\big)\,(4-r'(\eta)).
$$
Since $r(\eta)>1$ and $r'(\eta)<0$, we have $
-\omega'(r(\eta))\,r'(\eta)<0$,
so it suffices to bound the remaining terms from above.    

For $\eta\ge2$,
$$
4\omega'(\eta)
=\frac{4(1-\eta^2)}{(1+\eta^2)^2}
\le -\frac{1}{2\eta^2}.
$$
Also $x(\eta):=4\eta-r(\eta)>1$ and
$$
-\omega'(x)\le \frac{1}{x^2}.
$$
From \eqref{eq:rprime}, $|r'(\eta)|\le1$ for $\eta\ge2$, hence
$4-r'(\eta)\le5$. Therefore
$$
-\omega'(x(\eta))\,(4-r'(\eta))
\le \frac{5}{(4\eta-r(\eta))^2}
\le \frac{5}{(4\eta-\sqrt{\frac{7}{3}})^2}.
$$
A direct comparison shows that for all $\eta\ge2$,
$$
\frac{5}{(4\eta-\sqrt{\frac{7}{3}})^2}
<\frac{1}{2\eta^2},
$$
and consequently $H'(\eta)<0$ for $\eta\ge2$. Thus $H$ is strictly
decreasing on $[2,\infty)$.

We already have $H(2)>0$. Moreover, as $\eta\to\infty$,
$$
r(\eta)\to1,\qquad
\omega\big(r(\eta)\big)\to\frac12,\qquad
\omega(4\eta-r(\eta))\to0,\qquad
4\omega(\eta)\to0,
$$
so
$$
\lim_{\eta\to\infty}H(\eta)=-\frac12<0.
$$
By continuity and strict monotonicity on $[2,\infty)$, there exists a
unique $\eta_0>2$ such that $H(\eta_0)=0$.
Moreover \eqref{eq:def_xi_eta_nod} is the only solution\footnote{Verified by Mathematica}.
\end{proof}
\begin{lemma}\label{lem:no_roots}
    For all other valid coefficient pairs satisfying $|A| + |B| \le 5$, such as $(A,B) \in \{(4,1), (3,2), (2,1), \dots\}$, the generalized phase equation $\Phi_{A,B}(\eta) = 0$ admits no real roots in the domain $|\eta| > 1$. Consequently, subfamilies reducing to these pairs exhibit space resonances but no time resonances.
\end{lemma}

\noindent
\textbf{Family 1}.
\noindent
 In this case we must have 
 \begin{align}
 |\eta_{1}|=|\eta_{2}|=|\eta_{3}|=|\eta_{4}| =\begin{cases} 
  |\xi-\eta_{1}-\eta_{2}-\eta_{3}-\eta_{4}| 
  \\
   \quad\quad\quad \text{or}
 \\
\left|r\left(\xi-\eta_{1}-\eta_{2}-\eta_{3}-\eta_{4}\right)\right|
 \end{cases}
 \label{eqn:fam_1_main_eqn}
 \end{align}
There are several subfamilies where we consider
\begin{subequations}
\label{eqn:fam_1_subfam_quintic}
\begin{align}
\eta_1 &= \eta_2 = \eta_3 = \eta_4, \label{eqn:fam_1_a_quintic} \\
-\eta_1 &= \eta_2 = \eta_3 = \eta_4, \label{eqn:fam_1_b_quintic} \\
\eta_1 &= -\eta_2 = \eta_3 = \eta_4, \label{eqn:fam_1_c_quintic} \\
\eta_1 &= \eta_2 = -\eta_3 = \eta_4, \label{eqn:fam_1_d_quintic} \\
\eta_1 &= \eta_2 = \eta_3 = -\eta_4 ,
\label{eqn:fam_1_e_quintic}\\
-\eta_i &= -\eta_j = \eta_k = \eta_l \,\text{ for some permutation }(i,j,k,l) \text{ of }(1,2,3,4).\label{eqn:fam_1_f_quintic}
\end{align}
each corresponding to four equations for a critical point by \eqref{eqn:fam_1_main_eqn}.
\end{subequations}
\\\\
\noindent
{\it \underline{Subfamily $\eqref{eqn:fam_1_a_quintic}$}}.
For \eqref{eqn:fam_1_a_quintic}, using that $\xi-\eta_1-\eta_2-\eta_3-\eta_4=\xi-4\eta_1$ we are reduced to solve
\begin{equation}\label{eq:eta1_reduced}
\eta_1 = \pm(\xi-4\eta_1)
\quad\text{or}\quad
\eta_1 = \pm r(\xi-4\eta_1).
\end{equation}
The first two identities are linear, and yields that they only have the critical points at $$\eta_j=\frac{\xi}{5} \qquad\text{or}\qquad \eta_j=\frac{\xi}{3}\qquad \text{for all }j=1,2,3,4$$ and this space resonance vanishes only at $\xi=0$. Hence, the trivial solution $$(\eta_1,\eta_2,\eta_3,\eta_4;\xi=\eta=0)=(0,0,0,0,0;\xi=\eta=0)$$
which is a degenerate point  by Lemma \ref{lem:trivial_origin}. The remaining two identities are nonlinear, and for $\eta_1=r(\xi-4\eta_1)$, it implies that
$$4\eta_1+r(\eta_1)=\xi.$$
Since $(\eta_1,\eta_2,\eta_3,\eta_4,\eta_5)=(\eta,\eta,\eta,\eta,r(\eta))$ and $ \phi=-\omega(\xi)+\sum^5_{j=1}\omega(\eta_j)=0$, it follows that
\begin{align}\label{isolated_resonant_points_sub_1}
    \omega(\xi)&=+\sum^5_{j=1}\omega(\eta_j).
\end{align}
The right hand side of the \eqref{isolated_resonant_points_sub_1} gives
$$\omega(r(\eta))+4\omega(\eta)$$
and the left hand side of the \eqref{isolated_resonant_points_sub_1} gives
$$\omega(\xi)=\omega(4\eta+r(\eta)).$$
Thus, we have
\begin{equation}
4\omega(\eta)+\omega(r(\eta))-\omega(4\eta+r(\eta))=0.
\end{equation}
However, this is an equation admits no roots, as shown by Lemma \ref{lem:no_roots}.
For $\eta_1=-r(\xi-4\eta_1)$, it follows a very similar argument of the analysis for $\eta_1=r(\xi-4\eta_1)$, and it gives 
\begin{equation}\label{isolated_resonant_points_sub_2}
    -\omega(4\eta-r(\eta))+4\omega(\eta)-\omega(r(\eta))=0.
\end{equation}
Here, we have $\vert \eta\vert \approx 7.34$ and $\vert \xi\vert\approx 28.31$, which is categorized by Lemma \ref{lem:anomalous_res}.
\\\\
\noindent
{\it \underline{Subfamily $\eqref{eqn:fam_1_b_quintic}$}}.
For \eqref{eqn:fam_1_b_quintic}, using that $\xi-\eta_1-\eta_2-\eta_3-\eta_4=\xi+2\eta_1$ we are reduced to
\begin{equation}\label{eq:eta1_reduced_2}
\eta_1 = \pm(\xi+2\eta_1)
\quad\text{or}\quad
\eta_1 = \pm r(\xi+2\eta_1).
\end{equation}
The first set of identities yields the coefficient pair of the generalized phase function with $(A,B)=(\pm 3,0)$. This implies that the equation can only be satisfied if and only if $\xi=0$ and should be categorized by Lemma \ref{lem:trivial_origin}. The case $\eta_1= r(\xi+2\eta_1)$ of the second set nonlinear identities can be reduced by solving
$$-2\eta_1+r(\eta_1)=\xi.$$
Hence, we are solving the generalized phase function
$$-2\omega(\eta_1)+\omega(r(\eta_1))-\omega(-2\eta_1+r(\eta_1))=0$$
which should be classified by  Lemma \ref{lem:sqrt3_res} with roots $\vert \eta_1\vert \approx \sqrt{3}$ and $\xi=-2\eta_1+r(\eta_1)=-\eta_1$. So, we have a space--time resonance at $(\eta_1,\eta_2,\eta_3,\eta_4; \xi)=(-\sqrt{3},\sqrt{3},\sqrt{3},\sqrt{3};\sqrt{3})$ which is on the resonant line. The case $\eta_1=-r(\xi+2\eta_1)$ gives
$$-2\omega(\eta_1)-\omega(r(\eta_1))-\omega(-2\eta_1+r(\eta_1))=0$$
which also has no root by Lemma \ref{lem:no_roots}.
\\\\
\noindent
{\it \underline{Subfamily $\eqref{eqn:fam_1_c_quintic}$}}, {\it \underline{Subfamily $\eqref{eqn:fam_1_d_quintic}$}}, {\it \underline{Subfamily $\eqref{eqn:fam_1_e_quintic}$}} follow exactly the same process as {\it \underline{Subfamily $\eqref{eqn:fam_1_b_quintic}$}}.
\\\\
We next handle the subfamilies that have two minus sign, without loss of generality, take \eqref{eqn:fam_1_f_quintic} as example. Using that $\xi-\eta_1-\eta_2-\eta_3-\eta_4=\xi$, we are reduced to
$$\eta_1=\pm \xi\qquad \text{or}\qquad \eta_1=\pm r(\eta_1).$$
The first two identities are linear and yield either $(\eta_1,\eta_2,\eta_3,\eta_4,\eta_5)=(\eta,\eta,-\eta,-\eta,-\eta)\text{ or }(\eta,\eta,-\eta,-\eta,\eta)$ which are on the resonant lines based on Lemma \ref{lem:resonance_manifold}. For the remaining identities, we have 
$$\xi=\pm r(\eta_1),$$
which lies on the resonant curves based on Lemma \ref{lem:resonance_manifold}. The other members in this sub-family follow the exactly same process.\\\\
\noindent
\textbf{Family 2}.  In this case we must have 
 \begin{align}
 |\eta_{1}|=|\eta_{2}|=|\eta_3|=|r(\eta_4)| =\begin{cases} 
  |\xi-\eta_{1}-\eta_{2}-\eta_{3}-\eta_4| 
  \\
   \quad\quad\quad \text{or}
  \\
\left|r\left(\xi-\eta_{1}-\eta_{2}-\eta_{3}-\eta_4\right)\right|
 \end{cases}
 \label{eqn:fam_2_main_eqn}
 \end{align}
 We can break this expression into several subfamilies,
  \begin{subequations}
 \label{eqn:fam_2_subfam}
 \begin{align}
 \eta_{1}&=\phantom{-}\eta_{2}=\phantom{-}\eta_{3}=\phantom{-}r\left(\eta_4\right),\label{eqn:fam_2_a}
 \\
 -\eta_{1} &=\phantom{-} \eta_{2} =\phantom{-}\eta_{3}= \phantom{-}r\left(\eta_4\right),\label{eqn:fam_2_b}
 \\
 \eta_{1} &= -\eta_{2} =\phantom{-}\eta_{3}=\phantom{-}r\left(\eta_4\right),\label{eqn:fam_2_c}
 \\
  \eta_{1} &= \phantom{-}\eta_{2} =-\eta_{3}=\phantom{-}r\left(\eta_4\right),\label{eqn:fam_2_d}
 \\
 \eta_{1} &= \phantom{-}\eta_{2} =\phantom{-}\eta_{3} =- r\left(\eta_4\right),\label{eqn:fam_2_e}
  \\
 -\eta_{i} &= -\eta_{j} =\phantom{-}\eta_{k} = \phantom{-}r\left(\eta_4\right)\text{ for some permutation }(i,j,k)\text{ of } (1,2,3),\label{eqn:fam_2_f}
   \\
 -\eta_{i} &= \phantom{-} \eta_{j} =\phantom{-}\eta_{k} =- r\left(\eta_4\right)\text{ for some permutation }(i,j,k)\text{ of } (1,2,3).\label{eqn:fam_2_g}
 \end{align}
 \end{subequations}
 
\noindent
{\it \underline{Subfamily $\eqref{eqn:fam_2_a}$}}.
Using that $\xi-\eta_1-\eta_2-\eta_3-\eta_4=\xi-3\eta_1-r(\eta_1)$, we are reduced to solve
$$\eta_1=\pm\left(\xi-3\eta_1-r(\eta_1)\right)\qquad \text{or}\qquad \eta_1=\pm r(\xi-3\eta_1-r(\eta_1)).$$
These two sets of identities can be reduced to four cases:
\begin{subequations}
    \begin{align}
         \xi&=4\eta_1+r(\eta_1) \tag{a} \label{family_2_subfamily_1_case_1_a}
  \\
  \xi&=2\eta_1+r(\eta_1) \tag{b} \label{family_2_subfamily_1_case_1_b}
 \\
  \xi&=3\eta_1 +2r(\eta_1) \tag{c} \label{family_2_subfamily_1_case_1_c}
 \\
   \xi&= 3\eta_1. \tag{d} \label{family_2_subfamily_1_case_1_d}
    \end{align} 
\end{subequations}
For equation \eqref{family_2_subfamily_1_case_1_a}, we solve the nonlinear equation and  see if the solution gives a time resonance.  It follows that we need to compute
$$\Phi_a(\eta_1)=-\omega(4\eta_1+r(\eta_1))+4\omega(\eta_1)+\omega(r(\eta_1))=0.$$
This equation has no solutions, which follows from Lemma \ref{lem:no_roots}. Hence, there is no time resonance.\\
\noindent
For equation \eqref{family_2_subfamily_1_case_1_b}, we have the equation 
$$\Phi_b(\eta_1)=-\omega(2\eta_1+r(\eta_1))+2\omega(\eta_1)+\omega(r(\eta_1))=0.$$
This equation has no solutions, which follows from Lemma \ref{lem:no_roots}. Hence, there is no time resonance.
\\
\noindent
For equation \eqref{family_2_subfamily_1_case_1_c}, we have the equation 
$$\Phi_c(\eta_1)=-\omega(3\eta_1+2r(\eta_1))+3\omega(\eta_1)+2\omega(r(\eta_1))=0.$$
This equation has no solutions, which follows from Lemma \ref{lem:no_roots}. Hence, there is no time resonance.
\\
\noindent
For equation \eqref{family_2_subfamily_1_case_1_d}, we have the stationary point at $(\eta_1,\eta_2,\eta_3,\eta_4)=(\frac{\xi}{3},\frac{\xi}{3},\frac{\xi}{3},r(\frac{\xi}{3}))$, and this only hold when $\vert \xi\vert >3$, this implies the equation
$$\Phi_d (\xi)=\frac{8\xi^3}{9(1+\frac{\xi^2}{9})(1+\xi^2)}> 0\quad \forall \xi> 3.$$
Therefore, it degenerates to the resonance point $(\eta_1,\eta_2,\eta_3,\eta_4;\xi)=(0,0,0,0;0)$  by Lemma \ref{lem:trivial_origin} and lies on the resonant manifolds.\\

\noindent
{\it \underline{Subfamily $\eqref{eqn:fam_2_b}$}}.
Using that $\xi-\eta_1-\eta_2-\eta_3-\eta_4=\xi+\eta_1+r(\eta_1)$ we are reduced to solve
$$\eta_1=\pm \left(\xi+\eta_1+r(\eta_1)\right)\qquad \text{or}\qquad \eta_1=\pm r(\xi+\eta_1+r(\eta_1)).$$
These two sets of identities can be reduced to four cases:
\begin{subequations}
    \begin{align}
         \xi&=-r(\eta_1) \tag{a} \label{family_2_subfamily_1_case_2_a}
  \\
  \xi&=-2\eta_1-r(\eta_1) \tag{b} \label{family_2_subfamily_1_case_2_b}
 \\
  \xi&=-\eta_1 \tag{c} \label{family_2_subfamily_1_case_2_c}
 \\
   \xi&= -\eta_1-2r(\eta_1). \tag{d} \label{family_2_subfamily_1_case_2_d}
    \end{align} 
\end{subequations}
For equation \eqref{family_2_subfamily_1_case_2_a}, we have the equation
$$\Phi_a(\eta_1)=-\omega(-r(\eta_1))-\omega(\eta_1)-\omega(r(\eta_1))+\omega(\eta_1)=0.$$
This implies we have space--time resonance on $\xi=-r(\eta_1)$. By Lemma \ref{lem:resonance_manifold}, this is precisely the resonant curve $\Gamma$.
\\
\noindent
For equation \eqref{family_2_subfamily_1_case_2_b}, we have the equation
$$\Phi_b(\eta_1)=\omega(2\eta_1+r(\eta_1))-2\omega(\eta_1)-\omega(r(\eta_1))=0$$
and this equation has no solutions  by Lemma \ref{lem:no_roots}, which implies that there is no time resonance.
\\
\noindent
For equation \eqref{family_2_subfamily_1_case_2_c}, we have the equation 
$$\Phi_c(\eta_1)=-\omega(-\eta_1)-\omega(\eta_1)-\omega(r(\eta_1))+\omega(r(\eta_1))=0.$$
This implies that we have space--time resonances on $\xi=-\eta_1$. By Lemma \ref{lem:resonance_manifold}, this is precisely the resonant line $L$.
\\
\noindent
For equation \eqref{family_2_subfamily_1_case_2_d}, we have the equation 
$$\Phi_d(\eta_1)=\omega(\eta_1+2r(\eta_1))-\omega(\eta_1)-2\omega(r(\eta_1))=0.$$
and this equation has no solutions  by Lemma \ref{lem:no_roots}, which implies that there is no time resonance.\\\\
\noindent
{\it \underline{Subfamily $\eqref{eqn:fam_2_c}$}} and {\it \underline{Subfamily $\eqref{eqn:fam_2_d}$}} follows the same argument as {\it \underline{Subfamily $\eqref{eqn:fam_2_b}$}}.
\\\\
\noindent
{\it \underline{Subfamily $\eqref{eqn:fam_2_e}$}}
Using that $\xi-\eta_1-\eta_2-\eta_3-\eta_4=\xi-3\eta_1+r(\eta_1)$ we are reduced to solve
$$\eta_1=\pm \left(\xi-3\eta_1+r(\eta_1)\right)\qquad \text{or}\qquad \eta_1=\pm r(\xi-3\eta_1+r(\eta_1)).$$
These two sets of identities can be reduced to four cases:
\begin{subequations}
    \begin{align}
         \xi&=4\eta_1-r(\eta_1) \tag{a} \label{family_2_subfamily_1_case_3_a}
  \\
  \xi&=2\eta_1-r(\eta_1) \tag{b} \label{family_2_subfamily_1_case_3_b}
 \\
  \xi&=3\eta_1 \tag{c} \label{family_2_subfamily_1_case_3_c}
 \\
   \xi&= 3\eta_1-2r(\eta_1). \tag{d} \label{family_2_subfamily_1_case_3_d}
    \end{align} 
\end{subequations}
For equation \eqref{family_2_subfamily_1_case_3_a}, we have the equation 
$$4\omega(\eta_1) \mathbf{-} \omega(r(\eta_1)) - \omega(4\eta_1 - r(\eta_1)) = 0$$
and this equation has a coefficient pair $(A,B)=(4,-1)$. Therefore, this generates the anomalous resonance point $\eta_0 \approx 7.34$ by Lemma \ref{lem:anomalous_res}.
\\
\noindent
For equation \eqref{family_2_subfamily_1_case_3_b}, we have the equation
$$\Phi_b(\eta_1)=-\omega(2\eta_1-r(\eta_1))+2\omega(\eta_1)+\omega(r(\eta_1))=0$$
where we have the solution if and only if $\vert \eta_1\vert=\sqrt{3}$. By symmetry, without loss of generality we consider only $\eta_1=\sqrt{3}$, which gives the discrete resonant point $(\eta_1,\eta_2,\eta_3,\eta_4;\xi)=(\sqrt{3},\sqrt{3},\sqrt{3},-\sqrt{3};\sqrt{3})$ defined by Lemma \ref{lem:sqrt3_res}, which lies on the resonant manifolds.
\\
\noindent
For equation \eqref{family_2_subfamily_1_case_3_c},  we have the stationary point at $(\eta_1,\eta_2,\eta_3,\eta_4)=(\frac{\xi}{3},\frac{\xi}{3},\frac{\xi}{3},-r(\frac{\xi}{3}))$, and this only hold when $\vert \xi\vert >3$, this implies the equation
$$\Phi_c (\xi)=\frac{8\xi^3}{9(1+\frac{\xi^2}{9})(1+\xi^2)}+\omega(r(\frac{\xi}{3}))> 0\quad \forall \xi> 3.$$
Therefore, it degenerates to the resonance point $(\eta_1,\eta_2,\eta_3,\eta_4;\xi)=(0,0,0,0;0)$  by Lemma \ref{lem:trivial_origin}, which lies on the resonant manifolds.
\\
\noindent
For equation \eqref{family_2_subfamily_1_case_3_d}, we have the equation
$$\Phi_d(\eta_1)=-\omega(3\eta_1-2r(\eta_1))+3\omega(\eta_1)-2\omega(r(\eta_1)).$$
where we have the solution if and only if $\vert \eta_1\vert=\sqrt{3}$. By symmetry, without loss of generality we consider only $\eta_1=\sqrt{3}$, it gives the discrete resonant point $(\eta_1,\eta_2,\eta_3,\eta_4;\xi)=(\sqrt{3},\sqrt{3},\sqrt{3},-\sqrt{3};\sqrt{3})$ by Lemma \ref{lem:sqrt3_res}, which lies on the resonant manifolds.
\\\\
\noindent
{\it \underline{Subfamily $\eqref{eqn:fam_2_f}$}}
Using that $\xi-\eta_1-\eta_2-\eta_3-\eta_4=\xi-\eta_1+r(\eta_1)$ we are reduced to solve
$$\eta_1=\pm \left(\xi-\eta_1+r(\eta_1)\right)\qquad \text{or}\qquad \eta_1=\pm r(\xi-\eta_1+r(\eta_1)).$$
These two sets of identities can be reduced to four cases:

\begin{subequations}
    \begin{align}
         \xi&=2\eta_1-r(\eta_1) \tag{a} \label{family_2_subfamily_1_case_4_a}
  \\
  \xi&=-r(\eta_1) \tag{b} \label{family_2_subfamily_1_case_4_b}
 \\
  \xi&=\eta_1 \tag{c} \label{family_2_subfamily_1_case_4_c}
 \\
   \xi&= \eta_1-2r(\eta_1). \tag{d} \label{family_2_subfamily_1_case_4_d}
    \end{align} 
\end{subequations}
For equation \eqref{family_2_subfamily_1_case_4_a}, we have the equation
$$\Phi_a(\eta_1)=-\omega(2\eta_1-r(\eta_1))+2\omega(\eta_1)-\omega(r(\eta_1))=0$$
where we have the solution if and only if $\vert \eta_1\vert=\sqrt{3}$. By symmetry, without loss of generality we consider only $\eta_1=\sqrt{3}$, it gives the discrete resonant point $(\eta_1,\eta_2,\eta_3,\eta_4;\xi)=(\sqrt{3},\sqrt{3},-\sqrt{3},-\sqrt{3};-\sqrt{3})$  by Lemma \ref{lem:sqrt3_res}, which lies on the resonant manifolds.
\\
\noindent
For equation \eqref{family_2_subfamily_1_case_4_b}, we have the equation
$$\Phi_b(\eta_1)=-\omega(-r(\eta_1))+\omega(\eta_1)-\omega(r(\eta_1))+\omega(-\eta_1)\equiv0.$$
This implies we have space--time resonance on $\xi=-r(\eta_1)$. By Lemma \ref{lem:resonance_manifold}, this is precisely the resonant curve $\Gamma$.
\\
\noindent
For equation \eqref{family_2_subfamily_1_case_4_c}, we have the equation 
$$\Phi_c(\eta_1)=-\omega(\eta_1)+\omega(\eta_1)-\omega(r(\eta_1))+\omega(r(\eta_1))\equiv 0.$$
This implies that we have  space--time resonance on $\xi=\eta_1$. By Lemma \ref{lem:resonance_manifold}, this is precisely the resonant curve $L$.
\\
\noindent
For equation \eqref{family_2_subfamily_1_case_4_d}, we have the equation
$$\Phi_d(\eta_1)=-\omega(\eta_1-2r(\eta_1))+\omega(\eta_1)-2\omega(r(\eta_1))$$
where we have the solution if and only if $\vert \eta_1\vert=\sqrt{3}$. By symmetry, without loss of generality we consider only $\eta_1=\sqrt{3}$, it gives the discrete resonant point $(\eta_1,\eta_2,\eta_3,\eta_4;\xi)=(\sqrt{3},\sqrt{3},-\sqrt{3},-\sqrt{3};-\sqrt{3})$  by Lemma \ref{lem:sqrt3_res}, which lies on the resonant manifolds.
\\\\
\noindent
\textbf{Family 3}, \textbf{Family 4}, \textbf{Family 5} would be argued in the same way as \textbf{Family 2}.
\\\\
\noindent
\textbf{Family 6}.  In this case we must have 
 \begin{align}
 |\eta_{1}|=|\eta_{2}|=|r(\eta_3)|=|r(\eta_4)| =\begin{cases} 
  |\xi-\eta_{1}-\eta_{2}-\eta_{3}-\eta_4| 
  \\
   \quad\quad\quad \text{or}
  \\
\left|r\left(\xi-\eta_{1}-\eta_{2}-\eta_{3}-\eta_4\right)\right|
 \end{cases}
 \label{eqn:fam_3_main_eqn}.
 \end{align}
 We can break this expression into several subfamilies,
  \begin{subequations}
 \label{eqn:fam_3_subfam}
 \begin{align}
 \eta_{1}&=\phantom{-}\eta_{2}=\phantom{-}r(\eta_{3})=\phantom{-}r\left(\eta_4\right),\label{eqn:fam_3_a}
 \\
 -\eta_{1} &=\phantom{-} \eta_{2} =\phantom{-}r(\eta_{3})= \phantom{-}r\left(\eta_4\right),\label{eqn:fam_3_b}
 \\
 \eta_{1} &= -\eta_{2} =\phantom{-}r(\eta_{3})=\phantom{-}r\left(\eta_3\right),\label{eqn:fam_3_c}
 \\
  \eta_{1} &= \phantom{-}\eta_{2} =-r(\eta_{3})=\phantom{-}r\left(\eta_4\right),\label{eqn:fam_3_d}
 \\
 \eta_{1} &= \phantom{-}\eta_{2} =\phantom{-}r(\eta_{3}) =- r\left(\eta_4\right),\label{eqn:fam_3_e}
  \\
 -\eta_{1} &= -\eta_{2} =\phantom{-}r(\eta_{3}) = \phantom{-}r\left(\eta_4\right),\label{eqn:fam_3_f}
   \\
 -\eta_{1} &= \phantom{-} \eta_{2} =-r(\eta_{3}) =\phantom{-}r\left(\eta_4\right),\text{[one of the minus sign hit on $r(\eta)$].}\label{eqn:fam_3_g} 
    \\
 \phantom{-}\eta_{1} &= \phantom{-} \eta_{2} =-r(\eta_{3}) =-r\left(\eta_4\right).\label{eqn:fam_3_h} 
 \end{align}
 \end{subequations}
  
\noindent
{\it \underline{Subfamily $\eqref{eqn:fam_3_a}$}}.
Using that $\xi-\eta_1-\eta_2-\eta_3-\eta_4=\xi-2\eta_1-2r(\eta_1)$ we are reduce to solve
$$\eta_1=\pm\left(\xi-2\eta_1-2r(\eta_1)\right)\qquad \text{or}\qquad \eta_1=\pm r(\xi-2\eta_1-2r(\eta_1))$$
These two sets of identities can be reduced to four cases:
\begin{subequations}
    \begin{align}
         \xi&=3\eta_1+2r(\eta_1) \tag{a} \label{family_3_subfamily_1_a}
  \\
  \xi&=\eta_1+2r(\eta_1) \tag{b} \label{family_3_subfamily_1_b}
 \\
  \xi&=2\eta_1+3r(\eta_1) \tag{c} \label{family_3_subfamily_1_c}
 \\
   \xi&= 2\eta_1+r(\eta_1). \tag{d} \label{family_3_subfamily_1_d}
    \end{align} 
\end{subequations}
The equations \eqref{family_3_subfamily_1_a} to \eqref{family_3_subfamily_1_d} corresponding to the generalized phase function 
\begin{align*}
    \Phi_a(\eta_1)&=-\omega(2\eta_1+r(\eta_1))+2\omega(\eta_1)+\omega(r(\eta_1))=0\\
    \Phi_b(\eta_1)&=-\omega(3\eta_1+2r(\eta_1))+3\omega(\eta_1)+2\omega(r(\eta_1))=0\\
    \Phi_c(\eta_1)&=-\omega(2\eta_1+3r(\eta_1))+2\omega(\eta_1)+3\omega(r(\eta_1))=0\\
    \Phi_d(\eta_1)&=-\omega(2\eta_1+r(\eta_1))+2\omega(\eta_1)+\omega(r(\eta_1))=0,
\end{align*}
respectively. These equations have no solutions, which follows from Lemma \ref{lem:no_roots}. 
Hence, there is no time resonance.
\\\\
\noindent
{\it \underline{Subfamily $\eqref{eqn:fam_3_b}$}}.
Using that $\xi-\eta_1-\eta_2-\eta_3-\eta_4=\xi-2r(\eta_1)$ we are reduced to solve
$$\eta_1=\pm\left(\xi-2r(\eta_1)\right)\qquad \text{or}\qquad \eta_1=\pm r(\xi-2r(\eta_1))$$
These two sets of identities can be reduced to four cases:

\begin{subequations}
    \begin{align}
         \xi&=\eta_1+2r(\eta_1) \tag{a} \label{family_3_subfamily_2_a}
  \\
  \xi&=-\eta_1+2r(\eta_1) \tag{b} \label{family_3_subfamily_2_b}
 \\
  \xi&=3r(\eta_1) \tag{c} \label{family_3_subfamily_2_c}
 \\
   \xi&= r(\eta_1). \tag{d} \label{family_3_subfamily_2_d}
    \end{align} 
\end{subequations}
For equation \eqref{family_3_subfamily_2_a}, we have the equation
$$\Phi_a(\eta_1)=-\omega(\eta_1+2r(\eta_1))+\omega(\eta_1)+2\omega(r(\eta_1))=0.$$
This equation has no solutions, which follows from Lemma \ref{lem:no_roots}. 
Hence, there is no time resonance.\\
\noindent
For equation \eqref{family_3_subfamily_2_b}, we have the equation
$$\Phi_b(\eta_1)=-\omega(-\eta_1+2r(\eta_1))-\omega(\eta_1)+2\omega(r(\eta_1))=0$$
where we have the solution if and only if $\vert \eta_1\vert=\sqrt{3}$. By symmetry, without loss of generality we consider only $\eta_1=\sqrt{3}$, it gives the discrete resonant point $(\eta_1,\eta_2,\eta_3,\eta_4;\xi)=(-\sqrt{3},\sqrt{3},\sqrt{3},\sqrt{3};\sqrt{3})$  by Lemma \ref{lem:sqrt3_res}, which lies on the resonant manifolds.\\
\noindent
For equation \eqref{family_3_subfamily_2_c}, we have the equation
$$\Phi_c(\eta_1)=3\omega(r(\eta_1))-\omega(3(r(\eta_1)))=0.$$
This equation has no solutions, which follows from Lemma \ref{lem:no_roots}. Hence, there is no time resonance.
\\
\noindent
For equation \eqref{family_3_subfamily_2_d}, we have the equation
$$\Phi_d(\eta_1)\equiv 0.$$
This implies we have space--time resonance on $\xi=r(\eta_1)$. By Lemma \ref{lem:resonance_manifold}, this is precisely the resonant curve $\Gamma$.
\\\\
\noindent
{\it \underline{Subfamily $\eqref{eqn:fam_3_c}$}} to 
{\it \underline{Subfamily $\eqref{eqn:fam_3_e}$}} follows the same argument as {\it \underline{Subfamily $\eqref{eqn:fam_3_b}$}}. 
\\\\
\noindent
{\it \underline{Subfamily $\eqref{eqn:fam_3_f}$}}.
Using that $\xi-\eta_1-\eta_2-\eta_3-\eta_4=\xi+2\eta_1-2r(\eta_1)$ we are reduced to solve
$$\eta_1=\pm\left(\xi+2\eta_1-2r(\eta_1)\right)\qquad \text{or}\qquad \eta_1=\pm r(\xi+2\eta_1-2r(\eta_1)).$$
These two sets of identities can be reduced to four cases:
\begin{subequations}
    \begin{align}
         \xi&=-\eta_1+2r(\eta_1) \tag{a} \label{family_3_subfamily_3_a}
  \\
  \xi&=-3\eta_1+2r(\eta_1) \tag{b} \label{family_3_subfamily_3_b}
 \\
  \xi&=-2\eta_1+3r(\eta_1) \tag{c} \label{family_3_subfamily_3_c}
 \\
   \xi&= -2\eta_1+r(\eta_1). \tag{d} \label{family_3_subfamily_3_d}
    \end{align} 
\end{subequations}
For equation \eqref{family_3_subfamily_3_a}, we have the equation $$\Phi_a(\eta_1)=-\omega(-\eta_1+2r(\eta_1))-\omega(\eta_1)+2\omega(r(\eta_1))=0$$
where we have the solution if and only if $\vert \eta_1\vert=\sqrt{3}$. By symmetry, without loss of generality we consider only $\eta_1=\sqrt{3}$, it gives the discrete resonant point $(\eta_1,\eta_2,\eta_3,\eta_4;\xi)=(-\sqrt{3},-\sqrt{3},\sqrt{3},\sqrt{3};\sqrt{3})$  by Lemma \ref{lem:sqrt3_res}, which lies on the resonant manifolds.\\
\noindent
For equation \eqref{family_3_subfamily_3_b}, we have the equation $$\Phi_b(\eta_1)=-\omega(-3\eta_1+2r(\eta_1))-3\omega(\eta_1)+2\omega(r(\eta_1))=0$$
where we have the solution if and only if $\vert \eta_1\vert=\sqrt{3}$. By symmetry, without loss of generality we consider only $\eta_1=\sqrt{3}$, it gives the discrete resonant point $(\eta_1,\eta_2,\eta_3,\eta_4;\xi)=(-\sqrt{3},-\sqrt{3},\sqrt{3},\sqrt{3};-\sqrt{3})$  by Lemma \ref{lem:sqrt3_res}, which lies on the resonant manifolds.\\
\noindent
For equation \eqref{family_3_subfamily_3_c}, we have the equation $$\Phi_c(\eta_1)=-\omega(-2\eta_1+3r(\eta_1))-2\omega(\eta_1)+3\omega(r(\eta_1))=0$$
where we have the solution if and only if $\vert \eta_1\vert=\sqrt{3}$. By symmetry, without loss of generality we consider only $\eta_1=\sqrt{3}$, it gives the discrete resonant point $(\eta_1,\eta_2,\eta_3,\eta_4;\xi)=(-\sqrt{3},-\sqrt{3},\sqrt{3},\sqrt{3};\sqrt{3})$  by Lemma \ref{lem:sqrt3_res}, which lies on the resonant manifolds.\\
\noindent
For equation \eqref{family_3_subfamily_3_d}, we have the equation $$\Phi_d(\eta_1)=-\omega(-2\eta_1+r(\eta_1))-2\omega(\eta_1)+\omega(r(\eta_1))=0$$
where we have the solution if and only if $\vert \eta_1\vert=\sqrt{3}$. By symmetry, without loss of generality we consider only $\eta_1=\sqrt{3}$, it gives the discrete resonant point $(\eta_1,\eta_2,\eta_3,\eta_4;\xi)=(-\sqrt{3},-\sqrt{3},\sqrt{3},\sqrt{3};-\sqrt{3})$  by Lemma \ref{lem:sqrt3_res}, which lies on the resonant manifolds.
\\\\
\noindent
{\it \underline{Subfamily $\eqref{eqn:fam_3_g}$}}.
It is reduced to solve
$$\eta_1=\pm \xi\qquad \text{or}\qquad \eta_1=\pm r(\xi)$$
These two sets of identities can be reduced to four cases:
\begin{subequations}
    \begin{align}
         \xi&=\eta_1 \tag{a} \label{family_3_subfamily_4_a}
  \\
  \xi&=-\eta_1 \tag{b} \label{family_3_subfamily_4_b}
 \\
  \xi&=r(\eta_1) \tag{c} \label{family_3_subfamily_4_c}
 \\
   \xi&= -r(\eta_1). \tag{d} \label{family_3_subfamily_4_d}
    \end{align} 
\end{subequations}
The cases \eqref{family_3_subfamily_4_a} and \eqref{family_3_subfamily_4_b} correspond to the resonant line $L$, while \eqref{family_3_subfamily_4_c} and \eqref{family_3_subfamily_4_d} correspond to the resonant curve $\Gamma$. Moreover, \eqref{family_3_subfamily_4_b} is the symmetric counterpart of \eqref{family_3_subfamily_4_a}, and \eqref{family_3_subfamily_4_d} is the symmetric counterpart of \eqref{family_3_subfamily_4_c}.
\\\\
\noindent
{\it \underline{Subfamily $\eqref{eqn:fam_3_h}$}}.
Using that $\xi-\eta_1-\eta_2-\eta_3-\eta_4=\xi-2\eta_1+2r(\eta_1)$ we are reduced to solve
$$\eta_1=\pm \left(\xi-2\eta_1+2r(\eta_1)\right)\qquad \text{or}\qquad \eta_1=\pm r\left((\xi-2\eta_1+2r(\eta_1)\right)$$
These two sets of identities can be reduced to four cases:
\begin{subequations}
    \begin{align}
         \xi&=3\eta_1-2r(\eta_1), \tag{a} \label{family_3_subfamily_5_a}
  \\
  \xi&=\eta_1-2r(\eta_1), \tag{b} \label{family_3_subfamily_5_b}
 \\
  \xi&=2\eta_1-r(\eta_1), \tag{c} \label{family_3_subfamily_5_c}
 \\
   \xi&= 2\eta_1-3r(\eta_1). \tag{d} \label{family_3_subfamily_5_d}
    \end{align} 
\end{subequations}
For the four equations we have the phase function
\begin{align*}
    \Phi_a(\eta_1)&=3\,\omega(\eta_1)\;-\;2\,\omega(r(\eta_1))\;-\;\omega(3\eta_1-2r(\eta_1)),\\
    \Phi_b(\eta_1)&=\omega(\eta_1)\;-\;2\,\omega(r(\eta_1))\;-\;\omega(\eta_1-2r(\eta_1)),\\
    \Phi_c(\eta_1)&=2\,\omega(\eta_1)\;-\;\omega(r(\eta_1))\;-\;\omega(2\eta_1-r(\eta_1)),\\
    \Phi_d(\eta_1)&=2\,\omega(\eta_1)\;-\;3\,\omega(r(\eta_1))\;-\;\omega(2\eta_1-3r(\eta_1)).
\end{align*}
Each of the four phase equations yields admits a solution \textit{if and only if} $|\eta_1|=\sqrt3$  by Lemma \ref{lem:sqrt3_res}. Taking $\eta_1=\sqrt3$, \eqref{family_3_subfamily_5_a} and \eqref{family_3_subfamily_5_c} yield the discrete space--time resonant point $(\sqrt3,\sqrt3,-\sqrt3,-\sqrt3;\ \sqrt3)$, whereas \eqref{family_3_subfamily_5_b} and \eqref{family_3_subfamily_5_d} yield $(\sqrt3,\sqrt3,-\sqrt3,-\sqrt3;\ -\sqrt3)$.

%\newpage
\section{Energy bounds and weighted estimates}

We will establish the desired scattering result via a bootstrap argument.

Fix $s\geq 100$ and choose $0<\eps_0\ll 1$ sufficiently small.
Assume the initial datum (posed at $t=1$) satisfies
\begin{equation}\label{eq:X1-small}
\snorm{u_0}_{X_1}:= \snorm{\widehat{u_0}}_{L^\infty_\xi}+\snorm{x u_0}_{L^2_x}+\snorm{u_0}_{H^s_x} \leq \eps_0.
\end{equation}
For $1\leq T<\infty$, set
$$X_T:=\Big\{u:\ u\in C([1,T];X_T)\ \text{and}\ \snorm{u}_{X_T}<\infty\Big\},$$
where the bootstrap norm is
\begin{equation}\label{eq:XT-norm}
\snorm{u}_{X_T}:=\sup_{t\in[1,T]}\Big(\snorm{\widehat f(t)}_{L^\infty_\xi}+\snorm{x f(t)}_{L^2_x}+ \snorm{f(t)}_{H^s_x}\Big).
\end{equation}
We assume the bootstrap bound
\begin{equation}\label{eq:bootstrap_assump_XT}
\snorm{u}_{X_T}\le \eps_1.
\end{equation}
In this section, we will prove the following proposition.
\begin{proposition}\label{prop:bootstrap}
    Suppose $\eps_1\in (\eps_0,1)$. Given the bootstrap assumption \eqref{eq:bootstrap_assump_XT}, we have
    \begin{equation}\label{aprioriweakbound}
        \snorm{u}_{X_T}\leq \eps_0+\eps_1^5.
    \end{equation}
\end{proposition}
\begin{proof}
This follows from Proposition \ref{prop:Hs_f_bootstrap}, Proposition \ref{prop:weightnorm} and Proposition \ref{prop:pointwisebound}.
\end{proof}
\noindent{\bf Proof of Theorem \ref{thm:main}:}
With Proposition \ref{prop:bootstrap}, we can prove our main result.
\begin{proof}[Proof of Theorem \ref{thm:main}]

Take $\eps_0$ small enough and take $\eps_1=\eps_0^{\frac{2}{5}}$.

We make the {\it a priori} assumption \eqref{eq:bootstrap_assump_XT} which can be ensured about the standard local well-posedness theory.   Proposition \ref{prop:bootstrap} implies that
\begin{equation}
      \snorm{u}_{X_T}\leq \eps_0+\eps_1^5\leq \frac{\eps_1}{2}
\end{equation}improving on \eqref{eq:bootstrap_assump_XT}, so that a standard continuation argument gives us a global solution which is bounded in the $X_\infty$ norm. 

Finally, the scattering is shown in Proposition \ref{prop:scattering}.
\end{proof}

In this remaining of this section, we study the proof of Proposition \ref{prop:bootstrap}.  

We first record the pointwise decay of the solution following from Lemma \ref{lem:poitwise}.
\begin{corollary}
 Under the bootstrap assumption, \eqref{eq:bootstrap_assump_XT},     one has
\begin{equation}\label{eq:udecay}
    \|u\|_{L^\infty_x} \lesssim t^{-\frac{1}{3}} \|u\|_{X_T}\lesssim t^{-\frac{1}{3}} \varepsilon_1.
\end{equation}
\end{corollary}
\subsection{Energy estimates}
First of all, we show the bounds for the $H^s$ norms.

\begin{proposition}\label{prop:Hs_f_bootstrap}
Let $T\ge 1$. Assume that $f(t)$ satisfies the Duhamel formula \eqref{eq:f_duhamel}. Given the bootstrap assumption \eqref{eq:bootstrap_assump_XT}, then
$$\snorm{f(t)}_{H^s_x}\lesssim \eps_0+\eps_1^5,\qquad \forall t\in[1,T].$$
\end{proposition}
\begin{proof}
To begin, we use \eqref{eq:f_duhamel} and smallness of the initial data to write:
\begin{equation}
    \snorm{f}_{H^s_x}\lesssim \eps_0+\int^t_1 d \tau\snorm{\langle \xi\rangle^s \int d\eta_{1234}e^{-i\tau\phi}\prod^5_{j=1}\hat{f}(\tau,\eta_j)}_{L^2_\xi}
\end{equation}
By applying Plancherel’s theorem, \eqref{eq:bootstrap_assump_XT}, \eqref{eq:udecay}, the chain rule, and the Moser-type estimate, it follows that
\begin{align*}
    \snorm{f}_{H^s_x}&\lesssim \eps_0+\int^t_1 d \tau\snorm{u^5}_{H^s_x}\lesssim  \varepsilon_0+\int^t_1 d \tau\snorm{u}^4_{L^\infty_x}\snorm{u}_{H^s_x}\\
    &\lesssim \varepsilon_0+\eps_1^5\int^t_1 \tau^{-\frac{4}{3}}\,d\tau \lesssim \varepsilon_0+\eps_1^5
\end{align*}as desired.
\end{proof}

 \subsection{Weighted $L^2_x$ and $L^{\infty}_{\xi}$ bounds near resonances}
 \begin{proposition}\label{prop:weightnorm}
     Fix time $T\geq 1$, and further assume that the bootstrap assumption \eqref{aprioriweakbound} holds in the time interval $[1,T]$. Then, if the initial data satisfies 
    $\snorm{xf(1)}_{L^2_x}\lesssim \eps_0$, we have for all $t\in [1,T]$,
    \begin{equation}
        \snorm{xf(t)}_{L^2_x}\lesssim \eps_0+\eps_1^5.
    \end{equation}
 \end{proposition}
    \begin{proof}
To estimate weighted norms of our solution, we start by differentiating both sides of \eqref{eq:f_duhamel} with respect to $\xi$:
\begin{align*}
\partial_{\xi} \widehat{f}(t,\xi) &= \partial_{\xi} \widehat{f}(1,\xi) -i\left(2\pi\right)^{-2}\omega(\xi)
    \int_{1}^{t} d s \int d\eta_{1234}\, e^{-i\tau\phi} \left\{(-i\tau) \ \partial_{\xi}\phi \prod^5_{j=1} \hat{f}(\tau,\eta_{j}) + \prod^{4}_{j=1}
    \hat{f}(\tau,\eta_{j})\partial_{\xi}\hat{f}\left(\tau,\eta_{5}\right)
    \right\}
    \\
    \phantom{=} &-i\left(2\pi\right)^{-2}\omega'(\xi)
    \int_{1}^{t} d s \int  d\eta_{1234}\, e^{-i\tau\phi} \prod^5_{j=1}\hat{f}(\tau,\eta_{j})
\end{align*}
where we recall that $\eta_5=\xi-\sum^4_{j=1}\eta_j$.
Writing for short, we define
\begin{subequations}
    \begin{align}
    I_{\text{1}} &:=  \omega(\xi)
    \int_{1}^{t} d \tau \int  d\eta_{1234} \ e^{-i\tau\phi} \prod^5_{j=1} \hat{f}(\tau,\eta_{j})\ \partial_{\xi}\hat{f}\left(\tau,\eta_5\right)
    \\
     I_{\text{2}} &:=  \omega'(\xi)
    \int_{1}^{t} d \tau \int  d\eta_{1234}\ e^{-i\tau\phi}\prod^5_{j=1}  \hat{f}(\tau,\eta_{j})
    \\
    I &:= \omega(\xi)
    \int_{1}^{t}  d \tau \ \tau \int d\eta_{1234} \ e^{-i\tau\phi} \partial_{\xi}\phi \prod^5_{j=1} \hat{f}(\tau,\eta_{j})
    \end{align}
\end{subequations}
    so that
    \begin{equation}
        \partial_\xi \hat{f}(t,\xi)=\partial_\xi \hat{f}(1,\xi)-(2\pi)^{-2}\left[iI_1+iI_2+I\right].
    \end{equation}
    Specifically, we have
    \begin{equation}
        \snorm{xf(t,x)}_{L^2_x}\lesssim \eps_0+\snorm{I_1(t,\xi)}_{L^2_\xi}+\snorm{I_2(t,\xi)}_{L^2_\xi}+\snorm{I(t,\xi)}_{L^2_\xi}.
    \end{equation}
Apparently, the factor $\tau$ in the integral in $I$ makes the analysis of $I$ more involved.  We begin by estimating the simpler terms.

    For bounding $I_1$, we follow the method in the proof of Proposition \ref{prop:Hs_f_bootstrap} and it follows that
    \begin{align*}
        \|I_1(t,\xi)\|_{L^2_\xi}&\lesssim \int^t_1 d \tau \cdot \|u^4 \cdot (e^{i\tau \omega(\partial_x)}xf)(\tau,x)\|_{L^2_x}\\
        &\lesssim \eps_1^5 \int^t_1 \tau^{-\frac{4}{3}}\,d \tau \lesssim \eps_1^5,
    \end{align*}
     Using Plancherel's theorem gives $\snorm{I_2}_{L^2_\xi}\lesssim \eps_1^5$.

     For term $I$, we close the proof  using  the Littlewood-Paley decomposition, and the proof is presented in the later parts, see the study of Lemma \ref{lem:dyadicLJ}.
     \end{proof}
     \begin{proposition}\label{prop:pointwisebound}
     Fix time $T\geq 1$, and further assume that the bootstrap assumption \eqref{aprioriweakbound} holds in the time interval $[1,T]$. Then, if the initial data satisfies 
    $\snorm{\hat{f}(1)}_{L^\infty_\xi}\lesssim \eps_0$, we have for all $t\in [1,T]$,
    \begin{equation}
        \snorm{\hat{f}(t)}_{L^\infty_\xi}\lesssim \eps_0+\eps_1^5.
    \end{equation}
 \end{proposition}
\begin{proof}
     From the Duhamel expansion, \eqref{eq:f_duhamel},  it suffices to show that the integral
\begin{equation}
    \label{eqn:J_defn}
    J(\xi) = \omega(\xi)\int_{1}^{t} d \tau \int d\eta_{1234}\ e^{-i\tau\phi} \prod^5_{j=1}  \widehat{f}\left(\tau, \eta_j\right) 
\end{equation}
has the bound
\begin{equation}
    \label{eqn:Linfty_bnd_wts}   \snorm{J}_{L^{\infty}_{\xi}}\lesssim \eps_1^5
\end{equation}
under the assumption \eqref{eq:bootstrap_assump_XT}.    Again as the estimate for $I$, we will present the proof of the bound in later parts of this section, see the study of Lemma \ref{lem:dyadicLJ}.
\end{proof}

\noindent
We now performing a {\it dyadic decomposition} of $[1,t]$: 
 if $m_{*}$ is the smallest integer satisfying $2^{m_{*}}\geq t$, then we write
$$
[1,t] = \bigcup_{m=0}^{m_{*}} \left\{\tau \sim 2^{m}\right\}\cap [1,t]. 
$$
Note that each dyadic piece has length $\sim 2^{m}$, and there are a total of $\mathcal{O}\left(\operatorname{log} t\right)$ dyadic pieces. 

Then, we fix one $m=0,...,m_{*}$ and define $t_1, t_2$ by $[t_1, t_2] = \left\{\tau \sim 2^{m}\right\}\cap [1,t]$ and
\begin{subequations}
\begin{align}
\label{eq:I_dfn}
I_{m} &:= \omega(\xi)
    \int_{t_1}^{t_2}  d \tau \ \tau \int d\eta_{1234}\ e^{-i\tau\phi}\  \partial_{\xi}\phi \prod^5_{j=1} \hat{f}(\tau,\eta_{j})
    \\ \label{eq:J_dfn}
    J_{m} &:=\omega(\xi)\int_{t_1}^{t_2} d \tau \int d\eta_{1234}\ e^{-i\tau\phi} \prod^5_{j=1}  \hat{f}\left(\tau,\eta_j\right).
\end{align}
\end{subequations}
\begin{lemma}\label{lem:dyadicLJ}
 Under the bootstrap hypothesis \eqref{eq:bootstrap_assump_XT}, we have
\begin{subequations}
    \begin{align}
        \label{eqn:claimed_bound}
\sum_{m=0}^{m_{*}}\snorm{I_{m} }_{L^2_{\xi}} \lesssim \eps_1^5,\qquad
\sum_{m=0}^{m_{*}}\snorm{J_{m} }_{L^\infty_{\xi}} \lesssim \eps_1^5 . 
    \end{align}
\end{subequations}
\end{lemma}

This is accomplished by localizing the integrands of $I_{m}$ and $J_m$ in $\eta_1\eta_2\eta_3\eta_4\xi$-space.
\subsubsection{Anomalous Resonance}
We first consider the contribution associated with the anomalous space--time resonant point $(\eta_0,\eta_0,\eta_0,\eta_0;\xi_0)$  defined in \eqref{eq:def_xi_eta_nod}. 
For each dyadic time index $m$, we introduce a localization scale
$k_{*}$ defined by 
\begin{equation}\label{eq:k_star_dfn_delta_1}
2^{k_{*}} \lesssim 2^{m(\delta_1-\frac12)}
\end{equation}
where $\delta_1$ is a small constant to be chosen later. Throughout the computations around  the anomalous resonance, we decompose the integrands $I_m$ and $J_m$ according to the relative sizes of $\xi-\xi_0$ and $\eta_j-\eta_0$, using a Littlewood-Paley frequency localization.
\begin{lemma}\label{lem:anomalous_resonance_general_taylor}
    For some $\eta_{\mathfrak{j}}\approx \eta_0$, $\xi_\theta\approx \xi_0$ and $\eta_{\theta}\approx -r(\eta_0)$. It follows that, we have the Taylor expansion
\begin{equation}\label{eq:anomalous_resonance_general_taylor}
    \phi= \left[\omega'\left(r(\eta_0)\right)-\omega'(\xi_0)\right]\eta_5-\frac{1}{2}\omega''(\xi_\theta)(\xi-\xi_0)^2+\frac{1}{2}\omega''(\eta_\theta)+\frac{1}{2}\sum^{4}_{j=1}\omega''(\eta_{\mathfrak{j}})(\eta_j-\eta_0)^2.
\end{equation}
\end{lemma}
\begin{proof}
    The proof follows from \cite[p. 123]{MorganThesis}.
\end{proof}
\begin{lemma}
\label{prop:anomalous_resonance_near_resonance}
In the frequency region
\begin{equation}\label{eqn:near_res_region}
    \Omega_1:= \left\{ (\eta_1, \eta_2, \eta_3, \eta_4; \xi) : |\eta_j - \eta_0| \le 2^{k_*} \ \forall j=1,2,3,4, \text{ and } |\xi - \xi_0| \le 2^4 2^{2k_*} \right\},
\end{equation}
where $k_*$  is given in \eqref{eq:k_star_dfn_delta_1}. Under the bootstrap assumption \eqref{eq:bootstrap_assump_XT}, the contribution of $\Omega_1$ satisfies
\begin{align}
\snorm{\omega(\xi)\int_{t_1}^{t_2}d\tau \ \tau\ \varphi_{2k_{*}+4}\left(\xi-\xi_0\right)e^{-i\phi \tau} \partial_{\xi}\phi \prod^4_{j=1} \widehat{f}_{\lesssim 2^{k_{*}}} \left(\tau,\eta_{j}\right)  
    \widehat{f}\left(\tau,\eta_{5}\right)}_{L^2_\xi} &\lesssim \eps_1^5, \\
\snorm{\omega(\xi)\int_{t_1}^{t_2}d\tau \ \varphi_{2k_{*}+4}\left(\xi-\xi_0\right)e^{-i\phi \tau}  \prod^4_{j=1} \widehat{f}_{\lesssim 2^{k_{*}}} \left(\tau,\eta_{j}\right)  
    \widehat{f}\left(\tau,\eta_{5}\right)}_{L^\infty_\xi} &\lesssim \eps_1^5.
\end{align}
\end{lemma}
\begin{proof}
 We first localize to the region
\begin{equation*}
\left|\eta_{j}-\eta_0\right| \leq 2^{k_{*}} \quad \forall \ j=1,2,3,4, \quad \left|\xi-\xi_0\right| \leq 2^{4}2^{2k_{*}}.
\end{equation*} 
by smooth cutoff functions. Accordingly, it suffices to bound
\begin{equation}
         \omega(\xi)\int_{t_1}^{t_2}d\tau \ \tau\ \varphi_{2k_{*}+4}\left(\xi-\xi_0\right)e^{-i\phi \tau} \partial_{\xi}\phi \prod^4_{j=1} \widehat{f}_{\lesssim 2^{k_{*}}} \left(\tau,\eta_{j}\right)  
    \widehat{f}\left(\tau,\eta_{5}\right)=:I_{mk_{*}}.
\end{equation}
In this case, neither oscillations nor curvature of the phase can be exploited. Instead, the smallness of the contribution comes solely from the size of the Fourier integration domain.
%More precisely, although the integrand does not exhibit additional pointwise decay, the region of integration shrinks rapidly as $m$ increases. 

%This intuition is standard in the space--time resonance framework; see, for instance, \cite{STR}.

Using the boundedness of $\omega(\xi)$ and $\partial_\xi \phi$, we estimate
\begin{align*}
\snorm{I_{m k_{*}}}_{L^2_\xi}
&\lesssim \int_{t_1}^{t_2} \tau\, \snorm{\varphi_{2k_{*}+4}(\xi-\xi_0) \int e^{-i\phi\tau}\partial_\xi\phi \prod_{j=1}^4 \widehat f_{\lesssim 2^{k_{*}}}(\tau,\eta_j) \,\widehat f(\tau,\eta_5)\,d\eta_{1234}}_{L^2_\xi} \, d\tau \\
&=: \int_{t_1}^{t_2} \tau\, \snorm{H_{k_*}(\tau,\xi)}_{L^2_\xi}\, d\tau.
\end{align*}
By Hölder’s inequality in $\xi$, we further obtain from the measure estimate
\begin{align*}
\snorm{H_{k_*}(\tau,\xi)}_{L^2_\xi}
&\lesssim \snorm{H_{k_*}(\tau,\xi)}_{L^\infty_\xi}
\cdot \sqrt{\operatorname{meas.}\big(|\xi-\xi_0|\lesssim 2^{2k_{*}}\big)}.
\end{align*}
Taking the supremum in $\xi$ and using the localization $|\eta_j-\eta_0|\leq 2^{k_*}$ for all $j=1,2,3,4$, we obtain
\begin{align*}
\snorm{H_{k_{*}}(\tau,\xi)}_{L^\infty_\xi}
&\lesssim \int_{|\eta_j-\eta_0|\le 2^{k_{*}}}
\prod_{j=1}^4 \left|\widehat f_{\lesssim 2^{k_{*}}}(\tau,\eta_j)\right|\cdot \left|\widehat f(\tau,\eta_5)\right|
\, d\eta_{1234} \\
&\lesssim \eps_1^5 2^{4k_{*}}.
\end{align*}
Combining the above estimates and integrating in time yields 
$$ \snorm{I_{m k_{*}}}_{L^2_\xi} \lesssim \eps_1^5\,2^{2m}\cdot 2^{5k_{*}} \lesssim \eps_1^5\, 2^{m(5\delta_1-\frac{1}{2})},$$
where we used $2^{k_{*}}\sim 2^{m(\delta_1-\frac12)}$. Since $(5\delta_1-\frac{1}{2})<0$, the series in $m$ is summable, and hence
$$ \sum_m \snorm{I_{mk_{*}}}_{L^2_\xi} \lesssim \varepsilon_1^5.$$
The same argument applied to the $L^{\infty}_{\xi}$, where writing for short we define
\begin{equation*}
    \omega(\xi)\int_{t_1}^{t_2}d\tau \ \varphi_{2k_{*}+4}\left(\xi-\xi_0\right)e^{-i\phi \tau}  \prod^4_{j=1} \widehat{f}_{\lesssim 2^{k_{*}}} \left(\tau,\eta_{j}\right)  
    \widehat{f}\left(\tau,\eta_{5}\right):=J_{m,k_*}.
\end{equation*}
This contribution yields
\begin{align*} 
\snorm{J_{m, k*}}_{L^\infty_\xi} &\lesssim \eps_1^5 \cdot 2^m \cdot \left(2^{m(\delta_1 - \frac{1}{2})}\right)^4= \eps_1^5 \cdot 2^{m(4\delta_1 - 1)} .
\end{align*}
Since $4\delta_1-1<0$ for $\delta_1$ sufficiently small,
this bound is also summable in $m$, and therefore
$$
\sum_{m}\snorm{J_{m,k_{*}}}_{L^{\infty}_{\xi}} \lesssim \eps_1^5.
$$
Here we finished the analysis for the region $\Omega_1$.
\end{proof}

Now we allow  the output  frequency $\xi$ to be away from $\xi_0$.
\begin{lemma}
\label{prop:anomalous_resonance_case_2}
In the frequency region
\begin{multline}\label{eqn:near_res_region2}
    \Omega_2:= \left\{ (\eta_1, \eta_2, \eta_3, \eta_4; \xi) : |\eta_j - \eta_0| \leq 2^{k_*} \ \forall j=1,2,3,4,
    \text{ and } |\xi - \xi_0| \sim 2^k,\; 2^k\in \left[2^{4}2^{k_*}, 2^{-10}\right] \right\},
\end{multline}
where $k_*$  is given in \eqref{eq:k_star_dfn_delta_1}. Under the bootstrap assumption \eqref{eq:bootstrap_assump_XT}, the contribution of $\Omega_2$ satisfies
\begin{align}
\sum_m \snorm{I_{m,k}}_{L^2_\xi} \lesssim \eps_1^5,\qquad
\sum_m \snorm{J_{m,k}}_{L^\infty_\xi} \lesssim \eps_1^5.
\end{align}
\end{lemma}
\begin{proof}
We localize using the dyadic frequency shells around $\xi_0$ and the smooth cutoff functions for each $\eta_j-\eta_0$ in this case.
Also, in the $\Omega_2$ we have $
\eta_j\approx \eta_0$, $\xi\approx\xi_0$ which implies that
$\eta_5\approx -r(\eta_0)$. Therefore, by Lemma \ref{lem:anomalous_resonance_general_taylor} the Taylor expansion is
\begin{multline}\label{eqn:case2:taylor_phi}
    \phi\approx \left[\omega'(r(\eta_0))-\omega'(\xi_0)\right](\xi-\xi_0)-\frac{1}{2}\omega''(\xi_\theta)(\xi_-\xi_0)+\frac{1}{2}\omega''(\eta_{\theta})\left(\eta_5+r(\eta_0)\right)^2+\frac{1}{2}\sum^4_{j=1}\omega''(\eta_{\theta})(\eta_{\mathfrak{j}})(\eta_j-\eta_0)^2
\end{multline}
that implies the lower bound
\begin{equation}\label{eq:anomalous_case2_phase_lower_bound}
    |\phi|\gtrsim 2^{k}.
\end{equation}
This lower bound allows us to integrate by parts in time. We therefore estimate
\begin{equation}
    I_{m,k} := \omega(\xi) \int_{t_1}^{t_2}d\tau \; \tau\; \psi_{k}\left(\xi-\xi_0\right) \int d\eta_{1234} \ e^{-i\phi \tau} \partial_{\xi}\phi
    \prod^4_{j=1}\widehat{f}_{\lesssim 2^{k_{*}}} \left(\tau,\eta_{j}\right)
    \widehat{f}\left(\tau,\eta_{5}\right). 
\end{equation}
Using \eqref{eq:anomalous_case2_phase_lower_bound} integrate by parts in $\tau$ yields
\begin{align*}
    I_{m,k} &= \omega(\xi) \psi_{k}(\xi-\xi_0) \int_{t_1}^{t_2} d \tau \int d\eta_{1234}  \ e^{-i\tau\phi} \frac{\partial_{\xi}\phi}{\phi}  \tau\  \partial_{\tau}\left( \prod^4_{j=1}\widehat{f}_{\lesssim 2^{k_{*}}} \left(\tau,\eta_{j}\right)
    \widehat{f}\left(\tau,\eta_{5}\right)  \right) 
    \\  
    &\phantom{=} + e^{-i\tau\phi} \ \frac{\partial_{\xi}\phi}{\phi} \  \prod^4_{j=1}\widehat{f}_{\lesssim 2^{k_{*}}} \left(\tau,\eta_{j}\right)
    \widehat{f}\left(\tau,\eta_{5}\right)
    \\
    &\phantom{=} + \left[\omega(\xi) \psi_{k}(\xi-\xi_0) \ \tau \int d\eta_{1234}  \ e^{-i\tau\phi} \frac{\partial_{\xi}\phi}{\phi} \  \prod^4_{j=1}\widehat{f}_{\lesssim 2^{k_{*}}} \left(\tau,\eta_{j}\right)
    \widehat{f}\left(\tau,\eta_{5}\right) \right]_{\tau=t_1}^{\tau=t_2}. 
\end{align*}
By the same strategy in  Lemma \ref{prop:anomalous_resonance_near_resonance}, with \eqref{eq:anomalous_case2_phase_lower_bound} and Lemma \ref{lem:time_derivative_bnd}, it follows that
\begin{equation*}
    \snorm{I_{mk}}_{L^2_\xi}\lesssim \eps_1^5 2^{-\frac{k}{2}}2^{m(4\delta_1-1)}.
\end{equation*}
Summing over all dyadic scales $m$ with $1 \le 2^m \le 2^{m_*}$ and
over frequencies satisfying $2^{2k_*} \lesssim 2^k \ll 1$,
$$
\sum_m \sum_k \snorm{I_{mk}}_{L^2_{\xi}} \lesssim \eps_1^5  t^{\left(3\delta_1 - \frac{1}{2}\right)}\lesssim \eps_1^5.
$$
We now turn to the $L^\infty_\xi$ estimate for $J_m$. Arguing as in the weighted case, and summing over frequencies satisfying $2^{2k_*} \lesssim 2^k \ll 1$ we find,
\begin{align*}
     \sum_k \snorm{J_{mk}}_{L^{\infty}_{\xi}} &\lesssim \eps_1^5 2^{m\left(2\delta_1-1\right)}. 
     \end{align*}
Choosing $\delta_1>0$ sufficiently small and summing over $m$ yields the desired bound.
\end{proof}

We now record the estimates which are adapted from \cite[lemma ~4.1]{Morgan} and \cite[lemma ~4.2]{Morgan}, and the proof applies up to elementary changes.
\begin{lemma}\label{lem:time_derivative_bnd}
Given \eqref{eq:bootstrap_assump_XT}, then, for any $t \in [1, T]$, we have the following estimate holds
\begin{equation}\label{eqn:time_derivative_bnd}
\snorm{\partial_{\tau}\widehat{f}_{\lesssim 2^{k_{*}}}}_{L^{\infty}_{\xi}} \lesssim \eps_1^5 2^{-m},\qquad
\snorm{\psi_{k}\left(\xi-\xi_0\right)\prod^4_{j=1}\varphi_{\lesssim 2^{k_{*}}}\left(\eta_j\right)\partial_{\tau}\hat{f}\left(\tau,\eta_5\right)}_{L^{\infty}_{\xi}} \lesssim \eps_1^5 2^{-m}. 
\end{equation}
\end{lemma}

\begin{lemma}\label{prop:anomalous_resonance_case_3}
Suppose, in the region $\Omega_3$, among the input frequencies $\eta_1,\eta_2,\eta_3,\eta_4$
there exists a dyadic separation of scales, in the sense that
$$2^{-10}\geq \max_{1\leq j\leq 4} |\eta_j-\eta_0|\gg\min_{1\le j\le 4} |\eta_j-\eta_0|\geq 2^{k_*}.$$
Then, for each dyadic output localization $|\xi-\xi_0|\leq 2^4 2^{2k_*}$ or
$|\xi-\xi_0|\sim 2^k$ with $2^{k}\in [2^{2k_*},2^{-10}]$, under the bootstrap assumptions, the contribution of satisfies
\begin{align}
\sum_m \snorm{I_{mk}}_{L^2_\xi} \lesssim \eps_1^5, \qquad
\sum_m \snorm{J_{mk}}_{L^\infty_\xi} \lesssim \eps_1^5.
\end{align}
\end{lemma}
\begin{proof}
We decompose the input frequencies dyadically so that $|\eta_j-\eta_0|\sim 2^{k_j}$. By symmetry, we can set  $k_1\ge k_2\ge k_3\ge k_4$. Without loss of generality, we only discuss the case $|\xi-\xi_0|\sim 2^k$ here. A direct application of mean value theorem shows that
\begin{equation}\label{eq:anomalous_resoannce_case_3_lower_bound}
    |\partial_{\eta_1-\eta_2}\phi|\gtrsim 2^{k_1}
\end{equation}
which allows us to integrate by parts in frequency. We integrate by parts in the $\eta_1-\eta_2$ direction and write
\begin{equation}\label{eqn:case3_IBP}
I_{m,k_1k_2k_3k_4,k}
=
i\bigl(\tilde I_1+\tilde I_2+\tilde I_3\bigr),
\end{equation}
where
\begin{align*}
\tilde I_1
&:=-\psi_k(\xi-\xi_0)\omega(\xi)\int_{t_1}^{t_2}\!\!\int e^{-i\tau\phi}\,a(\eta,\xi)\,(\partial_{\eta_1}\hat f_{k_1})\prod^5_{j=2}\hat f_{k_j}(\tau,\eta_j) \,d\eta_{1234}\,d\tau,\\
\tilde I_2
&:=\psi_k(\xi-\xi_0)\omega(\xi)
\int_{t_1}^{t_2}\!\!\int
e^{-i\tau\phi}\,
a(\eta,\xi)\,
\hat f_{k_1}(\partial_{\eta_2}\hat f_{k_2})\prod^5_{j=3}\hat f_{k_j}(\tau,\eta_j)
\,d\eta_{1234}\,d\tau,\\
\tilde I_3
&:=
-\psi_k(\xi-\xi_0)\omega(\xi)
\int_{t_1}^{t_2}\!\!\int
e^{-i\tau\phi}\,m(\eta,\xi) \prod_{j=1}^4 \hat f_{k_j}(\tau, \eta_j)\hat f(\tau,\eta_5)
\,d\eta_{1234}\,d\tau
\end{align*}
with the symbols
\begin{equation}\label{eqn:def_a_symbol}
a(\eta,\xi)
:=\frac{\partial_\xi\phi(\eta,\xi)}{\partial_{\eta_1-\eta_2}\phi(\eta,\xi)},\qquad
m(\eta,\xi)
:=\partial_{\eta_1-\eta_2} a(\eta,\xi).
\end{equation}
Also, notice that the bootstrap assumptions give
\begin{subequations}\label{eqn:derivative_hits_cutoff}
    \begin{align}
        \snorm{\partial_{\eta_{1}}\widehat{f}_{k_{1}}(\tau, \eta_{1})}_{L^2_{\eta_{1}}} \lesssim \eps_1\left(1 + 2^{\frac{k_1}{2}}\right),\qquad\snorm{\partial_{\eta_{2}}\widehat{f}_{k_2}(\tau, \eta_{2})}_{L^2_{\eta_{2}}} \lesssim \eps_1\left(1 + 2^{\frac{k_2}{2}}\right). 
    \end{align}
\end{subequations}
Moreover, on the region under consideration, we have $|\omega(\xi)|, |\partial_{\xi}\phi|\lesssim 1$. Hence, by \eqref{eq:anomalous_resoannce_case_3_lower_bound}, it follows that  as symbols, $a(\eta,\xi)$ and $m(\eta,\xi)$, we have the following bounds
    \begin{align}\label{eq:a_m_bound}
     |a(\eta,\xi) | \lesssim 2^{-k_{1}}, \quad |m(\eta,\xi)| \lesssim 2^{-2k_{1}}. 
    \end{align}
Now, we control $\tilde{I}_{1}, \tilde{I}_{2}$, and $\tilde{I}_{3}$ using our $L^2_{\xi}$-multilinear estimate. We now estimate the contribution of $\tilde I_1$.
The terms $\tilde I_2$ is treated similarly thanks to \eqref{eq:a_m_bound}, and thus will be omitted.

To estimate $\tilde I_1$, we observe that, in view of \eqref{eq:a_m_bound},
the associated symbol falls within the scope of \cite[Lemma~B.3]{Morgan}
with parameter $A=2^{-k_1}$ and the symbol
\begin{equation}\label{eq:anomalous_case_3_B_bound}
B(\eta,\xi) := a(\eta,\xi)\,\psi_k(\xi-\xi_0)\prod_{j=1}^4 \psi_{k_j}(\eta_j-\eta_0).
\end{equation}
We may therefore apply \cite[Proposition~B.1]{Morgan} with the symbol
$B(\eta,\xi)$ defined in \eqref{eq:anomalous_case_3_B_bound}. This yields
\begin{align*}
    \snorm{\tilde{I}_{1}}_{L^2_{\xi}} &\lesssim \snorm{\psi_{k}(\xi-\xi_0)\omega(\xi)}_{L^{\infty}_{\xi}} \int_{t_1}^{t_2} d \tau \snorm{\int e^{-i\tau\phi}\,a(\eta,\xi)\,(\partial_{\eta_1}\hat f_{k_1})\prod^5_{j=2}\hat f_{k_j}(\tau,\eta_j) \,d\eta_{1234}
    }_{L^2_{\xi}}
    \\
    &\lesssim  \eps_1^{5}\left(2^{-k_{1}}2^{-m} + 2^{-\frac{1}{2}k_{1}}2^{-m}\right).
\end{align*}
In the above estimate, we have used the worst case case dispersive decay, \eqref{eq:udecay}.

For the term $\tilde{I}_3$. The symbol estimate follows by \eqref{eq:a_m_bound}
\begin{equation*}
    B(\eta,\xi)=m(\eta,\xi)\,\psi_k(\xi-\xi_0)\prod_{j=1}^4 \psi_{k_j}(\eta_j-\eta_0).
\end{equation*}
with $A=2^{-2k_1}$. By a direct  $L^2_\xi$-multilinear estimate, it follows that
\begin{align*}
    \|\tilde{I}_3\|_{L^2_{\xi}} &\lesssim \snorm{\psi_{k}(\xi-\xi_0)\omega(\xi)}_{L^{\infty}_{\xi}} \int_{t_1}^{t_2} d \tau \snorm{\int e^{-i\tau\phi}\,m(\eta,\xi)\,\prod^5_{j=1}\hat f_{k_j}(\eta_j) \,d\eta_{1234}
    }_{L^2_{\xi}}
    \lesssim  \eps_1^{5}2^{-2\delta_1}.
\end{align*}
Summing over all dyadic scales $m$ with $1\leq 2^m\leq 2^{m_*}$ and frequencies $k, k_1,k_2,k_3,k_4$ satisfying $$2^{k_*}\lesssim 2^{k_j}\lesssim 1 \quad (j=1,\dots,4), \qquad 2^{2k_*}\lesssim 2^k\lesssim 1, \qquad k_1\gg k_2$$ gives
\begin{equation}
    \sum_{m}\;\sum_{\operatorname{freq.}}
        \snorm{ I_{m k_1 k_2 k_3 k_4} }_{L^2_\xi}
        \lesssim \eps_1^5\left[t^{-\frac{1}{2}}+1\right]
\end{equation}
which is the desired bound.

The estimate of $J_m$ is followed from the same method as above but with the  $L^\infty_\xi$-multilinear estimate. We only discuss the case for $|\xi-\xi_0|\sim 2^k$ and the rest of the case will be omitted.  Using integration by parts, it follows that
\begin{multline*}
    J_{mk_1k_2k_3k_4k}= i \psi_{k}(\xi-\xi_0) \ \omega(\xi) \int_{t_1}^{t_2}  d \tau \ \tau^{-1} \int d\eta_{1234}\ e^{-i\tau\phi}\;\tilde{m}(\eta,\xi) \ \prod^5_{j=1}\widehat{f}_{k_{j}}(\tau,\eta_j)
\\
    -i \psi_{k}(\xi-\xi_0) \ \omega(\xi) \int_{t_1}^{t_2}  d \tau \ \tau^{-1}\int d\eta_{1234}\ e^{-i\tau\phi}\; \tilde{a}(\eta,\xi) \ \partial_{\eta_1}\widehat{f}_{k_{1}}(\tau,\eta_{1})\ \prod^5_{j=2}\widehat{f}_{k_{j}}(\tau,\eta_j)
\\
    +i \psi_{k}(\xi-\xi_0) \ \omega(\xi) \int_ {t_1}^{t_2}  d \tau  \ \tau^{-1}\int d\eta_{1234} \ e^{-i\tau\phi}\; \tilde{a})(\eta,\xi) \ \hat{f}_{k_{1}}(\tau,\eta_{1})\ \left(\partial_{\eta_{2}}\hat{f}_{k_2}(\tau,\eta_{2})\right)\prod^5_{j=3}\widehat{f}_{k_{j}}(\tau,\eta_j)
\end{multline*}
with the symbol
\begin{equation}\label{eq:def_a_tilde_symbol}
\tilde{a}(\eta,\xi):=\frac{1}{ \partial_{\eta_1-\eta_2}\phi},\qquad
\tilde{m}(\eta,\xi):=\frac{\partial^2_{\eta_1-\eta_2}\phi}{ (\partial_{\eta_1-\eta_2}\phi)^2} .
\end{equation}
Therefore, the estimates of the symbols are
\begin{align}
     \left|\tilde{a}(\eta,\xi) \right| \lesssim 2^{-k_{1}}, \quad \left|\tilde{m}(\eta,\xi)\right| \lesssim 2^{-2k_{1}}. 
    \end{align}
By the multilinear estimates and \cite[Proposition B.6]{Morgan}, it follows that
\begin{align*}
    \snorm{J_{mk_1k_2k_3k_4k}}_{L^\infty_\xi}&\lesssim\int_{t_1}^{t_2} d\tau \, 2^{-2k_1} \tau^{-1} 
\left( \eps_1 2^{\frac{k_1}{2}} \right)
\left( \eps_1 \tau^{-\frac{1}{2}} \right)^3 (\eps_1)
+ 2^{-k_1} \tau^{-1} 
\left( \eps_1 \right)
\left( \eps_1 \tau^{-\frac{1}{2}} \right)^3 (\eps_1). \\
    &\lesssim \eps_1^5\left[2^{-\frac{3}{2}m}2^{-\frac{3}{2}k_1}+2^{-\frac{3}{2}m}2^{-k_1}\right].
\end{align*}
Summing over all dyadic scales $m$ with $1\leq 2^m\leq 2^{m_*}$ and frequencies $k, k_1,k_2,k_3,k_4$ satisfying $$2^{k_*}\lesssim 2^{k_j}\lesssim 1 \quad (j=1,\dots,4), \qquad 2^{2k_*}\lesssim 2^k\lesssim 1, \qquad k_1\gg k_2$$ gives
\begin{equation}
    \sum_m\sum_{\operatorname{freq.}}
        \snorm{J_{m k_1 k_2 k_3 k_4 k} }_{L^\infty_\xi}
        \lesssim \eps_1^5\sum_{m} \left[2^{m(-\frac{3}{4}-\frac{3}{2}\delta_1)}+2^{m(-\frac{3}{2}-\delta_1)}\right]\lesssim \eps_1^5
    \end{equation}
as desired.    
\end{proof}

\begin{lemma}
\label{prop:secondary_resonance}
Consider the frequency region where the input frequencies are clustered near the resonance (and larger than the cutoff scale), and the output frequency remains in a neighborhood of the resonance:
\begin{multline}\label{eq:Omega_4}
\Omega_4
:=\Bigl\{(\eta_1,\eta_2,\eta_3,\eta_4;\xi):\ 
|\eta_j-\eta_q|\sim 2^{k_j},\ 2^{k_*}\lesssim 2^{k_j}\lesssim 1,\ 
2^{k_1}\approx 2^{k_2}\approx 2^{k_3}\approx 2^{k_4},\\ 
|\xi-\xi_0|\lesssim 2^{2k_*}\ \text{or}\ |\xi-\xi_0|\sim 2^k,\ 2^{2k_*}\lesssim 2^k\lesssim 1
\Bigr\}.
\end{multline}
Under the bootstrap assumptions, the contribution of this region satisfies
\begin{align}
    \sum_{m} \|I_{m,k}\|_{L^2_\xi} \lesssim \eps_1^5, \qquad
    \sum_{m} \|J_{m,k}\|_{L^\infty_\xi} \lesssim \eps_1^5.
\end{align}
\end{lemma}
\begin{proof}
    In this case, no direction of the form $\eta_i-\eta_j (i\neq j)$ provides a non-degenerate derivative of the phase, and integration by parts in such directions is therefore unavailable. That is to say, $\partial_{\eta_i-\eta_j}\phi$ may be degenerate. Moreover, by the Taylor expansion in Lemma \ref{lem:anomalous_resonance_general_taylor} the relation $|\xi-\xi_0|\approx|\eta-\eta_0|^2$ may occur, in which case the expansion of the phase degenerates and no uniform lower bound on $|\phi|$ is available. Hence, the proof of Lemma \ref{prop:anomalous_resonance_case_2} does not work.

    We would like to perform integration by parts along certain frequency $\eta_j$.  But this is not direct.  For example,  we consider $\eta_1$. In order to perform integration by parts in $\eta_1$, one has to have a lower bound for $ \partial_{\eta_1}\phi=\omega'(\eta_1)-\omega'(\eta_5)$. \ A direct application of the mean value theorem to $\omega'(\eta_1)-\omega'(\eta_5)$ is obstructed by the possible degeneracy of $\omega''$ along the interval joining $\eta_1$ and $\eta_5$. Instead, we exploit the identity
$$\omega'(\eta_1)-\omega'(\eta_5)=\omega'(r(\eta_1))-\omega'(-\eta_5),$$ see \eqref{eq:omegarsym},
which relocates both arguments into a compact subinterval of $(0,\sqrt{3})$ where $\omega''$ does not vanish.
Applying the mean value theorem in this region yields
\begin{equation}\label{eq:eta1diff}
    |\partial_{\eta_1}\phi| = |\omega'(r(\eta_1)) - \omega'(-\eta_5)| \approx |\omega''(\theta)| \cdot |r(\eta_1) - (-\eta_5)| \gtrsim |r(\eta_1) + \eta_5|,
\end{equation}
where $\theta$ lies between $r(\eta_1)$ and $-\eta_5$.

We distinguish two cases according to the size of the quantities $|r(\eta_j)+\eta_5|$. Either there exists $j\in\{1,2,3,4\}$ and a dyadic scale $2^{\ell_j}$ with $2^{k_*}\ll 2^{\ell_j}\ll 1$ such that $|r(\eta_j)+\eta_5|\sim 2^{\ell_j}$, or all such quantities satisfy $|r(\eta_j)+\eta_5|\lesssim 2^{k_*}$.

For the first case, without loss of generality, consider $j=1$, from \eqref{eq:eta1diff}, we have that
\begin{equation*}
    |\partial_{\eta_1}\phi|\gtrsim 2^{\ell_1}. 
\end{equation*}
Therefore, we introduce the symbol
\begin{equation}\label{eqn:def_a_symbol_case_4}
a(\eta,\xi)
:=\frac{\partial_\xi\phi(\eta,\xi)}{\partial_{\eta_1}\phi(\eta,\xi)},\qquad
m(\eta,\xi)
:=\partial_{\eta_1} a(\eta,\xi).
\end{equation}
Also, notice that the bootstrap assumptions give
\begin{subequations}\label{eq:derivative_bootstrap_assump_case_1}
\begin{align}
\big\|\partial_{\eta_1}\widehat f_{k_1}(\tau,\eta_1)\big\|_{L^2_{\eta_1}}
\lesssim \eps_1\Big(1 + 2^{\frac{k_1}{2}}\Big),\qquad\big\|\partial_{\eta_5}\widehat f(\tau,\eta_5)\big\|_{L^2_{\eta_5}}
\lesssim \eps_1\Big(1+ 2^{\frac{\ell_1}{2}}\Big).
\end{align}
\end{subequations}
We integrate by parts in the $\eta_1$ direction and define the symbol
\begin{equation}\label{eq:symbol_subcase4-1}
    B(\eta,\xi)=\frac{\partial_\xi\phi}{\partial_{\eta_1}\phi}\,\psi_k(\xi-\xi_0)\,\prod^{4}_{j=1}\psi_{k_j}(\eta_j)\,\chi_{\ell_1}\left(r(\eta_1+\eta_5)\right).
\end{equation}
Note that in this case, we do not have the stronger decay $t^{-\frac{1}{2}}$. However, the proof still follows closely to the proof of Lemma \ref{prop:anomalous_resonance_case_3}, but with the worst case case decay $t^{-\frac{1}{3}}$. Moreover, $r(\eta_1)$ in secondary localization function of the symbol is nonlinear which results in a difficulty on applying the muiltilinear estimate. We therefore apply the inverse function theorem to establish a new coordinate with 
$$\partial_{\mu_1}= \frac{1}{r'(\eta_1)-1}\partial_{\eta_1} - \partial_{\eta_2} - \partial_{\eta_3}-\partial_{\eta_4} +\partial_{\xi}$$
where $|r'(\eta_j)| \neq 1$. This identity shows that differentiation with respect to $\left(r(\eta_1)+\eta_5\right)$ can be expressed as a linear combination of derivatives in the original frequency variables. Since the symbol enjoys bounds under differentiation in $(\eta,\xi)$, the same bounds apply to $\partial_{\mu_1} B$. From then, the proof follows exactly the proof of Lemma \ref{prop:anomalous_resonance_case_3}.

For the second case, we have
$$|r(\eta_j)+\eta_5|\lesssim 2^{k_{*}} \quad \forall\ j=1,2,3,4. $$
By Taylor's expansion, it follows that
\begin{equation}
\xi-\xi_0 =\left(1-r'(\eta_0)\right)\left(\eta_1-\eta_0\right) + \sum_{j=2}^4 \left(\eta_j-\eta_0\right)+\mathcal{O}\left(\max\left\{2^{k_*},2^{2k_1}\right\}\right). \label{eq:taylor_case_2_case_4_anomalous}
\end{equation}
Then in this setting, one has  $r(\eta_j)=-\eta_5+\mathcal{O}(2^{k_*})$ for all $j$. That is to say, $\eta_{i}=\eta_1+\mathcal{O}(2^{k_*})$ for $i=2,3,4$. Bring this back to \eqref{eq:taylor_case_2_case_4_anomalous} yields
\begin{equation*}
\xi-\xi_0 = \left(4-r'(\eta_0)\right)\left(\eta_1-\eta_0\right) + \mathcal{O}\left(\max\left\{2^{k_*},2^{2k_1}\right\}\right).
\end{equation*}
Also, once the secondary localizations $|r(\eta_1)+\eta_5|\lesssim 2^{k_*}$ are imposed and the dyadic scale $k_1$ is fixed, the output scale $k$ is forced to lie in a finite set $k\in\{k_1+\mathcal{O}(1)\}$. And, by noticing that $|r'(\eta_q)|\in (0,2^{-4})$, above implies that $2^{k}\approx 2^{k_1}$. Combine this and Lemma \ref{lem:anomalous_resonance_general_taylor}, we have the lower bound $|\phi|\gtrsim 2^{k}$. We thus reduce this case to the proof of Lemma \ref{prop:anomalous_resonance_case_2}.

\par It remains to handle $J_{m}$ parts. There are two cases after performing a secondary dyadic localization in $\vert r(\eta_j)+\eta_5\vert$. For the first case, we have $$|r(\eta_j)+\eta_5|\sim 2^{\ell_j},\quad 2^{k_*}\lesssim 2^{\ell_j}\ll 1$$ 
where the proof follows from the first case of $I_{mk}$ that we have already discussed. The second case can be reduced to the proof of Lemma \ref{prop:anomalous_resonance_case_2} by noticing $|\phi|\gtrsim  2^{k}$. 
\end{proof}

\subsubsection{Degenerate resonance point}
We now focus on the contribution arising from the degenerate space--time resonant point at the origin.
Recall that
$$\omega(0)=0, \qquad \partial_\xi \phi(0,0,0,0;0)=0.$$ Fix $\delta_2>0$ sufficiently small and, for each dyadic time interval $t\sim 2^m$, define the threshold scale $k_*\in\mathbb{Z}$ by
\begin{equation}\label{eq:def_k_star_degenrate}
2^{k_*} := \max\Big\{2^k : 2^k \le 2^{-20}\,2^{m(\delta_2-\frac13)}\Big\}.
\end{equation}
\begin{lemma}\label{lem:deg_taylor}
By invoking Taylor's expansion adapted to the degeneracy at the origin.
There exist points $\xi_\theta, \widetilde{\xi_\theta}\in (0,\xi)$, and $\eta_{\theta}\in 0,\eta_5$ such that
\begin{subequations}\label{eq:degenerate_taylor}
\begin{align}
\omega(\xi) &= \omega'(\xi_\theta)\,\xi, \\
\partial_\xi \phi &= -\omega'''(\widetilde{\xi_\theta})\,\xi^2
+ \omega'''(\eta_\theta)\,\eta_5^2.
\end{align}
\end{subequations}
\end{lemma}
\begin{lemma}
\label{prop:deg_case_1}
    In the frequency region
    \begin{equation*}
        \mathcal{D}_1=\left\{ (\eta_1,\eta_2,\eta_3,\eta_4; \xi) : |\eta_j| \le 2^{k_*} \ \forall j=1,\dots,4, \text{ and } |\xi| \le 2^{k_*} \right\},
    \end{equation*}
    where $k_*$ follows the definition in \eqref{eq:def_k_star_degenrate}. Then, under the bootstrap assumption \eqref{eq:bootstrap_assump_XT},  the contribution of $\mathcal{D}_1$ satisfies
    \begin{align*}
        \sum_m \snorm{I_{m,k_*}}_{L^2_\xi}&\lesssim \eps_1^5,\qquad
        \sum_m \snorm{J_{m,k_*}}_{L^\infty_\xi}\lesssim \eps_1^5.
    \end{align*}
    
\end{lemma}
\begin{proof}
    For the $I_{m}$ part, we further split
\begin{equation}\label{deg_I_m_part}
    I_{m,k_*}:= \varphi_{\leq k_*}(\xi)\omega(\xi)\int^{t_2}_{t_1} d \tau\, \tau\int d\eta_{1234}\;e^{-i\tau \phi}\ \partial_{\xi}\phi\prod^4_{j=1}\widehat{f}_{\leq 2^{k_*}}(\eta_j) \widehat{f}(\tau,\eta_5).
\end{equation}
and by the measure estimate and following the proof of Lemma \ref{prop:anomalous_resonance_near_resonance} follows that
\begin{align*}
\hspace{-1.3cm}
    \|I_{m,k_*}\|_{L^2_\xi}&\lesssim \int^{t_2}_{t_1}d \tau \ \tau\ \Big\| \varphi_{\leq k_*}(\xi)\omega(\xi) \int d\eta_{1234} \;\;e^{-i\tau \phi}\ \partial_{\xi}\phi \prod^4_{j=1}\widehat{f}_{\leq 2^{k_*}}(\eta_j) \widehat{f}(\tau,\eta_5)\Big\|_{L^2_\xi}\\
    &\lesssim \eps_1^5\int^{t_1}_{t_2} d \tau \ \tau \ 2^{k_*}\cdot2^{2k_*}\cdot 2^{4k_*}\sqrt{\operatorname{meas.}\{\xi: \vert \xi\vert\lesssim 2^{k_*}\}}\lesssim \eps_1^5 2^{m(\frac{15}{2}\delta_2-\frac{1}{2})}.
\end{align*}
Summing over $m$ yields
$$\sum^{m_*}_{m=0} \|I_{m,k_*}\|_{L^2_\xi}\lesssim \eps_1^5 2^{m_*(\frac{15}{2}\delta_2-\frac{1}{2})}\lesssim \eps_1^5 t^{\left(\frac{15}{2}\delta_2-\frac{1}{2}\right)}.$$
Choosing $\delta_2>0$ sufficiently small so that $\frac{15}{2}\delta_2-\frac12<0$, the above contribution is admissible. And for part $J$, we have
$$\sum_m\|J_m\|_{L^\infty_\xi}\lesssim \sum_m \eps_1^5 2^{m(4\delta_2-\frac{1}{3})}\lesssim \eps_1^5.$$
By choosing $\delta_2$ sufficiently small so that $4\delta_2-\frac{1}{3}<0$, thus the above contribution is admissible.
\end{proof}
Before moving to the next region, to prepare integration by parts in $\tau$  we recall the following bounds for the time derivative of the profile, see \cite{Morgan}.
\begin{lemma}
    \begin{equation}\label{Lemma_estimate_deg_1}
        \snorm{\partial_\tau \widehat{f}_{\leq 2^{k_*}}(\tau,\eta_1)}_{L^\infty_{\eta_1}}\lesssim \eps_1^5 2^{k_*-m}\varphi_{\leq 2^{k_*}}(\eta_1),\qquad
    \snorm{\psi_k(\xi)\prod^4_{j=1}\varphi_{\leq 2^{k_*}}(\eta_j)\partial_\tau\widehat{f}(\tau,\eta_5)}\lesssim \eps_1^5 2^{k_*-m}.
\end{equation}
\end{lemma}
\begin{lemma}\label{prop:deg_case_2}
    In the frequency region
    \begin{equation*}
        \mathcal{D}_2=\left\{ (\eta_1,\eta_2,\eta_3,\eta_4; \xi) : |\eta_j| \le 2^{k_*} \ \forall j=1,2,3,4, \text{ and } |\xi| \sim 2^{k} \text{ for }k_*\leq k\ll 1 \right\},
    \end{equation*}
    where $k_*$ follows the definition in \eqref{eq:def_k_star_degenrate}. Then, under the bootstrap assumption \eqref{eq:bootstrap_assump_XT},  the contribution of $\mathcal{D}_2$ satisfies
    \begin{align*}
        \sum_m\snorm{I_{m,k_*}}_{L^2_\xi}\lesssim \eps_1^5,\qquad\sum_m\snorm{J_{m,k_*}}_{L^\infty_\xi}\lesssim \eps_1^5.
    \end{align*}
\end{lemma}
\begin{proof}
The argument would be analogous to the proof of Lemma  \ref{prop:anomalous_resonance_case_2}. The lower bound for the phase and corresponding derivative bounds \begin{equation}\label{eq:deg_phase_bound_case_2}
        |\phi|\gtrsim 2^{2k} 2^{k_*},\quad
        |\partial_\xi  \phi|\lesssim 2^{2k}.
\end{equation}
follows by invoking the Taylor expansion. As a consequence of \eqref{eq:deg_phase_bound_case_2}, we obtain the bound
\begin{equation}\label{eq:symbol_ratio_bound}
\left|m\right|\lesssim\frac{2^{k}\cdot 2^{2k}}{2^{2k+k_*}}\lesssim
2^{k-k_*}
\end{equation}
for $m:=\omega\frac{\partial_\xi \phi}{\phi}$. We have $I_{mk}$
\begin{align*}
I_{mk}:=\;&\psi_k(\xi)\,\omega(\xi)\!\int_{t_1}^{t_2}d\tau\,\tau
\int d\eta_{1234} \;e^{-i\tau \phi}\ \partial_{\xi}\phi \prod^{4}_{j=1}
    \widehat{f}_{\leq 2^{k_*}}(\tau,\eta_j)\widehat{f}(\tau,\eta_5).
\end{align*}
The bounds \eqref{eq:deg_phase_bound_case_2} allow to use integration by part in $\tau$. Integrating by parts in $\tau$ will result in three terms $I_j$, $j=1,2,3$, and we   estimate each term. . 
We present the proof for the term  $I_1$ from the integration by parts, the terms $I_2$ and $I_3$ will thus be omitted. It follows that
\begin{align*}
    \snorm{I_1}_{L^2_\xi}&\lesssim \int^{t_2}_{t_1}d\tau \snorm{\int d\eta_{1234}  \ e^{-i\tau\phi} \frac{\partial_{\xi}\phi}{\phi}  \tau \, \partial_{\tau}\left(  \prod^4_{j=1}\widehat{f}_{\lesssim 2^{k_{*}}} \left(\tau,\eta_{j}\right) 
    \widehat{f}\left(\tau,\eta_{5}\right)  \right) }\\
    &\lesssim \int^{t_1}_{t_2} d\tau \ 2^{m}2^{k-k_*}\left(\eps^9_1 2^{4k_*-m}+\eps_1^52^{k_*-m}\right)\sqrt{\operatorname{meas.}\{\xi: \vert \xi\vert\sim 2^{k}\}}\\
    &\lesssim 2^{\frac{3}{2}k+3k_*}(\eps^9_1+\eps_1^5).
\end{align*}
Therefore,
$$\snorm{I_{mk}}_{L^2_\xi}\lesssim 2^{m(\frac{9}{2}\delta_2-\frac{3}{2})}\left(\eps^9_1 +\eps_1^5\right).$$
Fix $\delta_2$ sufficiently small such that $\frac{9}{2}\delta_2-\frac{3}{2}<0$, and summing over frequencies that satisfying $k_*\leq k\ll 1$ and $m$ gives
$$\sum_{m}\sum_{\operatorname{freq.}}\snorm{I_{m,k}}_{L^2_\xi}\lesssim \eps_1^5.$$
For $J_{mk}$, it follows that
$$\snorm{J_{mk}}_{L^\infty_\xi}\lesssim \eps_1^5 2^{-k-k_*}2^{4k_*}+\eps^9_1 2^{3k_*-k}\lesssim \eps_1^5 2^{-k}2^{3k_*}.$$
Summing over frequencies that satisfy $k_*\leq k\ll 1$ and $m$ gives
$$\sum_m\sum_{\operatorname{freq.}}\snorm{J_{mk}}_{L^\infty_\xi}\lesssim \eps_1^5 t^{2\delta_2-\frac{2}{3}}$$
which is the desire bound by choosing $\delta_2$ sufficiently small.    
\end{proof}
\begin{lemma}\label{prop:deg_case_3}
Let $\mathcal{D}_3$ be the frequency region defined by the following conditions:
\begin{itemize}
    \item The input frequencies satisfy
    $$ |\eta_j| \sim 2^{k_j}, \qquad 2^{k_*} \le 2^{k_j} \le 2^{-10}, \qquad j=1,2,3,4,$$
    with an ordering
    $$  k_1 \ge k_2 \ge k_3 \ge k_4,$$
    and at least one strict separation among them, i.e.
    $$\max_{i<j}(k_i-k_j) \gg 1. $$
    \item The output frequency satisfies either
    $$  |\xi| \le 2^{k_*},  \quad \text{or} \quad |\xi| \sim 2^{k}, \qquad 2^{k_*} \le 2^{k} \le 2^{-10}, $$
    with $k \lesssim k_1$.
\end{itemize}
Then, under the bootstrap assumptions, the contribution of $\mathcal{D}_3$ satisfies
    \begin{align*}
        \sum_m\snorm{I_{m,k_*}}_{L^2_\xi}\lesssim \eps_1^5,\qquad
        \sum_m\snorm{J_{m,k_*}}_{L^\infty_\xi}\lesssim \eps_1^5.
    \end{align*}
    
\end{lemma}
\begin{proof}
Without loss of generality, we consider the case when $k_1\gg k_2$ and $|\xi|\sim 2^{k}$. A Taylor's expansion of the phase shows that
\begin{equation}\label{eq:deg_case_3_freq_derivative_bound}
    \big|\partial_{\eta_1-\eta_2}\phi\big| \gtrsim 2^{2k_1}.
\end{equation}
Moreover, by the mean value theorem, we have 
$$\left|\partial_\xi\phi\right|=\left|\omega'(\eta_5)-\omega'(\xi)\right|\leq \sup_{s\in [\xi,\eta_5]}|\omega''(s)|\cdot\left|\eta_5-\xi\right|.$$
Notice that in $\mathcal{D}_3$ we have $\left|\eta_5-\xi\right|\lesssim 2^{k_1}$. Also, by an application of triangle inequality it follows that
$$\sup_{s\in [\xi,\eta_5]}|\omega''(s)|\lesssim \sup_{s\in [\xi,\eta_5]}|s|\lesssim \max\{|\xi|,|\eta_5|\}\lesssim 2^{k}+2^{k_1}\lesssim 2^{k_1}.$$
Therefore, we have the bound
\begin{equation}\label{eq:deg_case_3_via_freq_bounds}
    |\partial_\xi \phi| \lesssim 2^{2k_1},
    \qquad
    |\partial_{\eta_1-\eta_2}^2\phi| \lesssim 2^{k_1}.
\end{equation}
Consequently, we have 
\begin{align}
\label{eqn:case3_symbol_bounds}
 \left|a(\eta,\xi)\right| \lesssim 1,
 \quad  \left|m(\eta,\xi) \right| \lesssim 2^{-2k_{1}}2^{k_1}\lesssim 2^{-k_1}.
\end{align}
with the symbols
$$a(\eta,\xi):=\frac{\partial_\xi \phi}{\partial_{\eta_1-\eta_2}\phi},\qquad m(\eta,\xi):=\partial_{\eta_1-\eta_2}a(\eta,\xi).$$
Since we have $2^k\lesssim 2^{k_1}$ in this case, it follows that
\begin{align*}
B_1(\eta,\xi) =  a(\eta,\xi)\ \psi_{k}\left(\xi\right) \prod^{4}_{j=1}  \psi_{k_j}\left(\tau,\eta_j\right),\qquad
B_2(\eta,\xi) =  m(\eta,\xi)\ \psi_{k}\left(\xi\right) \prod^{4}_{j=1}  \psi_{k_j}\left(\tau,\eta_j\right).
\end{align*}
with $A=1, 2^{-k_1}$, respectively, satisfy the hypotheses of \cite[Lemma~B.3]{Morgan}. We integrate by parts in the $\eta_1-\eta_2$ direction and write
\begin{equation}\label{eq:deg_case_3_IBP}
I_{m,k_1k_2k_3k_4,k}
=
i\bigl(\tilde I_1+\tilde I_2+\tilde I_3\bigr),
\end{equation}
where
\begin{align*}
\tilde I_1
&:=-\psi_k(\xi-\xi_0)\omega(\xi)\int_{t_1}^{t_2}\!\!\int e^{-i\tau\phi}\,a(\eta,\xi)\,(\partial_{\eta_1}\hat f_{k_1})\prod^5_{j=2}\hat f_{k_j}(\tau,\eta_j) \,d\eta_{1234}\,d\tau,\\
\tilde I_2
&:=\psi_k(\xi-\xi_0)\omega(\xi)
\int_{t_1}^{t_2}\!\!\int
e^{-i\tau\phi}\,
a(\eta,\xi)\,
\hat f_{k_1}(\partial_{\eta_2}\hat f_{k_2})\prod^5_{j=3}\hat f_{k_j}(\tau,\eta_j)
\,d\eta_{1234}\,d\tau,\\
\tilde I_3
&:=
-\psi_k(\xi-\xi_0)\omega(\xi)
\int_{t_1}^{t_2}\!\!\int
e^{-i\tau\phi}\,m(\eta,\xi) \prod_{j=1}^4 \hat f_{k_j}(\tau,\eta_j)\hat f(\tau,\eta_5)
\,d\eta_{1234}\,d\tau.
\end{align*}
Invoking \cite[Proposition B.1]{Morgan} yields
\begin{align*}
\snorm{I_{m k_1 k_2 k_3 k}}_{L^2_\xi}
&\lesssim \int_{t_1}^{t_2} \! d\tau \;2^k \bigl(2^{-k_1}\bigr)\bigl(\eps_1 2^{\frac{1}{2}k_1}\bigr)\prod^4_{j=2}\bigl(\eps_1 2^{-\frac{1}{2}k_j} 2^{-\frac{m}{2}}\bigr)\bigl(\eps_1 2^{-\frac{m}{3}}\bigr) \\
&\quad \quad \quad  \quad \quad \quad \quad \quad \quad \;+ \; 2^k \Bigl(2^{m} + 2^{\frac{1}{2}k_1}\Bigr)\prod^4_{j=2} \bigl(\eps_1 2^{-\frac{1}{2}k_j} 2^{-\frac{m}{2}}\bigr)\bigl(\eps_1 2^{-\frac{m}{3}}\bigr) \\
&\lesssim \eps_1^5 \int_{t_1}^{t_2} \! d\tau \;
2^{\,k - \frac{1}{2}(k_1+k_2+k_3+k_4)} 2^{-\frac{11}{6}m}
+ 2^{\,k - \frac{1}{2}(k_2+k_3+k_4)} 2^{-\frac{11}{6}m}.
\end{align*}
Summing over frequencies and $m$ yields
$$\sum_{m}\sum_{\operatorname{freq.}}\snorm{I_{m k_1 k_2 k_3 k}}_{L^2_\xi}\lesssim \eps_1^5\left(t^
{-2\delta_2-\frac{1}{6}}+t^{-\frac{3}{2}\delta_2-\frac{4}{3}}\right)\lesssim \eps_1^5.$$
For $J_m$, notice that
$$\left|\partial_{\eta_1-\eta_2}\left(\frac{1}{\partial_{\eta_1-\eta_2}\phi}\right)\right|\lesssim 2^{-3k_1}.$$
Hence, define the symbol
\begin{equation}\label{eq:def_a_tilde_symbol_deg_case_3}
\tilde{a}(\eta,\xi):=\frac{1}{ \partial_{\eta_1-\eta_2}\phi},\qquad 
\tilde{m}(\eta,\xi):=\frac{\partial^2_{\eta_1-\eta_2}\phi}{ (\partial_{\eta_1-\eta_2}\phi)^2} .
\end{equation}
Therefore, the estimates of the symbols are
    \begin{align}\label{eqn:a_m_tilde_bound}
     \left|\tilde{a}(\eta,\xi) \right| \lesssim 2^{-2k_{1}}, \quad \left|\tilde{m}(\eta,\xi)\right| \lesssim 2^{-3k_{1}}. 
    \end{align}
Using the multilinear estimates, we are led to
\begin{align*}
    \snorm{J_{mk_1k_2k_3k_4k}}_{L_\xi^\infty}&\lesssim \int^{t_2}_{t_1} d \tau \ \tau^{-1}\ 2^k\   2^{-3k_1}\left(\eps_12^{\frac{k_1}{2}}\right)\prod^4_{j=2} \bigl(\eps_1 2^{-\frac{1}{2}k_j}2^{-\frac{m}{2}}\bigr)\left(\eps_1\right)\\
    &\quad \quad \quad \quad \quad  \quad \quad \;+ \;\tau^{-1}\ 2^k\ 2^{-2k_1}\eps_1 \prod^4_{j=2} \bigl(\eps_1 2^{-\frac{1}{2}k_j}2^{-\frac{m}{2}}\bigr)\left(\eps_1\right)\\
    &\lesssim \eps_1^5\left[2^{-\frac{3}{2}k_1-\frac{k_2}{2}-\frac{k_3}{2}-\frac{k_4}{2}-m}+2^{-k_1-\frac{k_2}{2}-\frac{k_3}{2}-\frac{k_4}{2}-\frac{5m}{2}}\right].
\end{align*}
Summing over all frequencies and $m$ gives
$$\sum_{m}\sum_{\operatorname{freq.}}\snorm{J_{mk_1k_2k_3k_4k}}_{L^\infty_\xi}\lesssim \sum_{m}\eps_1^5\cdot\left[2^{(-3m\delta_2)}+2^{m(-\frac{5}{3}-\frac{5}{2}\delta_2)}\right]\lesssim \eps_1^5$$as desired. 
\end{proof}
\begin{lemma}\label{prop:deg_case_4}
Let $\mathcal{D}_4$ be the frequency region defined by the following conditions:
\begin{itemize}
    \item The input frequencies satisfy
    $$ |\eta_j| \sim 2^{k_j}, \qquad 2^{k_*} \le 2^{k_j} \ll 1, \qquad j=1,2,3,4,$$
    with an ordering
    $$  k_1 \ge k_2 \ge k_3 \ge k_4.$$
    \item The output frequency satisfies either
    $$  |\xi| \sim 2^k, \qquad 2^{k_*} \le 2^{k_j} \ll 1$$
    with the condition $2^k \lesssim 2^{k_1}$.
\end{itemize}
Then, under the bootstrap assumptions, the contribution of $\mathcal{D}_4$ satisfies
    \begin{align*}
    \sum\snorm{I_{m,k_*}}_{L^2_\xi}\lesssim \eps_1^5,\qquad\sum\snorm{J_{m,k_*}}_{L^\infty_\xi}\lesssim \eps_1^5.
    \end{align*}
\end{lemma}
\begin{proof}
The proof combines the proof of  Lemma \ref{prop:secondary_resonance} and the proof of Lemma \ref{prop:deg_case_3}. We have 
\begin{equation}
    \label{eq:deg_case_4_derivative_bound}
    |\partial_{\eta_{1}}\phi|\approx|\omega(\eta_1)-\omega(\eta_5)|\gtrsim |\eta_5|^2 \gtrsim 2^{2k}. 
\end{equation}
which implies that the integration by parts with respect to $\eta_1$ is allowed. By invoking the Taylor expansion it follows that
    \begin{align}
        &|\partial_\xi\phi| \approx |\omega'(\eta_5)-\omega'(\xi)|\lesssim 2^k\cdot|\eta_5-\xi|\lesssim 2^{2k}
    \end{align}
and also
\begin{equation}
    |\partial_{\eta_1}\partial_\xi \phi|\lesssim 2^k,\qquad |\partial^2_{\eta_1}\phi|\lesssim 2^k.
\end{equation}
Therefore, we have the symbol bounds
\begin{equation}
 \left|a(\eta,\xi)\right| \lesssim 1, \quad  \left| m(\eta,\xi) \right| \lesssim 2^{-k}, \label{eq:deg_case_4_a_m_bound} 
\end{equation}
with
\begin{equation}\label{eq:a_m_define_eta_1_derivative}
    a(\eta,\xi):=\frac{\partial_{\xi}\phi}{\partial_{\eta_{1}}\phi},\qquad m(\eta,\xi):=\partial_{\eta_1}a(\eta,\xi).
\end{equation}
Define,
\begin{equation*}
    \begin{aligned}
B_1 (\eta,\xi) =a(\eta,\xi) \psi_{k}(\xi)  \prod^4_{j=1}\psi_{k_{j}}\left(\tau,\eta_{j}\right),\qquad
B_2 (\eta,\xi) =m(\eta,\xi)\psi_{k}(\xi) \prod^4_{j=1}\psi_{k_{j}}\left(\tau,\eta_{j}\right)
\end{aligned}
\end{equation*}
with $A=1, 2^{-k}$, respectively. It then follows a similar estimate  to the proof of Lemma \ref{prop:deg_case_3}, which yields
\begin{equation}
    \begin{aligned}
\snorm{I_{m k_{1} k_{2} k_{3} k_{4} k}}_{L_{\xi}^{2}} &\lesssim \eps_{1}^{5} \int_{t_{1}}^{t_{2}} d \tau\ 2^{-\frac{1}{2}\left(-k_{1}+k_{2}+k_{3}+k_{4}\right)} 2^{-\frac{11}{6}m}+2^{k} 2^{-\frac{1}{2}\left(k_{2}+k_{3}+k_{4}\right)} 2^{-\frac{11}{6}m}.
\end{aligned}
\end{equation}
Summing over frequencies and $m$, we have
$$\sum_m\sum_{\operatorname{freq.}}\snorm{I_{m k_1 k_2 k_3 k_4 k}}_{L^2_\xi}\lesssim \eps_1^5\left(t^{-2\delta_2-\frac{1}{6}}+t^{-\frac{3}{2}\delta_2-\frac{4}{3}}\right)\lesssim \eps_1^5.$$
And for part $J_m$, we acquire the bound
$$\sum_{m=0}^{m_*}\sum_{\operatorname{freq.}}\snorm{J_{mk_1k_2k_3k_4k}}_{L^\infty_\xi}\lesssim \eps_1^5$$
by the estimates of part $J$ in the proof of Lemma  \ref{prop:deg_case_3}.
\end{proof}
\begin{lemma}\label{prop:deg_case_5}
Let $\mathcal{D}_5$ be the frequency region defined by the following conditions:
\begin{itemize}
    \item The input frequencies satisfy
    $$ |\eta_j| \sim 2^{k_j}, \qquad 2^{k_*} \lesssim 2^{k_j} \lesssim 1, \qquad j=1,2,3,4,$$
    with the relation
    $$  k_1 \approx k_2 \approx k_3 \approx k_4.$$
    \item The output frequency satisfies either
    $$  |\xi| \lesssim 2^{k_*},\quad 2^{k_*}\lesssim 2^{k_1}  \quad \text{or} \quad |\xi| \sim 2^{k}, \quad 2^{k_*} \lesssim 2^{k} \lesssim 1, \quad 2^{k}\lesssim 2^{k_1}.$$
\end{itemize}
Then, under the bootstrap assumptions \eqref{eq:bootstrap_assump_XT}, the contribution of $\mathcal{D}_5$ satisfies
    \begin{align*}
        \sum\snorm{I_{m,k_*}}_{L^2_\xi}\lesssim \eps_1^5,\qquad
        \sum\snorm{J_{m,k_*}}_{L^\infty_\xi}\lesssim \eps_1^5.
    \end{align*}
\end{lemma}
\begin{proof}
By the Taylor expansion it follows that
\begin{equation}
    |\partial_{\eta_j}\phi| \gtrsim \left|\eta_j-\eta_5\right| \ \left|\eta_j+\eta_5\right|. 
\end{equation}
Hence, we reduce the discussion to several cases by the size of $|\eta_j\pm\eta_5|$. If all these terms are localized at scale $\lesssim 2^{k_*}$, the interaction is confined near the degenerate resonant point and the estimate reduces to the proof of Lemma  \ref{prop:deg_case_1}. We next start to discuss the rest of the four cases.
\\\\
\noindent
\underline{\textit{Case 1}}
\\\\
\noindent 
Assume there is at least one $j$ such that
$$|\eta_j\pm \eta_5|\sim 2^{\ell_j},\quad 2^{k_*}\lesssim 2^k\ll 1\quad \forall j=1,2,3,4.$$
Without loss of generality, we fix $j=1$. We distinguish two cases according to the relative size of $2^{\ell_1}$ and $2^{k_1}$. First, assume that $2^{\ell_1}\gtrsim 2^{k_1}$. In this case, we have $|\partial_{\eta_1}\phi|\gtrsim 2^{2\ell_1}$. By the proof of Lemma  \ref{prop:deg_case_3} with the localization $2^{k_1}\lesssim 2^{k}$, we have 
$$\left|a(\eta,\xi)\right| \lesssim 2^{2k_1-2\ell_1}\lesssim 1,\quad \left|m(\eta,\xi)\right| \lesssim 2^{2k_1}2^{-2\ell_1}\lesssim  2^{-2\ell_1}$$
where $a$ and $m$ follow the definition in \eqref{eq:a_m_define_eta_1_derivative}.
In order to reduce to the proof of Lemma  \ref{prop:deg_case_3}, we define the symbols because of the secondary  localization $2^{k_1}\lesssim 2^{\ell_1}$
  \begin{align}
  B_1(\eta,\xi) &= a(\eta,\xi)
  \psi_k(\xi)\prod^4_{j=1}\psi_{k_j}(\eta_j)\prod^{4}_{j=1}
  \psi_{\ell_j}(\eta_j\pm\eta_5)\label{eq:B_1_deg_case_5_case_2}\\
  B_2(\eta,\xi) &= m(\eta,\xi)
  \psi_k(\xi)\prod^4_{j=1}\psi_{k_j}(\eta_j)\prod^{4}_{j=1}
  \psi_{\ell_j}(\eta_j\pm\eta_5),\label{eq:B_2_deg_case_5_case_2}
  \end{align}
which satisfied the condition of \cite[Lemma B.3]{Morgan} with $A=1,2^{-\ell_1}$, respectively. From now, the proof is reduced to follow the proof of Lemma \ref{prop:deg_case_3} by noticing that $2^{-\ell_1}\leq 2^{-k_1}$. 

We next assume $2^{\ell_1}\ll 2^{k_1}$ in this case,  which is only one of the $|\eta_1+\eta_5|\gtrsim 2^{k_1}$, $|\eta_1-\eta_5|\gtrsim 2^{k_1}$ holds. By the mean value theorem, we have $\left|\partial_{\eta_1}\phi\right|\gtrsim 2^{k_1}2^{\ell_1}$ which implies that
    \begin{align*}
         \left|a(\eta,\xi)\right| \lesssim 2^{2k_1}2^{-2\ell_1} \lesssim 2^{k_1}2^{-\ell_1}, \quad \left|m(\eta,\xi)\right| \lesssim 2^{2k_1}2^{-2\ell_1} \lesssim 2^{k_1}2^{-2\ell_1}
    \end{align*}
where $a$ and $m$ follow the definition in \eqref{eq:a_m_define_eta_1_derivative}.
We define the symbols $B_1(\eta,\xi)$ and $B_2(\eta,\xi)$ as in
\eqref{eq:B_1_deg_case_5_case_2} and \eqref{eq:B_2_deg_case_5_case_2}, with $A = 2^{k_1}2^{-\ell_1}$ and $A = 2^{k_1}2^{-2\ell_1},$ respectively. Hence, it follows that
\begin{align*}
    \snorm{I_{mk_1k_2k_3k_4k\ell_1\ell_2\ell_3\ell_4}}_{L^2_\xi}&\lesssim \eps_1^5 2^k 2^m \left(2^{-\frac{k_1}{2}}2^{-\frac{m}{2}}\right)^4\cdot \left(2^{k_1}2^{-2\ell_1}2^{\frac{k_1}{2}}+2^{k_1}2^{-\ell_1}2^{\frac{k_1}{2}}\right)\\
    &\lesssim \eps_1^5\,2^k\,2^{-m}\,2^{-\,\frac12 k_1}\Big(2^{-2\ell_1}+2^{-\ell_1}\Big).
\end{align*}
Summing over all frequencies and $m$ and fix $\delta_2$ sufficiently small, it follows that,
$$\sum_{m}\sum_{\operatorname{freq.}}\snorm{I_{mk_1k_2k_3k_4k\ell_1\ell_2\ell_3\ell_4}}_{L^2_\xi}\lesssim \sum^{m_*}_{m=0}\eps_1^5\,m^{6}\Big(t^{-\frac16-\frac{5}{2}\delta_2}\;+\;t^{-\frac12-\frac{3}{2}\delta_2}\Big)\lesssim \eps_1^5.$$
\\
\noindent
\underline{\textit{Case 2}}
\\\\
We first define some notation in this subcase. We set
$$\sigma_j\in \{\pm 1\}\quad \text{for}\quad j=1,2,3,4,5,\xi$$ 
which satisfies 
$$\sigma_\xi-\sum^4_{j=1}\sigma_j=\sigma_5.$$
Also, set  $\eta_j=(\sigma_j\sigma_5)\eta_5$ for $j=1,2,3,4$.
In this case, suppose that
$$
\begin{cases}
    \vert \eta_j-\eta_5\vert\lesssim 2^{k_*}\quad \;\;\;\text{if}\quad\sigma_j\sigma_5=1\\
    \vert \eta_j+\eta_5\vert\lesssim 2^{k_*}\quad \;\;\;\text{if}\quad\sigma_j\sigma_5=-1
\end{cases}
$$
but the other three combinations are $\gtrsim 2^{k_*}$. The above localization implies that 
$$
\begin{cases}
    \vert \eta_j+\eta_5\vert\sim 2^{k_1}\quad \;\;\;\text{if}\quad\sigma_j\sigma_5=1\\
    \vert \eta_j-\eta_5\vert\sim 2^{k_1}\quad \;\;\;\text{if}\quad\sigma_j\sigma_5=-1
\end{cases}
.
$$
The localization above implies 
\begin{align*}
    \xi=\left(\sum^4_{j=1}\sigma_j\sigma_5+1\right)\eta_5+\sum^4_{j=1}(\eta_j-\sigma_j\sigma_5\eta_5)
\end{align*}
and since $\sum \sigma_j\sigma_5\in\{0,\pm 2\}$ it follows that $\vert\xi\vert\sim\vert\eta_5\vert\sim2^{k_1}.$ From this, this case is close to the scenario in the anomalous resonance from previous subsection. We then perform a tertiary localization according to 
$$
\begin{cases}
    \vert \eta_j-\eta_5\vert\sim 2^{k}\quad \;\;\;\text{if}\quad\sigma_j\sigma_5=1\\
    \vert \eta_j+\eta_5\vert\sim 2^{k}\quad \;\;\;\text{if}\quad\sigma_j\sigma_5=-1
\end{cases}
,
$$
where $2^{-100m}\ \le\ 2^k\ \le\ 2^{k_*}\ll 2^{k_1}.$
Therefore, we have
$$\left|\partial_{\eta_1}\phi\right|\gtrsim \left|\eta_1-\eta_5\right|\,\left|\eta_1+\eta_5\right|\sim 2^k2^{k_1}.$$
Consequently, $|\partial_{\xi}\phi|\lesssim 2^{2k_1}$ gives
\begin{equation*}
    \left|a(\eta,\xi)\right| \lesssim 2^{-k}2^{k_1}, \quad \left|m(\eta,\xi)\right| \lesssim 2^{-k},
\end{equation*}
where $a$ and $m$ follow the definition in \eqref{eq:a_m_define_eta_1_derivative}. We define symbols
\begin{equation*}
    \begin{aligned}
B_1 (\eta,\xi) =a(\eta,\xi) \psi_{k}(\xi)\prod^4_{j=1} \psi_{k_j}\left(\tau,\eta_j\right),\qquad
B_2 (\eta,\xi) =m(\eta,\xi) \psi_{k}(\xi) \prod^4_{j=1} \psi_{k_j}\left(\tau,\eta_j\right)
\end{aligned}
\end{equation*}
with $A=1, 2^{-k}$, respectively.
Applying the $L^2_\xi$ multilinear estimate we have
\begin{align*}
    \snorm{I_{mk_1k_2k_3k_4k}}_{L^2_\xi}&\lesssim \eps_1^5 2^k 2^m \left(2^{-\frac{k_1}{2}}2^{-\frac{m}{2}}\right)^4\cdot \left(2^{\frac{k_1}{2}}+2^{-k}2^{\frac{k_1}{2}}\right)\\
    &\lesssim \eps_1^5\,2^{-m}\Big(2^{\,k-\frac{3}{2}k_1}+2^{-\frac{3}{2}k_1}\Big).
\end{align*}
Summing over all frequencies and $m$ and fix $\delta_2$ sufficiently small, it follows that,
$$\sum_{m}\sum_{\operatorname{freq.}}\snorm{I_{mk_1k_2k_3k_4k}}_{L^2_\xi}\lesssim \sum_m\eps_1^5\,m^{6}\Big(t^{-(\frac56+\frac{\delta_2}{2})}
 + t^{-(\frac12+\frac{3}{2}\delta_2)}\Big)\lesssim \eps_1^5.$$
\\
\noindent
\underline{\textit{Case 3}}
\\\\
We next assume
$$
\left| \eta_j - \eta_5\right|\lesssim 2^{k_{*}} \quad \forall \ j =1,2,3,4, \quad \text{but} \quad \left| \eta_j + \eta_5\right|\gtrsim 2^{k_{*}} \quad \forall \ j =1,2,3,4. 
$$
Under this localization, it follows that
$$\left| \eta_j + \eta_5\right|\sim 2^{k_{1}} \quad \forall \ j =1,2,3,4. $$
We distinguish two cases according to the size of $\xi$.
If $|\xi| \sim 2^{k_*}$, the contribution is handled by following the proof of Lemma \ref{prop:deg_case_1}.
If $|\xi| \sim 2^{k_1}$, it follows from the \underline{\textit{Case 2}} proved above.

For $J_m$, we find that when $|\xi|\sim 2^{k}\gtrsim 2^{-100m}$ we have
\begin{align*}
    \left|\tilde{a}(\eta,\xi)\right| \lesssim 2^{-k} 2^{-k_1}, \quad \left|\tilde{m}(\eta,\xi)\right| \lesssim 2^{-k} 2^{-2k_1}.
\end{align*}
where 
$$\tilde{a}(\eta,\xi):=\frac{1}{\partial_{\eta_1}\phi},\quad \tilde{m}(\eta,\xi):=\partial_{\eta_1}\tilde{a}(\eta,\xi).$$
Therefore, after presenting integration by parts, it follows that
\begin{align*}
    \snorm{J_{mk_1k_2k_3k_4k}}_{L^{\infty}_{\xi}} \lesssim \eps_1^5\left[2^{-\frac{3}{2}k_1-\frac{k_2}{2}-\frac{k_3}{2}-\frac{k_4}{2}-m}+2^{-k_1-\frac{k_2}{2}-\frac{k_3}{2}-\frac{k_4}{2}-\frac{5m}{2}}\right].
\end{align*}
Summing over all frequencies and $m$ gives
$$\sum_{m}\sum_{\operatorname{freq.}}\snorm{J_{mk_1k_2k_3k_4k}}_{L^\infty_\xi}\lesssim \sum_{m}\eps_1^5 \cdot\left[2^{(-3m\delta_2)}+2^{m(-\frac{5}{3}-\frac{5}{2}\delta_2)}\right]\lesssim \eps_1^5$$
as desired. 
\end{proof}

%--------------------------------------------------------------%
\subsubsection{Non-degenerate resonance point}
In this section, we start to consider the non-degenerate resonance point. Here, certain null conditions from \cite{Morgan} no longer apply. 
%with the refined pointwise decay bound that we acquired in the previous section. 
Recall that we have the non-degenerate resonance points such that
$$(\sigma_1\sqrt{3},\sigma_2\sqrt{3},\sigma_3\sqrt{3},\sigma_4\sqrt{3};\sigma_{\xi}\sqrt{3})$$
where $\sigma_{j,\xi}\in \{\pm 1\}$ for $j=1,2,3,4$ such that $\sigma_\xi-\sum^4_{j=1}\sigma_j\in\{\pm 1\}$. Also we point out that, for $\sigma_5$, we have 
if $\mathcal{S}:=\sum_{j=1}^4\sigma_j$, then
$$
\sigma_5 \in \{\mathcal{S}-1, \mathcal{S}+1\}\cap\{\pm1\}.
$$
We define the threshold scale $k_*\in\mathbb{Z}$ by
\begin{equation}\label{eq:def_k_star_non_degenrate}
2^{k_*} := \max\Big\{2^k : 2^k \le 2^{-20}\,2^{m(\delta_3-\frac13)}\Big\}.
\end{equation}
where $\delta_3>0$ small enough so that we choose later.
\begin{lemma}\label{lem:taylor_non_degenerate_sum}
    In this section, the Taylor expansion of the phase function is 
    \begin{equation}\label{eq:non_degenerate_Taylor_expansion}
\phi= \frac{1}{32} \left( \sum_{j=1}^5 (\eta_j - \sigma_j\sqrt{3})^3 - (\xi - \sqrt{3})^3 \right) + \text{\{\emph{Higher Order Terms}\}}.
\end{equation}
It then follows that, the leading order term of the differences of group velocity is \begin{align}
    \partial_{\xi}\phi
=\frac{3}{32}\left[
\left(
(\xi-\sqrt{3})
-\Big((\eta_{1}+\sqrt{3})+(\eta_{2}+\sqrt{3})+(\eta_{3}-\sqrt{3})+(\eta_{4}-\sqrt{3})\Big)
\right)^{2}
-(\xi-\sqrt{3})^{2}
\right].
\end{align}
\end{lemma}
\begin{proof}
We first compute the constant term. It follows that we have 
$ \phi^{(0)}= -\omega(\sqrt{3})+\sum_{j=1}^5 \omega(\sigma_j\sqrt{3}).$
Bring the value in, we have $$\phi^{(0)} = -\frac{\sqrt{3}}{4}+\sum_{j=1}^5 \sigma_j \frac{\sqrt{3}}{4}.$$ Notice that $\sum^5_{j=1}\sigma_j=1$ which implies that $\phi^{(0)}=0.$\\
For the linear term,  we have $\omega'(\pm\sqrt{3})=-\frac{1}{8}.$ Therefore, it follows that \begin{align*}
    \phi^{(1)}&=-\omega'(\sqrt{3}) \left(\sum \eta_j - \sigma_j \sqrt{3}\right) + \sum_{j=1}^5 \omega'(\sigma_j \sqrt{3}) (\eta_j- \sigma_j \sqrt{3})\\
    &=-(-\frac{1}{8})(\sum \eta_{j}-\sigma_j\sqrt{3})+\sum (-\frac{1}{8})\eta_{j}-\sigma_j\sqrt{3}=0.
\end{align*}
For the quadratic term, it follows that we have $\omega''(\sqrt{3})=0$. Therefore,
\begin{align}
    \phi^{(2)}=-\frac{1}{2}\omega''(\sqrt{3})\left(\sum^5_{j=1} \eta_j - \sigma_j \sqrt{3}\right)^2+\sum^5_{j=1}\frac{1}{2}\omega''(\sigma_5\sqrt{3})(\eta_j-\sigma_j\sqrt{3})^2=0.
\end{align}
Finally for the cubic term, it follows that $\omega'''(\pm \sqrt{3})=\frac{3}{16}$, and thus,
\begin{align*}
\phi^{(3)}&=-\frac{1}{6}\omega'''(\sqrt{3})\left(\sum^5_{j=1} \eta_j - \sigma_j \sqrt{3}\right)^3+\sum^5_{j=1}\frac{1}{6}\omega''(\sigma_5\sqrt{3})(\eta_j-\sigma_j\sqrt{3})^3\\
&= -\frac{1}{32} \left( \sum_{j=1}^5 (\eta_j - \sigma_j\sqrt{3})^3 - (\xi - \sqrt{3})^3 \right).   
\end{align*}
The differences of group velocity follows by taking $\xi$-derivative of the phase function.
\end{proof}
\begin{lemma}\label{prop:nondeg_case_1}
Let $k_*$ be defined in \eqref{eq:def_k_star_non_degenrate}.
We define the region
$$\mathcal{U}_{1}:= \Bigl\{(\eta_1,\eta_2,\eta_3,\eta_4,\xi):|\eta_j-\sigma_j\sqrt{3}|\le 2^{k_*} \ \forall j=1,2,3,4,\ |\xi-\sigma_5\sqrt{3}|\le 2^{4k_*}\Bigr\},$$
where $\sigma_j,\sigma_5\in\{\pm1\}$ satisfy the sign constraint
$\sigma_5-\sum_{j=1}^4\sigma_j\in\{\pm1\}$. Then, under the bootstrap assumptions \eqref{eq:bootstrap_assump_XT}, the contribution of $\mathcal{U}_1$ satisfies
    \begin{align*}
        \sum\snorm{I_{m,k_*}}_{L^2_\xi}\lesssim \eps_1^5,\qquad \sum\snorm{J_{m,k_*}}_{L^\infty_\xi}\lesssim \eps_1^5.
    \end{align*}  
\end{lemma}
\begin{proof}
We proceed the proof  using the measure estimates  as  the proof of Lemma  \ref{prop:anomalous_resonance_near_resonance}. First, for the factor $\partial_\xi \phi$, we have 
\begin{align*}
    \partial_\xi \phi&\approx \omega'(\eta_5)-\omega'(\xi)\\
    &\approx \left[\omega'(\xi)+\omega''(\xi)(\eta_5-\xi)+\frac{1}{2}\omega'''(\xi)(\eta_5-\xi)\right]-\omega'(\xi).
\end{align*}
Notice that, we have $\xi\approx \sqrt{3}$ under the current setting, it follows that the leading terms is will be higher order. Therefore, we have 
$$\partial_\xi\phi\sim 2^{2k_*}.$$
It follows that we have the estimate
\begin{align*}
    \snorm{I_{m,k_*}}_{L^2_\xi}&\lesssim \int^{t_2}_{t_1} \dd \tau\; \tau\; \Big\|\varphi_{2^{4k_*}}(\xi-\sqrt{3})\omega(\xi)\int d\eta_{1234}e^{-i\tau\phi}\partial_\xi \phi \prod^{4}_{j=1}\hat{f}_{\lesssim 2^{k_*}}(\tau,\eta_j)\hat{f}(\tau,\eta_5)\Big\|_{L^2_\xi}\\
    &\lesssim \eps_1^5\int^{t_2}_{t_1} \dd \tau\ \tau \ 2^{6k_*}\sqrt{\operatorname{meas.}(|\xi-\sigma_5\sqrt{3}|)\lesssim2^{4k_*}}\lesssim \eps_1^5\cdot 2^{m(8\delta_3-\frac{2}{3})}.
\end{align*}
Fix $\delta_3$ sufficiently small and summing over $m$ yields

$$\sum_{m}\snorm{I_{m,k_*}}_{L^2_\xi}\lesssim \eps_1^5.$$
For the part of $J_m$, applying the same trick gives
$$\snorm{J_{m,k_*}}_{L^\infty_\xi}\lesssim \eps_1^5\cdot 2^{2k_*}\cdot 2^m\lesssim \eps_1^5\cdot 2^{m(4\delta_3-\frac{1}{3})}.$$
Summing over $m$ and by fixing the $\delta_3$ small enough it gives that 
$$\sum_m \snorm{J_{m,k_*}}_{L^\infty_\xi}\lesssim \eps_1^5$$
as desired.
\end{proof}
\begin{lemma}\label{prop:nondeg_case_2}
Let $k_*$ be defined in \eqref{eq:def_k_star_non_degenrate}.
We define the region
$$\mathcal{U}_{2}:= \Bigl\{(\eta_1,\eta_2,\eta_3,\eta_4,\xi):|\eta_j-\sigma_j\sqrt{3}|\le 2^{k_*} \ \forall j=1,2,3,4,\ |\xi-\sigma_5\sqrt{3}|\sim 2^{k},\ 2^{4k_*}\lesssim 2^k\ll 1\Bigr\},$$
where $\sigma_j,\sigma_5\in\{\pm1\}$ satisfies the sign constraint
$\sigma_5-\sum_{j=1}^4\sigma_j\in\{\pm1\}$.
Then, under the bootstrap assumptions \eqref{eq:bootstrap_assump_XT}, the contribution of $\mathcal{U}_2$ satisfies
    \begin{align*}
        \sum\snorm{I_{m,k_*}}_{L^2_\xi}\lesssim \eps_1^5,\qquad \sum\snorm{J_{m,k_*}}_{L^\infty_\xi}\lesssim \eps_1^5.
    \end{align*}
\end{lemma}
\begin{proof}
Again, using  the Taylor expansion from Lemma \ref{lem:taylor_non_degenerate_sum}, in this region, for  $|\partial_\xi \phi|\lesssim 2^{2k}$,  we establish the lower bound
\begin{equation}\label{eqn:non_deg_case_2_phase_bnd} 
|\phi|\gtrsim 2^{2k+k_*}. 
\end{equation}
Here, the first non-zero terms in the Taylor expansion comes from $\omega'''(\xi)$ so that we have $|\omega(\xi)|\sim \mathcal{O}(1)$.  Integration by parts in $\tau$ and an application of the measure estimate yield
\begin{align*}
\snorm{I_{mk}}_{L^2_\xi}&\lesssim\int_{t_1}^{t_2}\tau\Bigg\|\omega(\xi)\psi_k(\xi-\sigma_5\sqrt{3})\int e^{-i\tau\phi}\frac{\partial_\xi\phi}{\phi}\,\partial_\tau \widehat f_{\lesssim 2^{k_*}}(\eta_1)\prod_{j=2}^4 \widehat f_{\lesssim 2^{k_*}}(\tau,\eta_j)\,\widehat f(\tau,\eta_5)\,\mathrm d\eta\Bigg\|_{L^2_\xi}\,\mathrm d\tau \\
&\quad+ \varepsilon_1^5\int_{t_1}^{t_2}2^{m\left(4\delta_3-\frac{4}{3}\right)}2^{\frac{k}{2}}\,\mathrm d\tau+ \{\textnormal{similar terms}\}\\
&=:I_1+I_2+ \{\textnormal{similar terms}\}.
\end{align*}
For $I_1$, we define the symbol
\begin{align*}
B(\eta,\xi) &= a(\eta,\xi) \prod^{4}_{j=1}\varphi_{\lesssim 2^{k_*}}(\eta_j +\sigma_j \sqrt{3})
\end{align*}
where $a(\eta,\xi)=\frac{\partial_\xi \phi}{\phi}$ and the symbol satisfies \cite[Lemma B.3]{Morgan} with $A=2^{-k}$. Therefore, by applying mutilinear estimates it follows that 
\begin{multline*}
     \snorm{I_{mk}}_{L^2_{\xi}}
    \lesssim 2^{-k}\cdot \int_{t_1}^{t_2}d \tau \ 2^{m} \ \snorm{\partial_{\tau}f_{\lesssim 2^{k_{*}}}}_{L^2_{x}} \ \left(\eps_1 2^{-\frac{m}{3}}\right)^4 
    + 2^{-k}\cdot\eps_1^5\;\int_{t_1}^{t_2} d \tau \ 2^{m\left(4\delta_3-\frac{4}{3}\right)}2^{\frac{k}{2}} + \left\{\text{similar terms}\right\}. 
\end{multline*}
Notice that \eqref{eq:gBBM} can be rewritten as $\partial_{\tau}f \simeq u^5$, then it follows that
\begin{equation*}
     \snorm{I_{mk}}_{L^2_{\xi}}\lesssim 2^{-k}\cdot \int_{t_1}^{t_2}d \tau \ \eps_1^9 2^{-\frac{5}{3}m} +\eps_1^5 2^{m\left(4\delta_3-\frac{4}{3}\right)}2^{\frac{k}{2}}.
\end{equation*}
Summing over the frequency such that $4k_*\lesssim k\ll 0$ and $m$ yields
\begin{equation*}
    \sum_m\sum_{\operatorname{freq.}}\snorm{I_{m,k}}_{L^2_\xi}\lesssim \eps_1^9 \, t^{-\delta_3-\frac{1}{3}}+\eps_1^4 2^{4\delta_3 - \frac{1}{3}}
\end{equation*}
which is the desired bound by fixing $\delta_3$ sufficiently small.

For $J_m$, we have
\begin{equation*}
    \snorm{J_{mk}}_{L^\infty_\xi}\lesssim \int_{t_1}^{t_2}d \tau \ \Bigg\|\omega(\xi)\psi_{k}\left(\xi-\sigma_5 \sqrt{3}\right)\int d\eta_{1234} \ e^{-i\tau\phi}\prod^4_{j=1}\widehat{f}_{\lesssim 2^{k_{*}}} \left(\tau,\eta_{j}\right)
    \widehat{f}\left(\tau,\eta_5\right)\Bigg\|_{L^2_\xi}.
\end{equation*}    
Therefore, integrating by parts in $\tau$, and using the bound
$$\left|\omega(\xi)\frac{e^{-i\tau\phi}}{\phi}\right|\lesssim 2^{-2k-k_*}$$
together with the measure estimate, we obtain
\begin{align*}
     \snorm{J_{mk}}_{L^\infty_\xi}&\lesssim \eps_1^5\cdot 2^{k_*}+2^{-2k-k_*}
     \cdot\int_{t_2}^{t_1} \eps_1^9 \cdot \tau^{-\frac{8}{3}}\mathrm{d}\tau \lesssim \eps_1^5 2^{m(4\delta_3-\frac{1}{3})}+\eps_1^9 2^{-\frac{2}{3}m}.
\end{align*}
Summing over $m$ and fixing $\delta_3$ small gives
$$\sum_{m} \snorm{J_{mk}}_{L^\infty_\xi} \lesssim \eps_1^5$$
which is the desired bound.
\end{proof}
\begin{lemma}\label{prop:non_deg_case_3}
Let $k_*$ be defined as in \eqref{eq:def_k_star_non_degenrate}.
We define the region
\begin{multline*}
\mathcal{U}_{3}:=\Bigl\{(\eta_1,\eta_2,\eta_3,\eta_4,\xi):|\eta_j-\sigma_j\sqrt{3}|\sim 2^{k_j},\ 2^{k_*}\lesssim 2^{k_j}\ll 1 \ \forall j,\ k_1\ge k_2\ge k_3\ge k_4,\\
|\xi-\sigma_5\sqrt{3}|\sim 2^{k},\
2^{4k_*}\lesssim 2^{k}\ll 1,\
2^{k}\gg 2^{4k_1}
\Bigr\},
\end{multline*}
where $\sigma_j,\sigma_5\in\{\pm1\}$ satisfy the sign constraint
$\sigma_5-\sum_{j=1}^4\sigma_j\in\{\pm1\}$.
Then, under the bootstrap assumptions \eqref{eq:bootstrap_assump_XT}, the contribution of $\mathcal{U}_3$ satisfies
    \begin{align*}
        \sum\snorm{I_{m,k_*}}_{L^2_\xi}\lesssim \eps_1^5,\qquad \sum\snorm{J_{m,k_*}}_{L^\infty_\xi}\lesssim \eps_1^5.
    \end{align*}    
\end{lemma}
\begin{proof}
In this region, Lemma \ref{lem:taylor_non_degenerate_sum} implies that we have $|\phi|\gtrsim 2^{2k+k_*}$ and $|\partial_\xi \phi|\lesssim 2^{2k}$. Also, in this region, the frequency cut-off gives a stronger decay followed by \eqref{frequencybounds}. We thus have 
$$\snorm{u_{k_j}}_{L^\infty_x}\lesssim \eps_1 2^{-\frac{k_j}{2}-\frac{m}{2}},\qquad j=1,2,3,4.$$
With the stronger decay, the rest of the proof follows from the proof of Lemma  \ref{prop:nondeg_case_2}.
\end{proof}
\begin{lemma}\label{prop:non_deg_case_4}
Let $k_*$ be defined as in \eqref{eq:def_k_star_non_degenrate}.
We define the region
\begin{multline*}
\mathcal{U}_{4}:=\Bigl\{(\eta_1,\eta_2,\eta_3,\eta_4,\xi):|\eta_j-\sigma_j\sqrt{3}|\sim 2^{k_j},\ 2^{k_*}\lesssim 2^{k_j}\ll 1 \ \forall j,\ k_1\gg k_2\ge k_3\ge k_4,\\
|\xi-\sigma_5\sqrt{3}|\sim 2^{k},\
2^{4k_*}\lesssim 2^{k}\ll 1,\
2^{k}\lesssim 2^{4k_1} \text{ or }|\xi|\lesssim 2^{4k_*}
\Bigr\},
\end{multline*}
where $\sigma_j,\sigma_5\in\{\pm1\}$ satisfy the sign constraint
$\sigma_5-\sum_{j=1}^4\sigma_j\in\{\pm1\}$.
Then, under the bootstrap assumptions, the contribution of $\;\mathcal{U}_4$ satisfies
    \begin{align*}
\sum\snorm{I_{m,k_*}}_{L^2_\xi}\lesssim \eps_1^5,\qquad\sum\snorm{J_{m,k_*}}_{L^\infty_\xi}\lesssim \eps_1^5.
    \end{align*}

\end{lemma}
\begin{proof}
In this region, Lemma \ref{lem:taylor_non_degenerate_sum} will not be fully canceled. And, the frequency cut-off gives a stronger decay followed by \eqref{frequencybounds}. We thus have 
$$\snorm{u_{k_j}}_{L^\infty_x}\lesssim \eps_1 2^{-\frac{k_j}{2}-\frac{m}{2}},\qquad j=1,2,3,4.$$
With the stronger decay, the rest of the proof follows from the proof of Lemma \ref{prop:nondeg_case_2}.
\end{proof}

\begin{proposition}\label{prop:non_deg_case_5}
Let $\;\mathcal{U}_5$ be the frequency region defined by the following conditions:
\begin{itemize}
    \item The input frequencies satisfy
    $$ |\eta_j-\sigma_j\sqrt{3}| \sim 2^{k_j}, \qquad 2^{k_*} \lesssim 2^{k_j} \ll 1, \qquad j=1,2,3,4,$$
    with the relation
    $$  k_1 \approx k_2 \approx k_3 \approx k_4.$$
    \item The output frequency satisfies either
    $$  |\xi-\sigma_5\sqrt{3}| \sim 2^{k},\quad 2^{4k_{*}}\lesssim 2^{k} \ll 1, \quad 2^{k} \lesssim 2^{4k_1}  \quad \text{or} \quad |\xi-\sigma_5\sqrt{3}| \lesssim 2^{4k_{*}}.$$
\end{itemize}
Then, under the bootstrap assumptions, the contribution of $\mathcal{U}_5$ satisfies
    \begin{align}
        \sum\snorm{I_{m,k_*}}_{L^2_\xi}\lesssim \eps_1^5\label{eq:goal_nondeg_case_5_I}\\
        \sum\snorm{J_{m,k_*}}_{L^\infty_\xi}\lesssim \eps_1^5.\label{eq:goal_nondeg_case_5_J}
    \end{align}
\end{proposition}
\begin{proof}
    In the $\mathcal{U}_5$, we will integrate by parts in $\partial_{\eta_1}$. By Lemma \ref{lem:taylor_non_degenerate_sum}, it follows that several combinations of symbols might occur. Without loss of generality, we only consider the combination of 
    $$(\sigma_1, \sigma_2,\sigma_3,\sigma_4;\sigma_\xi)=(-1,1,1,1;1)$$
    and
    $$(\sigma_1, \sigma_2,\sigma_3,\sigma_4;\sigma_\xi)=(-1,-1,1,1;\pm1).$$
    For $\rho_j:=\sigma_j\sigma_5\in \{\pm 1\}$, we have
    \begin{align*}
        |\partial_{\eta_j} \phi|&\gtrsim |\eta_j-\rho_j\,\eta_5|\left|\eta_j+\rho_j\,\eta_5-\sigma_j\,2\sqrt{3}\right|,\qquad j=1,2,3,4.
    \end{align*}
    However, a secondary localization is needed in the case when $\sigma_1=-1$, $\sigma_j=1$ for $j=2,3,4$. That is, we are localizing at dyadic scale $2^{\ell_j}$ with
    $$2^{k_*}\lesssim 2^{\ell_j}\ll 1.$$
    More precisely, for each $j$ we assume either
    $$\big|\eta_j+\sigma_j\eta_5\big| \lesssim 2^{k_*}\quad \text{or}\quad \big|\eta_j+\sigma_j\eta_5\big|\sim 2^{\ell_j},$$
    and either
    $$\big|\eta_j-\sigma_j\eta_5-\sigma_j(2\sqrt3)\big|\lesssim 2^{k_*}\quad \text{or}\quad \big|\eta_j-\sigma_j\eta_5-\sigma_j(2\sqrt3)\big|\sim 2^{\ell_j}.$$
    Therefore, the proof is reduced to Lemmas 
\ref{lem:nondeg_case_5_subcase_1}--\ref{lem:nondeg_case_5_subcase_4},  established below.
\end{proof}
\begin{lemma}\label{lem:nondeg_case_5_subcase_1}
    Under the setting of $\;\mathcal{U}_5$ we defined in Proposition \ref{prop:non_deg_case_5}, and we further assume that
    $$\left|\eta_j+\sigma_j\eta_5\right|\lesssim 2^{k_*},\quad \text{and}\quad \big|\eta_j-\sigma_j\eta_5-\sigma_j(2\sqrt3)\big|\lesssim 2^{k_*}, $$
     for all $j$. Then, \eqref{eq:goal_nondeg_case_5_I} and \eqref{eq:goal_nondeg_case_5_J} hold.
\end{lemma}
\begin{proof}
This implies that 
     $$|\xi-\sigma_5\sqrt{3}|\lesssim 2^{k_*}\quad\text{and}\quad |\eta_j-\sigma_j\sqrt{3}|\lesssim 2^{k_*}.$$
     Hence, the proof follows exactly from  Lemma \ref{prop:nondeg_case_1}.
\end{proof}
\begin{lemma}\label{lem:nondeg_case_5_subcase_2}
    Under the setting of $\;\mathcal{U}_5$ which we defined in the Proposition \ref{prop:non_deg_case_5}, and we further assume that
	$$|\eta_1+\eta_5|\sim 2^{\ell_1}\quad\text{and}\quad \big|\eta_1-\eta_5+2\sqrt{3}\big|\sim 2^{\ell_1}.$$
    Then, \eqref{eq:goal_nondeg_case_5_I} and \eqref{eq:goal_nondeg_case_5_J} holds.
\end{lemma}
\begin{proof}
	We divide the argument into two cases.
	First, consider $2^{\ell_1}\gtrsim 2^{k_1}$ which together with the Taylor expansion in Lemma \ref{lem:taylor_non_degenerate_sum}, implies that 
	$$|\partial_{\eta_1}\phi|\gtrsim 2^{2\ell_1}.$$
	The above lower bound yields that
	$$|a(\eta,\xi)|\lesssim 1 \quad \text{and}\quad |m(\eta,\xi)|\lesssim 2^{-\ell_1}$$
	where $a(\eta,\xi):=\frac{\partial_\xi \phi}{\partial_{\eta_1}\phi}$ and $m(\eta,\xi):=\partial_{\eta_1}a(\eta,\xi)$. We therefore define the symbol
	\begin{align*}
	B_1(\eta,\xi)&=a(\eta,\xi)\psi_k\; \psi_{k_1}(\eta_1+\sqrt{3}) \psi_{k_2}(\eta_2-\sqrt{3}) \psi_{k_3}(\eta_3-\sqrt{3}) \psi_{k_4}(\eta_4-\sqrt{3})\\
	B_2(\eta,\xi)&=m(\eta,\xi)\psi_k\; \psi_{k_1}(\eta_1+\sqrt{3}) \psi_{k_2}(\eta_2-\sqrt{3}) \psi_{k_3}(\eta_3-\sqrt{3}) \psi_{k_4}(\eta_4-\sqrt{3}).
	\end{align*}
	 The symbols $B_1(\eta,\xi)$ and $B_2(\eta,\xi)$ satisfies \cite[Lemma B.3]{Morgan} with $A=1, 2^{-\ell_1}$ respectively.  Integrating by parts in $\eta_1$ and by multilinear estimate with using \cite[Proposition B.1]{Morgan}, it follows that
	   \begin{align*}
	       \snorm{I_{mk_1k_2k_3k_4k}}_{L^2_\xi}&\lesssim \eps_1^5\int^{t_2}_{t_1} d\tau\;2^{-\ell_1}\; 2^{-\frac{1}{2}(\ell_1+\ell_2+\ell_3+\ell_4)}2^{-2m}+2^{-\frac{1}{2}(\ell_2+\ell_3+\ell_4)}2^{-\frac{3}{2}m}\\
           &\lesssim \eps_1^5 \left(2^{-3\delta_3m}+2^{-\frac{3}{2}\delta_3 m}\right).
	   \end{align*}
     Summing over $m$ gives
     $$\sum_m\sum_{\operatorname{freq.}}\snorm{I_{mk_1k_2k_3k_4k}}_{L^2_\xi}\lesssim \eps_1^5. $$
     For $J$, it follows by the same strategy. We have,
     $$|\tilde{a}(\eta,\xi)|\lesssim 2^{-3\ell_1},\quad |\tilde{m}(\eta,\xi)|\lesssim 2^{-2\ell_1}$$
     where $\tilde{a}(\eta,\xi):=\frac{\partial^2_{\eta_1}}{(\partial_{\eta_1}\phi)^2}$ and $\tilde{m}(\eta,\xi):=\frac{1}{\partial_{\eta_1}\phi}$. Integration by parts in $\eta_1$ direction and using multilinear estimates with \cite[Proposition B.6]{Morgan} yield 
     \begin{align*}
         \snorm{J_{mk_1k_2k_3k_4k}}_{L^\infty_\xi}&\lesssim \eps_1^5\int^{t_2}_{t_1} d\tau\;\tau^{-1} \tau^{-\frac{3}{2}}\;2^{-3\ell_1}2^{\frac{1}{2}(\ell_1-\ell_2-\ell_3-\ell_4)}+\tau^{-1}2^{-2\ell_1}2^{-\frac{1}{2}(\ell_2+\ell_3+\ell_4)}\tau^{-\frac{3}{2}}\\
         &\lesssim \eps_1^5\left(2^{m(-4\delta_3-\frac{1}{6})}+2^{m(-\frac{7}{2}\delta_3-\frac{1}{3})}\right).
     \end{align*}
     Summing over $m$ and frequency gives
     $$\sum_m\sum_{\operatorname{freq.}}\snorm{J_{mk_1k_2k_3k_4k}}_{L^\infty_\xi}\lesssim \eps_1^5. $$

     We next assume that $2^{\ell_1}\ll 2^{k_1}$. This implies that either 
     \begin{equation}\label{eq:prop:U5_case_2_case_2_relation_1}
        |\eta_1+\eta_5|\sim 2^{k_1}
     \end{equation}
     holds or
     \begin{equation}\label{eq:prop:U5_case_2_case_2_relation_2}
     |\eta_1-\eta_5+2\sqrt{3}|\sim 2^{k_1}
     \end{equation}
     holds. However, \eqref{eq:prop:U5_case_2_case_2_relation_1} and \eqref{eq:prop:U5_case_2_case_2_relation_2} cannot be true simultaneously. To be more specific, if we have $|\eta_1+\eta_5|\sim 2^{\ell_1}$, it follows that 
     \begin{align*}
         |\eta_1-\eta_5+2\sqrt{3}|&=|\eta_1-\left(-\eta_1+\mathcal{O}(2^{\ell_1})\right)+2\sqrt{3}|\\
         &=|2(\eta_1+\sqrt{3})-\mathcal{O}(2^{\ell_1})|\sim 2^{k_1}.
     \end{align*}
    Therefore, we have $|\partial_{\eta_1}\phi|\gtrsim 2^{k_1+\ell_1}$. The proof thus reduces to the proof of the first case . 
    \end{proof}
\begin{lemma}\label{lem:nondeg_case_5_subcase_3}
   Under the setting of $\;\mathcal{U}_5$ we defined in Proposition \ref{prop:non_deg_case_5}, and we further assume that
    $$|\eta_1+\eta_5|, |\eta_2+\eta_5|, |\eta_3-\eta_5|,|\eta_4-\eta_5|\lesssim 2^{k_*}.$$
    Then, \eqref{eq:goal_nondeg_case_5_I} and \eqref{eq:goal_nondeg_case_5_J} holds.
\end{lemma}
\begin{proof}
    In this case, we are dealing 
    $$|\eta_1+\eta_5|, |\eta_2+\eta_5|, |\eta_3-\eta_5|,|\eta_4-\eta_5|\lesssim 2^{k_*}.$$ It implies that we have
    $$\left|\eta_1-\eta_5+2\sqrt{3}\right|=\mathcal{O}(2^{k_1})$$
    and the bound
    $$|\partial_{\eta_1}\phi|\gtrsim 2^{2k_1}.$$
    Thus, the proof reduces to follow the first case in Lemma \ref{lem:nondeg_case_5_subcase_2}.
\end{proof}
\begin{lemma}\label{lem:nondeg_case_5_subcase_4}
   Under the setting of $\;\mathcal{U}_5$ we defined in Proposition \ref{prop:non_deg_case_5}, and we further assume that
    $$|\eta_1-\eta_5|, |\eta_2+\eta_5|, |\eta_3+\eta_5|,|\eta_4+\eta_5|\lesssim 2^{k_*}$$
    or
    $$|\eta_1-\eta_5|, |\eta_2-\eta_5|, |\eta_3+\eta_5|,|\eta_4+\eta_5|\lesssim 2^{k_*}.$$
    Then, \eqref{eq:goal_nondeg_case_5_I} and \eqref{eq:goal_nondeg_case_5_J} holds.
\end{lemma}
\begin{proof}
   The two specific cases are happening because we have a different choice of $\sigma_\xi$.   To be more specific, we have $\sigma_1\neq \sigma_5$ and $\sigma_3=\sigma_5$ here. However, this yields a bound with the degeneration possibility 
    $$|\partial_{\eta_1}\phi|\gtrsim 2^{k_1}|\eta_1-\eta_5|$$
  and therefore, integration by parts in $\eta_1$ will have the possibility of blowing-up. Instead, we perform the integration by parts in $\eta_3$. Then, by noticing that $$|\eta_3-\eta_5-2\sqrt{3}|\gtrsim 2^{k_1}$$
  the proof reduces to Lemma \ref{lem:nondeg_case_5_subcase_2}
\end{proof}
%------------------------------------------------------------------------%
 \subsubsection{space--time resonant line}
We consider the contribution of the space--time resonant line
    \begin{equation*}L
        =\Big\{\,(\sigma_1\eta,\sigma_2\eta,\sigma_3\eta,\sigma_4\eta;\ \xi=\sigma_\xi\eta)\ :\ \eta\in\mathbb R,\ \sigma_j\in\{\pm1\},\ \sigma_\xi-\!\sum_{j=1}^4\sigma_j=\pm1 \,\Big\}\subseteq \mathbb{R}^5,
    \end{equation*}
 including the point $(0,0,0,0;0)$ and, when $|\eta|=\sqrt{3}$, \emph{all} sign–permutation points of the form
$$
(\eta_1,\eta_2,\eta_3,\eta_4;\xi)
= (\sigma_1 \eta,\sigma_2 \eta,\sigma_3 \eta,\sigma_4 \eta;\ \xi=\sigma_\xi \eta),
\quad \sigma_j\in\{\pm1\},
$$
that satisfy the resonance sign constraint
$$
\sigma_\xi-\sum_{j=1}^4 \sigma_j \in \{\pm1\}
$$
(equivalently, among $\eta_1,\eta_2,\eta_3,\eta_4,\eta_5$
there are exactly $3$ $+\eta$’s and $2$ $-\eta$’s, or vice versa). The degenerate contributions corresponding to $\eta\approx0$ and $\eta\approx\pm\sqrt{3}$ were discussed in the previous section.
We therefore restrict to the case $|\eta|\gg1$, where
degeneracy at infinity must still be controlled. The resonant line is given by the linear constraints
$$\eta_1-\eta_2=0,\qquad \eta_3-\eta_4=0,\qquad \eta_1+\eta_3=0.$$
Fix a dyadic time scale $2^m$. We define the lower and upper threshold scales
$k_{\operatorname{lo}},k_{\operatorname{hi}}\in\mathbb Z$ by
\begin{equation}\label{eq:def_k_lo}
2^{k_{\operatorname{lo}}}:= \max\Bigl\{2^k : 2^k \le C_{\operatorname{lo}}\,2^{m(\delta_2-\frac13)}\Bigr\},
\end{equation}
and
\begin{equation}\label{eq:def_k_hi}
2^{k_{\operatorname{hi}}}
:= \min\Bigl\{2^k : 2^k \ge C_{\operatorname{hi}}\,2^{m(\frac19-\delta_4)}\Bigr\},
\end{equation}
where $\delta_2>0$ is fixed as in the analysis of the $(0,0,0,0;0)$ resonance, and
$\delta_4>0$ will be chosen sufficiently small.  Having excluded the degenerate cases $\eta=0$ and $\eta=\pm\sqrt{3}$, it suffices to estimate the contribution from the high--frequency portion of
the resonant line, which we denote by
\begin{multline*}
    L_{\operatorname{hi}} = \Big\{ 
(\eta_1, \eta_2, \eta_3, \eta_4; \xi) \in \mathbb{R}^5 \ : \ 
2^5 \le |\eta_1| \le 2^{k_{hi}}, \ 
|\eta_1| \sim |\eta_2| \sim |\eta_3| \sim |\eta_4| \sim |\xi|, \\ 
\max(|\eta_1 - \eta_2|, |\eta_3 - \eta_4|, |\eta_1 + \eta_3|, |\xi - \eta_1|) \le 2^{k_{\operatorname{lo}}} 
\Big\}.
\end{multline*}
We further decompose $L_{\operatorname{hi}}$ into dyadic pieces
$$L_{\operatorname{hi}, k, k_1, \ell_1,\ell_2,\ell_3} = \left\{ 
\begin{aligned}
& |\eta_1| \sim 2^{k_1} \in [2^5, 2^{k_{\operatorname{hi}}}], \\
& |\eta_1 - \eta_2| \le 2^{k_{\operatorname{lo}}} \text{ or } \sim 2^{\ell_1}, \\
& |\eta_3 - \eta_4| \le 2^{k_{\operatorname{lo}}} \text{ or } \sim 2^{\ell_2}, \\
& |\eta_1 + \eta_3| \le 2^{k_{\operatorname{lo}}} \text{ or } \sim 2^{\ell_3}, \\
& |\xi - \eta_1| \le 2^{k_{\operatorname{lo}}} \text{ or } \sim 2^k 
\end{aligned}
\right\}.$$
The desired bound then follows by summing over $k,k_1,\ell_1,\ell_2,\ell_3$ and $m$. 
\begin{lemma}\label{lem:taylor_STL}
    In this subsubsection, for $\tilde{\eta}\in \R$, we have the Taylor expansion for the phase function about the space time resonance $(\eta_1,\eta_2,\eta_3,\eta_4;\xi)=(\tilde{\eta},\tilde{\eta},-\tilde{\eta},-\tilde{\eta};\tilde{\eta})$ with $|\eta|\gg1$ 
    \begin{multline}\label{eq:taylor_STL}
        \phi \approx  \left(-\omega'(\tilde{\eta}) + \omega'(\tilde{\eta})\right)(\xi - \tilde{\eta}) + \frac{1}{2}\omega''(\tilde{\eta}) \bigg[ (\eta_1-\tilde{\eta})^2 + (\eta_2-\tilde{\eta})^2 - (\eta_3+\tilde{\eta})^2 - (\eta_4+\tilde{\eta})^2 \\
 + \left((\xi - \tilde{\eta}) - \sum (\eta_j \pm \tilde{\eta})\right)^2 - (\xi - \tilde{\eta})^2 \bigg].
    \end{multline}
    For $|\tilde{\eta}|\gtrsim 2^5$, \eqref{eq:taylor_STL} reduces to 
    $$\phi \approx \frac{1}{2}\omega''(\tilde{\eta}) \left[ \mathcal{O}\left((\eta_j\pm \tilde{\eta})^2\right) + \mathcal{O}\left((\xi-\tilde{\eta})\cdot (\sum^4_{j=1}\eta_j)\right) \right].$$
\end{lemma}
This implies that if $k\gtrsim k_{\text{min}}$ and $\ell_1, \ell_2, \ell_3 \gtrsim \ell_{\text{min}}$, then $k_{\text{min}} =  \ell_{\text{min}}$.  As in the degenerate point analysis, to combat the loss of dispersion near $\xi, \eta \to \infty$, we choose$$\ell_{\text{min}} = k_{\operatorname{lo}} \approx -\frac{m}{3}.$$
We impose, without loss of generality, the ordering
$\ell_1 \ge \ell_2 \ge \ell_3$.
\begin{proposition}\label{prop:STR_line_case_1}
Let $k_{\operatorname{lo}}$ be defined in \eqref{eq:def_k_lo}.
We define the region
$$\mathcal{L}_{1}:= \Bigl\{(\eta_1,\eta_2,\eta_3,\eta_4,\xi):\left|\eta_1-\eta_2\right|, \left|\eta_3-\eta_4\right|, \left|\eta_1+\eta_3\right| \lesssim 2^{k_{\operatorname{lo}}},\ |\xi-\eta_1| \lesssim 2^{k_{\operatorname{lo}}}.\Bigr\},$$
Then, under the bootstrap assumptions, the contribution of $\mathcal{L}_1$ satisfies
    \begin{align*}
        \sum\snorm{I_{m,k_*}}_{L^2_\xi}\lesssim \eps_1^5,\qquad \sum\snorm{J_{m,k_*}}_{L^\infty_\xi}\lesssim \eps_1^5.
    \end{align*}
\end{proposition}
\begin{proof}
Since the integration by parts is not allowed in this region. We therefore
introduce the symbol
$$
 B(\eta,\xi) = \partial_{\xi}\phi \ \varphi_{\lesssim k_{\operatorname{lo}}}(\xi-\eta_1) \ \varphi_{\lesssim k_{\operatorname{lo}}}(\eta_1-\eta_2)\ \varphi_{\lesssim k_{\operatorname{lo}}}(\eta_3-\eta_4) \  \varphi_{\lesssim k_{\operatorname{lo}}}(\eta_1+\eta_3) \prod^4_{j=1} \psi_{k_{j}}(\eta_j).
$$
Notice that on the resonant line $\xi=\eta_1=\eta_2=-\eta_3=-\eta_4$, we have $$\partial_\xi \phi = -\omega'(\xi) + \omega'(\eta_5) = 0.$$ Also, $\omega''(\eta) \sim \eta^{-3}$ and therefore we have the smallness, $$|\partial_\xi \phi| \approx |\omega''(\eta_1)| \cdot |\eta_5 - \xi| \lesssim 2^{-3k_1} \cdot 2^{k_{\operatorname{lo}}}.$$ Therefore, it satisfies the \cite[Lemma B.3]{Morgan} and \cite[Lemma B.4]{Morgan} with $A= 2^{-3k_1} \cdot 2^{k_{\operatorname{lo}}}$. By multilinear estimate it follows that
$$\begin{aligned}
\snorm{I_{m, k_1}}_{L^2_\xi} &\lesssim \eps_1^5 2^{k_{1}}\int_{t_1}^{t_2} d\tau \, \tau \cdot \left( 2^{-3k_1} 2^{k_{\operatorname{lo}}} \right) \cdot \snorm{f_{k_1}}_{L^2} \cdot \left( \eps_1 \tau^{-\frac{1}{2}} 2^{\frac{3}{2}k_1} \right)^4\\
&\lesssim \eps_1^5 2^{k_{\operatorname{lo}}}\int_{t_1}^{t_2} d\tau \, 2^m \cdot 2^{k_{\operatorname{lo}}} \cdot 2^{-3k_1} \cdot 2^{-2m} \cdot 2^{6k_1} 
\end{aligned}$$
Summing over the frequencies and $m$ yields
$$
\sum_{m} \sum_{\text{freq.}} \snorm{I_{m,k_1}}_{L^2_{\xi}}\lesssim \eps_1^5\cdot 2^{m(-\frac{1}{9}+\delta_2-2\delta_4)}. 
$$
Fix $\delta_2-2\delta_4<\frac{1}{3}$, we then have $\sum_{m} \sum_{\text{freq.}} \snorm{I_{m,k_1}}_{L^2_{\xi}}\lesssim \eps_1^5$. For $J_m$ part, by measure estimates it follows that 
$$\begin{aligned}
\snorm{J_{m}}_{L^\infty_\xi} &\lesssim 2^{-k_1} \cdot \eps_1^5 \int_{t_1}^{t_2} \text{meas}.(\tau) \, d\tau \lesssim \eps_1^5 2^{-k_1} \int_{t_1}^{t_2} \tau^{-\frac{4}{3} + 4\delta_2} \, d\tau
\end{aligned}$$
which gives
$$\|J_{m}\|_{L^\infty_\xi} \lesssim \eps_1^5 \cdot 2^{-k_1} \cdot 2^{m(-\frac{1}{3} + 4\delta_2)}$$
Summing up over the frequencies and  $m$ yields the  bound
$$\sum_{m} \|J_{m}\|_{L^\infty_\xi} \lesssim \eps_1^5$$as desired.
\end{proof}
\begin{proposition}\label{prop:STR_line_case_2}
Let $\mathcal{L}_2$ be the frequency region defined by the following conditions:
\begin{itemize}
    \item The input frequencies satisfy
\begin{align*}
  \left|\eta_1-\eta_2\right| &\sim 2^{\ell_1}, \quad  \ 2^{k_{\operatorname{lo}}} \lesssim 2^{\ell_1}<2^{-10},
  \\
  \left|\eta_3-\eta_4\right| &\sim 2^{\ell_2}, \quad  \ 2^{k_{\operatorname{lo}}} \lesssim 2^{\ell_2}<2^{-10},
  \\
  \left|\eta_1+\eta_3\right| &\sim 2^{\ell_3}, \quad  \ 2^{k_{\operatorname{lo}}} \lesssim 2^{\ell_3}<2^{-10},
\end{align*}
or
\begin{align*}
    \left|\eta_1-\eta_2\right| &\lesssim 2^{k_{\operatorname{lo}}},\\
    \left|\eta_3-\eta_4\right| &\lesssim 2^{k_{\operatorname{lo}}},\\
     \left|\eta_1+\eta_3\right| &\lesssim 2^{k_{\operatorname{lo}}}.
\end{align*}
    \item The output frequency satisfies 
    $$ |\xi-\eta_1|\sim 2^{k}, \quad 2^{k_{\operatorname{lo}}} \lesssim 2^{k}<2^{-10} \quad k \gg \ell_1.$$
\end{itemize}
Then, under the bootstrap assumptions, the contribution of $\mathcal{L}_2$ satisfies
    \begin{align*}
        \sum\snorm{I_{m,k_*}}_{L^2_\xi}\lesssim \eps_1^5,\qquad\sum\snorm{J_{m,k_*}}_{L^\infty_\xi}\lesssim \eps_1^5.
    \end{align*}
\end{proposition}
\begin{proof}
It suffices to treat the configuration
$|\eta_3-\eta_4|\sim 2^{\ell_2}\gtrsim 2^{k_{\operatorname{lo}}}$, since the remaining cases can be handled by simpler variants of the same argument. Under the separation $k\gg \ell_j$, the phase satisfies
$$|\phi| \approx \frac{1}{2}|\omega''(\eta_1)|(\xi - \eta_1)^2 \sim 2^{-3k_1} 2^{2k}.$$
so that integration by parts in time applies and yields the bound of the multiplier
$$\left| \frac{1}{\phi} \right| \lesssim 2^{3k_1} 2^{-2k}.$$
We have
\begin{equation*}
    I_{m, k} = \omega(\xi)\int_{t_1}^{t_2} d\tau \ \tau \int d\eta_{1234}\, e^{i\tau\phi} \, \partial_\xi \phi \prod^4_{j=1}\hat{f}_{k_j}(\tau,\eta_j)\hat{f}_{\sim 2^k}(\xi-\eta_1) .
\end{equation*}
and integration by parts in $\tau$ gives the boundary term $I_1$ without the integral with respect to time and the bulk term $I_2$ that the time derivative hits the nonlinear terms. Here in the setting, we have $k\gg k_{\operatorname{lo}}$ which implies that the worst case case would be when $\ell$ lies in the resonant boundary $\ell\sim k_{\operatorname{lo}}$. Hence, $$  k \sim k_{\operatorname{lo}} \implies 2^{-2k} \sim 2^{-2k_{lo}} \approx (t^{-\frac{1}{3}})^{-2} = t^{2/3}.$$
Therefore, we have for the boundary term,
$$\snorm{I_1}_{L^2_\xi} \lesssim \eps_1^5 2^{-k_1}\cdot 2^{3k_1}\cdot 2^{-2k} \cdot 2^{-\frac{4}{3}m}\lesssim \eps_1^5 2^{m(\frac{2}{9}-2\delta_4)}\cdot 2^{\frac{2}{3}m}\cdot 2^{-\frac{4}{3}m}\lesssim 2^{-m(\frac{4}{9}+2\delta_4)}.$$ 
Summing over frequencies and $m$ yields
$$\sum_m \snorm{I_1}_{L^2_\xi}\lesssim \eps_1^5 t^{-\frac{4}{9}+2\delta_4}.$$
The desired bound follows by fixing $2\delta_4<\frac{4}{9}$. For the bulk term, we have
$$\snorm{I_2}_{L^2_\xi}\lesssim \eps^9_1\cdot 2^{-k_1}\int^{t_2}_{t_1} d\tau\  2^{3k_1}2^{-2k}\cdot 2^{-\frac{8}{3}m}\lesssim \eps^9_1 2^{-m(\frac 79-2\delta_4)}.$$
Summing over frequencies and $m$ yields
$$\sum_m \snorm{I_2}_{L^2_\xi}\lesssim \eps_1^5 t^{-\frac{7}{9}+2\delta_4}$$
and fix $2\delta_4 <\frac{7}{9}$ to acquire the desired bound. 

For the weighted norm term $J$, we have the multiplier
$$\frac{\partial_\xi \phi}{\phi} \sim \frac{2^{-3k_1} 2^k}{2^{-3k_1} 2^{2k}} = 2^{-k}$$
and it follows that
$$\snorm{J_{m,k}}_{L^\infty_\xi}\lesssim \eps_1^5 2^{\frac{1}{3}m}\cdot 2^m \cdot 2^{-\frac{4}{3}m}\lesssim \eps_1^5$$
which is the desired bound.
\end{proof}
\begin{proposition}\label{prop:STR_line_case_3}
Let $\mathcal{L}_3$ be the frequency region defined by the following conditions:
\begin{itemize}
    \item The input frequencies satisfy
\begin{align*}
  \left|\eta_1-\eta_2\right| &\sim 2^{\ell_1}, \quad  \ 2^{k_{\operatorname{lo}}} \lesssim 2^{\ell_1}<2^{-10},
  \\
  \left|\eta_3-\eta_4\right| &\sim 2^{\ell_2}, \quad  \ 2^{k_{\operatorname{lo}}} \lesssim 2^{\ell_2}<2^{-10},
  \\
  \left|\eta_1+\eta_3\right| &\sim 2^{\ell_3}, \quad  \ 2^{k_{\operatorname{lo}}} \lesssim 2^{\ell_3}<2^{-10}.
\end{align*}
    \item The output frequency satisfies 
    $$ |\xi-\eta_1|\sim 2^{k}, \quad 2^{k_{\operatorname{lo}}} \lesssim 2^{k}<2^{-10} \quad 2^{k}\lesssim 2^{\ell_3},$$
    or
    $$|\xi-\eta_1|\lesssim 2^{k_{\operatorname{lo}}}.$$
\end{itemize}
Then, under the bootstrap assumptions, the contribution of $\mathcal{L}_3$ satisfies
    \begin{align*}
        \sum\snorm{I_{m,k_*}}_{L^2_\xi}\lesssim \eps_1^5,\qquad
        \sum\snorm{J_{m,k_*}}_{L^\infty_\xi}\lesssim \eps_1^5.
    \end{align*}
\end{proposition}
\begin{proof}
It suffices to treat the case $|\xi|\sim 2^k$, the remaining cases following
by analogous arguments. By the mean value theorem and the bound $|\omega''(\eta_1)|\sim 2^{-3k_1}$, we obtain $|\partial_{\eta_1-\eta_2}\phi| \gtrsim 2^{-3k_1}2^{\ell_1}$. It follows that 
    \begin{align}
    \label{eq:STR_line_case_3_a_m_bound}
        \left|a(\eta,\xi)\right|\lesssim 2^{3k_1-\ell_1}, \quad \text{and} \quad 
         \left|m(\eta,\xi)\right|\lesssim 2^{3k_1-2\ell_1},
    \end{align}
    where
    \begin{equation*}
        a(\eta,\xi):=\frac{\partial_{\xi}\phi}{\partial_{\eta_1-\eta_2}\phi},\qquad m(\eta,\xi):=\partial_{\eta_1-\eta_2}a(\eta,\xi).
    \end{equation*}
Notice that these satisfy \cite[Lemma B.3]{Morgan} and \cite[Lemma B.4]{Morgan}. Also we have,
    \begin{align}
\label{eq:STR_line_case_3_a_m_omega_bound}
        \left|\omega(\xi)a(\eta,\xi)\right| \lesssim 2^{k}\cdot2^{3k_1-\ell_1}\lesssim 2^{3k_1},\qquad\left|\omega(\xi)m(\eta,\xi)\right|\lesssim 2^{k}\cdot2^{3k_1-2\ell_1}\lesssim 2^{3k_1-\ell_1}. 
    \end{align}
An integration by parts, combined with \cite[Lemma~B.3]{Morgan}, \cite[Proposition~B.1]{Morgan}, and
\eqref{eq:STR_line_case_3_a_m_omega_bound}, yields
\begin{align*}
   \snorm{I_{m,k}}_{L^2_{\xi}} &\lesssim \eps_1^5 \int_{t_1}^{t_2}d\tau \ 2^{3k_1-\ell_1}2^{-\frac{4m}{3}} +  2^{3k_1}2^{-\frac{4}{3}m}
 \end{align*}
 Summing over frequencies and  $m$  yields
\begin{equation}
\sum_{m}\sum_{\text{freqs.}}  \snorm{I_{mk_1k_2k_3k_4k\ell_1\ell_2\ell_3}}_{L^2_{\xi}} \lesssim \eps_1^5 \left(t^{-3\delta_4} 
 +t^{1-3\delta_4} \right)
\end{equation}
which is the desired bound.

For $J_m$, we still only present the proof for $|\xi|\sim 2^k$. An integration by parts in $\partial_{\eta_1-\eta_2}$ using \eqref{eq:STR_line_case_3_a_m_omega_bound}, together with the multilinear estimate, yields
\begin{align*}
   \snorm{J_{...}}_{L^\infty_{\xi}}
   &\lesssim \eps_1^5 \left[2^{7k_1}2^{-\frac{4}{3}m} + 2^{7k_1-\ell_1}\cdot 2^{-\frac{4}{3}m}\right]
\end{align*}
Summing over frequencies and $m$ yields
\begin{equation}
    \sum_m \sum_{\text{freq.}}  \snorm{J_{...}}_{L^2_{\xi}} \lesssim \eps_1^5\left[2^{m(\frac{7}{9}-7\delta_4)}\cdot2^{-\frac{4}{3}m}+2^{m(-\frac{2}{9}-\delta_2-7\delta_4)}\right]\lesssim \eps_1^5,
\end{equation}
as desired.
\end{proof}

\subsubsection{space--time resonant curve}\label{STR_curves}
\noindent We now turn to the contribution from the space--time resonant curve
$\Gamma$ defined in \eqref{eqn:resonant_curve_discussion}. The most singular points on this curve, which is the sign-permutation configurations
$$
(\eta_1,\eta_2,\eta_3,\eta_4;\xi)
= (\sigma_1 \eta,\sigma_2 \eta,\sigma_3 \eta,\sigma_4 \eta;\ \xi=\sigma_\xi \eta),
\quad \sigma_j\in\{\pm1\},
$$ have already been discussed. The remaining portion of $\Gamma$ can therefore
be handled by analogous arguments. In accordance with our convention that $|\eta_1|$ is the dominant frequency, it suffices to control the contributions from the dyadic pieces 
\begin{multline*}
\Gamma_{\text{hi}, kk_1\ell_1\ell_2\ell_3} = \Big\{
    |\eta_1|\sim 2^{k_{1}}\in \left[2^5, 2^{k_{\operatorname{hi}}}\right], \ 
    |\eta_1 - \eta_2|\leq 2^{k_{\operatorname{lo}}} \ \text{or} \ \sim 2^{\ell_1}, \\
    |\eta_1 + \eta_3|\leq 2^{k_{\operatorname{lo}}} \ \text{or} \ \sim 2^{\ell_2}, \ 
    |r(\eta_1) + \eta_4|\leq 2^{k_{\operatorname{lo}}} \ \text{or} \ \sim 2^{\ell_3}, \ 
    |\xi + r(\eta_1)|\leq 2^{2k_{\operatorname{lo}}} \ \text{or} \ \sim 2^{k}
\Big\}
\end{multline*}
where $k_{\operatorname{lo}}, k_{\operatorname{hi}}$ are defined in \eqref{eq:def_k_lo} and \eqref{eq:def_k_hi}. Moreover, as in the analysis of the resonant line, the relative sizes of the
frequency parameters $k_j$ and $\ell_j$ can be determined by the Taylor expansion
of the phase around the reference point $\left(\overline{\eta},\overline{\eta}, -\overline{\eta}, -r\left(\overline{\eta}\right) ;-r\left(\overline{\eta}\right)\right)$. Hence, it yields
    $$
    \phi \approx \omega'\left(r(\overline{\eta})\right)\left(\xi+r(\overline{\eta})\right) + \mathcal{O}\left(\left(\eta_1-\overline{\eta}\right)^2, \left(\eta_2-\overline{\eta}\right)^2, \left(\eta_3+\overline{\eta}\right)^2,\left(\eta_4+r(\overline{\eta})\right)^2\right),
    $$
    In particular, this shows that the present localization is consistent with the one adopted in the previous subsection. Consequently, the appropriate strategy is to integrate by parts in time. We next state a proposition which isolates the source of the improved decay for a single factor.
\begin{lemma}[Improved $L^\infty_x$ decay]\label{prop:eta4_improved_decay}
Fix a dyadic time scale $t\in[2^m,2^{m+1}]$, and 
on the support of $\Gamma_{\mathrm{hi},kk_1\ell_1\ell_2\ell_3}$, we have $|\eta_4|\approx 1$.
Consequently, the corresponding input factor satisfies the improved dispersive decay
$$
\snorm{u_{k_4}(t)}_{L^\infty_x}\lesssim \eps_1 2^{-\frac{m}{2}}.
$$
\end{lemma}
\begin{proof}
    We have $|\eta_1|\gg 1$, $|\eta_1-\eta_2|\ll 1$ which implies $|\eta_2|\sim 2^{k_1}$. Similarly, $|\eta_1+\eta_3|\ll 1$ implies $|\eta_3|\gg 1$. For $|r(\eta_1)+\eta_4|$, we have $r(\eta_1)\approx 1$ since $|\eta_1|\gg 1$. This implies $|\eta_4|\approx 1$ which enjoys a better decay.
\end{proof}
\par The analogues of cases 1 and 2 from the  previous section (resonant line) can be handled in the same approach when we move over to the resonant curve. Also, in proposition \ref{prop:eta4_improved_decay} we have proved that we have the better decay $\snorm{u_{k_4}}_{L^{\infty}_{x}} \lesssim \eps_1 2^{-\frac{m}{2}}$ than the worst case case bound $\eps_1 \lesssim 2^{-\frac{m}{3}}$. Thus, it follows that we can close the cases by employing the exact same approach in the previous section, and the stronger decay is suffice to close the case. We next flatten the coordinates by introducing the change of variables
\begin{equation}\label{eq:STR_curve_flat_coordinate}
\Psi:(\eta_1,\eta_2,\eta_3,\eta_4)\mapsto(\mu,\eta_2,\eta_3,\eta_4),
\qquad \mu := r(\eta_1).
\end{equation}
In the present nonlinear setting, \cite[Lemma B.4]{Morgan} is no longer applicable. Instead, we perform a coordinate flattening via a change of variables and appeal to \cite[Remark B.5]{Morgan}. This introduces a Jacobian factor $|r'(\eta_1)|^{-1} \sim 2^{3k_1}$. As a consequence, the resulting bounds are identical to those obtained in the analysis of the resonant line.

We next prove the analogue of Proposition~\eqref{prop:STR_line_case_3}, which covers the only remaining case.
\begin{proposition}\label{prop:STR_curve_only_remaining_case}
    Let $\mathcal{C}$ be the frequency region defined by the following conditions:
\begin{itemize}
    \item The input frequencies satisfy
\begin{align*}
&|\eta_1-\eta_2|\sim 2^{\ell_1}, \qquad
|\eta_1+\eta_3|\sim 2^{\ell_2}, \qquad
|r(\eta_1)+\eta_4|\sim 2^{\ell_3}
\end{align*}
where
$$2^{k_{\mathrm{lo}}}\lesssim 2^{\ell_j}<2^{-10}\quad \forall j=1,2,3, 
\qquad \ell_3\ge \max\{\ell_1,\ell_2\}.$$
    \item The output frequency satisfies 
    $$|\xi+r(\eta_1)|\sim 2^{k}, \quad \ 2^{2k_{\operatorname{lo}}} \lesssim 2^{k}<2^{-10}, \quad 2^k  \lesssim 2^{\ell_3},$$
    or
    $$ |\xi+r(\eta_1)|\lesssim 2^{k_{\operatorname{lo}}}$$
\end{itemize}
Then, under the bootstrap assumptions, the contribution of $\mathcal{C}$ satisfies
    \begin{align*}
        \sum\snorm{I_{m,k_*}}_{L^2_\xi}\lesssim \eps_1^5\,\qquad\sum\snorm{J_{m,k_*}}_{L^\infty_\xi}\lesssim \eps_1^5.
    \end{align*}
\end{proposition}
\begin{proof}
In region $\mathcal{C}$, we have 
$$\left|\partial_{\eta_1}\phi\right|=|\omega'(\eta_1)-\omega'(\eta_4)|=|\omega'\left(r(\eta_1)\right)-\omega'(-\eta_4)|\gtrsim |r(\eta_1)+\eta_4| \sim 2^{\ell_3}.$$ Therefore, it implies that we can integrate by parts in $\partial_{\eta_1}$. One time integration by parts and flatting the coordinates by \eqref{eq:STR_curve_flat_coordinate}, it yields the symbol $\lesssim 2^{3k_1}2^{-\ell_3}$. Therefore, the case reduces to apply the arguments of case 3 in the resonant line case, and by noticing that we have an improved decay in the $\eta_4$ term instead of the worst case case decay $\eps_12^{-\frac{m}{3}}$.
\end{proof}
\subsection{Part V: Weighted $L^2_x$ and $L^\infty_\xi$ bounds in non-resonant cases}
In the case of the quartic gBBM, Morgan \cite{Morgan} treated these contributions via standard integration by parts in time or space (frequency), utilizing the non-stationary phase to gain integrability in time.
For the quintic problem considered here, the analysis of these non-resonant terms applies \textit{mutatis mutandis}. Moreover, the quintic nonlinearity provides an additional factor of dispersive decay compared to the quartic case. Specifically, the multilinear integral involves five inputs rather than four, yielding an extra factor  in the physical side with the worst decay
$$\snorm{u}_{L^\infty} \lesssim \eps_1 t^{-\frac{1}{3}}.$$ Consequently, the closure of the bootstrap estimates in these non-resonant  regions follow immediately from the arguments established in \cite{Morgan}, with improved convergence rates. We then omit the repetitive details for brevity.

\section{Scattering}
\begin{proposition}\label{prop:scattering}
    There exists a unique $F(x)\in H^s_x$ with $\widehat{F}\in  H^1_\xi$ such that the profile $f(t,x)$ of $u(t,x)$ satisfies
    $$\lim_{t\to \infty}\snorm{f(t,x)-F(x)}_{H^s_x}=\lim_{t\to \infty}\snorm{\widehat{f}(t,\xi)-\widehat{F}(\xi)}_{H^1_\xi}=0.$$
    In particular, $u(t,x)$ scatters in $H^s_x$. 
\end{proposition}
\begin{proof}
We first consider the convergence in $H^1_\xi$.   
Fix any $1<t_1<t_2$. Let $m_1$ be the largest integer such that $2^{m_1}\le t_1$, and let $m_2$ be the smallest integer such that $2^{m_2}\ge t_2$. 
Then
$$
[t_1,t_2]\subset \bigcup_{m=m_1}^{m_2}\big([2^m,2^{m+1}]\cap [t_1,t_2]\big).
$$
We then write out the Duhamel's formula in frequency space, taking difference we have
$$
\widehat f(t_2,\xi)-\widehat f(t_1,\xi)=\sum_{m=m_1}^{m_2} J_m(\xi),
$$
where $J_m$ denotes the contribution of the time integration restricted to $[2^m,2^{m+1}]\cap [t_1,t_2]$. 
From the previous section, in particular Lemma \ref{lem:dyadicLJ}, we have shown that there exist constants $\alpha>0$ and $C>0$ such that
$$
\|J_m\|_{H^1_\xi}\lesssim\,\eps_1^5\,2^{-\alpha m}
\quad\text{for all } m\ge m_1.
$$
Therefore,
$$
\|\widehat f(t_2)-\widehat f(t_1)\|_{H^1_\xi}
\le \sum_{m=m_1}^{m_2}\|J_m\|_{H^1_\xi}
\lesssim \eps_1^5\sum_{m=m_1}^{\infty}2^{-\alpha m}
\lesssim \eps_1^5\,2^{-\alpha m_1}
\lesssim \eps_1^5\,t_1^{-\alpha}.
$$
Letting $t_1\to\infty$ shows that $\widehat f(t)$ is Cauchy in $H^1_\xi$. 
Since $H^1_\xi$ is complete, there exists a unique $\widehat F\in H^1_\xi$ such that
$$
\widehat f(t,\xi)\to \widehat F(\xi)\quad\text{in }H^1_\xi \text{ as } t\to\infty.
$$
To see the convergence in  $H^s_x$, we simply apply Minkowski's inequality, Plancherel, and the pointiwse decay \eqref{eq:udecay} to the Duhamel formula \eqref{eq:f_duhamel} by a similar argument as in Proposition \ref{prop:bootstrap}.
\end{proof}


\end{document}